\documentclass[reqno]{amsart}

\usepackage[all]{xy}
\usepackage{graphicx}

\usepackage{amssymb}
\usepackage{amsmath}
\usepackage{mathrsfs}
\usepackage{epsfig}
\usepackage{amscd}
\usepackage{graphicx, color}

\usepackage{graphicx}
\usepackage{amssymb}
\usepackage{amsmath}
\usepackage{enumerate}
\usepackage{amsfonts}
\usepackage[mathcal]{euscript}
\usepackage{amsthm}
\usepackage{xcolor}
\usepackage{amscd}
\usepackage{mathrsfs}
\numberwithin{equation}{section}

\usepackage[colorlinks=true, pdfstartview=FitV, linkcolor=black,
			   citecolor=black, urlcolor=black]{hyperref}
\usepackage{bookmark}
\usepackage[titletoc]{appendix} 
\hypersetup{hypertexnames=false} 

          \theoremstyle{definition}
          \newtheorem{theorem}{Theorem}[section]
          
          \newtheorem{prop}[theorem]{Proposition}
          \newtheorem{lemma}[theorem]{Lemma}

          \newtheorem{cor}[theorem]{Corollary}

          \newtheorem{defn}[theorem]{Definition}
          \newtheorem{notation}[theorem]{Notation}
          \newtheorem{example}[theorem]{Example}
          \newtheorem{choice}[theorem]{Choice}

          \newtheorem{remark}[theorem]{Remark}

\numberwithin{equation}{section}

          \newcommand{\nc}{\newcommand}
          \nc{\DMO}{\DeclareMathOperator}	
          \nc{\commentout}[1]{}

          \nc{\newnotation}{\nomenclature}
          \nc{\wrap}{\cW}
          \nc{\Tw}{\mathsf{Tw}}
          \nc{\loc}{\mathsf{Loc}}
          \nc{\Top}{Top}
          \nc{\emb}{\mathsf{emb}}
          \nc{\ind}{\mathsf{Ind}}
          \nc{\Ind}{\mathsf{Ind}}
          \nc{\Loc}{\mathsf{Loc}}
          \nc{\Cob}{\mathsf{Cob}}
          \nc{\mul}{\mathsf{Mul}}
          \nc{\fat}{\mathsf{fat}}
          \nc{\cob}{\mathsf{Cob}}
          \nc{\coh}{\mathsf{Coh}}
          \nc{\Liouaut}{\Aut_{\mathsf{Liou}}}
          \nc{\idem}{\mathsf{Idem}}
          \nc{\sets}{\mathsf{Sets}}
          \nc{\near}{\mathsf{near}}
          \nc{\sing}{\mathsf{Sing}}
          \nc{\Sing}{\mathsf{Sing}}
          \nc{\perf}{\mathsf{Perf}}
          \nc{\block}{\mathsf{block}}
          \nc{\ssets}{\mathsf{sSets}}
          \nc{\cmpct}{\mathsf{cmpct}}
          \nc{\compact}{\mathsf{cmpct}}
          \nc{\pwrap}{\mathsf{PWrap}}
          \nc{\coder}{\mathsf{Coder}}
          \nc{\bimod}{\mathsf{Bimod}}
          \nc{\grmod}{\mathsf{GrMod}}
          \nc{\Morita}{\mathsf{Morita}}
          \nc{\morita}{\mathsf{Morita}}
          \nc{\spaces}{\mathsf{Spaces}}
          \nc{\pwrms}{\mathsf{PWrFuk}_{M,S}}
          \nc{\pwrmf}{\mathsf{PWrFuk}_{M,F}}
          \nc{\pwrapmf}{\mathsf{PWrFuk}_{M,F}}
          \nc{\fuk}{\mathsf{Fukaya}}
          \nc{\infwr}{\mathsf{InfWr}}
          \nc{\fukaya}{\mathsf{Fukaya}}
          \nc{\autml}{\mathsf{Aut}_{M,\Lambda}}
          \nc{\fukml}{\mathsf{Fukaya}_{M,\Lambda}}
          \nc{\fukmle}{\mathsf{Fukaya}_{M,\Lambda,\epsilon}}
          \nc{\fukmod}{\wrfukcompact(M)\modules}
          \nc{\lag}{\mathsf{Lag}}
          \nc{\lagm}{\lag_M}
          \nc{\lago}{\lag^o}
          \nc{\lagml}{\lag_{M,\Lambda}} 
          \nc{\lagmle}{\lag_{M,\Lambda,\epsilon}}
          \nc{\Fun}{\mathsf{Fun}}
          \nc{\fun}{\mathsf{Fun}}
          \nc{\vect}{\mathsf{Vect}}
          \nc{\chain}{\mathsf{Chain}}
          \nc{\chainn}{Chain}
          \nc{\wrfuk}{\mathsf{WrFukaya}}
          \nc{\wrfukcompact}{\mathsf{WrFukaya}_{\mathsf{cmpct}}}
          \nc{\pwrfuk}{\mathsf{PWrFukaya}}
          \nc{\inffuk}{\mathsf{InfFuk}}
          \nc{\pwrfukml}{\mathsf{PWrFukaya}_{M,\Lambda}}
          \nc{\inffukml}{\mathsf{InfFuk}_{M,\Lambda}}
          \nc{\nattrans}{\mathsf{NatTrans}}
          \nc{\corres}{\mathsf{Corres}}
          \nc{\fukep}{\fukaya_\Lambda(M,\epsilon)}
          \nc{\fukepop}{\fukaya_\Lambda(M,\epsilon)^{\op}}
          \nc{\lagep}{\lag_\Lambda(M,\epsilon)}
          \DMO{\cyl}{cyl} 
          \nc{\dbcoh}{D^b\mathsf{Coh}}
          \nc{\corr}{\mathsf{Corr}}
          \nc{\Liouauto}{{\Aut^o}}
          \nc{\Liouautb}{\Aut^{b}}
          \nc{\Liouautgr}{{\Aut^{gr}}}
          \nc{\Liouautgrb}{\Aut^{gr,b}}
          \nc{\Fuk}{\mathsf{Fuk}}

          \DMO{\im}{im}
          \DMO{\ev}{ev}
          \DMO{\stable}{Ex}
          \DMO{\inj}{inj}
          \DMO{\fib}{fib}
          \DMO{\conf}{Conf}
          \DMO{\chains}{Chains}
          \DMO{\cochains}{Cochains}
          \DMO{\cone}{Cone}
          \DMO{\Map}{Map}
          \DMO{\ran}{Ran}
          \DMO{\rot}{Rot}
          \DMO{\leg}{Leg}
          \DMO{\imm}{imm}
          \DMO{\adj}{adj}
          \DMO{\symp}{Symp}
          \DMO{\tree}{Tree}
          \DMO{\cube}{Cube}
          \DMO{\deep}{deep}
          \DMO{\back}{back}
          \DMO{\Hoch}{Hoch}
          \DMO{\front}{front}
          \DMO{\flow}{Flow}
          \DMO{\floer}{Floer}
          \DMO{\Maps}{Maps}
          \DMO{\exact}{exact}
          \DMO{\excess}{Excess}
          \DMO{\Decomp}{Decomp}
          \DMO{\decomp}{Decomp}
          \DMO{\collar}{collar}
          \DMO{\yoneda}{Yoneda}
          \DMO{\hamspace}{Ham}
          \DMO{\sympspace}{Symp}
          \DMO{\holomaps}{Holomaps}
          \DMO{\comp}{Comp}
          \DMO{\crit}{Crit}
          \DMO{\test}{{test}}
          \DMO{\sign}{sign}
          \DMO{\topp}{top}
          \DMO{\indx}{Index}
          \DMO{\Break}{Break} 
          \DMO{\zero}{zero} 
          \DMO{\ob}{Ob}
          \DMO{\gr}{Gr} 
          \DMO{\Gr}{Gr} 
          \DMO{\cl}{Cl} 
          \DMO{\grlag}{GrLag}
          \DMO{\Pin}{Pin}
          \DMO{\Graph}{Graph}
          \DMO{\pin}{Pin}
          \DMO{\gap}{Gap}
          \DMO{\Ex}{Ex}
          \DMO{\id}{id}
          \DMO{\End}{End}
          \DMO{\sym}{Sym}
          \DMO{\aut}{Aut}
          \DMO{\Aut}{Aut}
          \DMO{\haut}{hAut}
          \DMO{\hAut}{hAut}
          \DMO{\DK}{DK} 
          \DMO{\poly}{poly} 
          \DMO{\diff}{Diff}
          \DMO{\coll}{coll}
          \DMO{\dist}{dist} 
          \DMO{\coker}{coker} 
          \nc{\kernel}{\ker} 
          \DMO{\sspan}{span}
          \DMO{\hocolim}{hocolim}	
          \DMO{\holim}{holim}
          \DMO{\sk}{sk}

          \DMO{\ho}{ho}
          \DMO{\fin}{fin}
          \DMO{\tor}{Tor}
          \DMO{\ext}{Ext}
          \DMO{\ret}{Ret}
          \DMO{\ham}{Ham}
          \DMO{\con}{con}
          \DMO{\leaf}{leaf}
          \DMO{\supp}{supp}
          \DMO{\edge}{edge}
          \DMO{\colim}{colim}
          \DMO{\edges}{edges}
          \DMO{\Image}{image}
          \DMO{\roots}{roots}
          \DMO{\height}{height}
          \DMO{\finmod}{FinMod}
          \DMO{\leaves}{leaves}
          \DMO{\planar}{planar}
          \DMO{\vertices}{vertices}

\nc{\norm}[2]{{ \ensuremath{\|} #1 \ensuremath{\|}}_{#2}}
\nc{\Dbar}[1]{\ensuremath{{\bar{\partial}}_{#1}}}
\nc{\Ce}{\ensuremath{\mathbb{C}}}
\nc{\B}{\ensuremath{\mathbb{B}}}
\nc{\osc}{\operatorname{osc}}
\nc{\leng}{\operatorname{leng}}

          \nc{\lagg}{\lag^{\cG}}
          \nc{\iso}{\mathsf{Iso}}
          \nc{\Set}{\mathsf{Set}}
          \nc{\Ass}{\mathsf{ \bf Ass}}
          \nc{\Mod}{\mathsf{Mod}}
          \nc{\modules}{\mathsf{Mod}}
          \nc{\sset}{\mathsf{sSet}}
          \nc{\liou}{\mathsf{Liou}}
          \nc{\poset}{\mathsf{Poset}}
          \nc{\trno}{T^*\RR^n_{\geq 0}}
          \nc{\spectra}{\mathsf{Spectra}}
          \nc{\tensorfin}{\tensor^{\fin}}
          \nc{\lagptg}{\lag_{pt,pt}^{\cG}}
          \nc{\Fin}{\mathcal{F}\mathsf{in}}
          \nc{\lagnl}{\lag_{N,\Lambda}}
          \nc{\lagmlg}{\lag_{M,\Lambda}^{\cG}}
          \nc{\lagsplit}{\lag^{\mathsf{split}}}
          \nc{\lagktimes}{(\lag^{\dd k})^\times}
          \nc{\lagplanar}{\lag^{\times,\planar}}
          \nc{\Cont}{\text{\rm Cont}}
          \nc{\Ham}{\text{\rm Ham}}
          \nc{\Dev}{\text{\rm Dev}}
          \nc{\Lin}{\text{\rm Lin}}
          \nc{\Int}{\text{\rm Int}}
          \nc{\Hom}{\text{\rm Hom}}
          \nc{\Chord}{\text{\rm Chord}}
          \nc{\nbhd}{\mathcal{N}\text{\rm{bhd}}}

          \nc{\onef}{1_{\fukaya}}
          \nc{\smsh}{\wedge}
          \nc{\un}{\underline}
          \nc{\xto}{\xrightarrow}
          \nc{\xra}{\xto}
          \nc{\tensor}{\otimes}
          \nc{\del}{\partial}
          \nc{\dd}{\diamond}
          \nc{\tri}{\triangle}
          \nc{\bb}{\Box}
          \nc{\into}{\hookrightarrow}
          \nc{\onto}{\twoheadrightarrow}
          \nc{\contains}{\supset}
          \nc{\transverse}{\pitchfork}
          \nc{\uncirc}{\underline{\circ}}
          \nc{\thetacontact}{\theta} 

          \nc{\Jbar}{\overline{J}}
          \nc{\Fbar}{\overline{F}}
          \nc{\delbar}{\overline{\del}}
          \nc{\thetabar}{\overline{\theta}}
          \nc{\omegabar}{\overline{\omega}}
          \nc{\Liou}{\text{\rm Liou}}

          \nc{\Yhat}{\widehat{Y}}

          \nc{\Mliou}{M}

          \nc{\vece}{ {\vec \epsilon}}	
          \nc{\vecd}{ {\vec \delta}}
          \nc{\ov}{\overline}
          \DMO{\op}{op}
          \nc{\opp}{ ^{\op}}

          \nc{\hiro}{\textcolor{blue}}
          \nc{\YG}{\textcolor{orange}}

\numberwithin{equation}{section}

\def\R{{\mathbb R}}
\def\osc{{\hbox{\rm osc }}}

\def\C{{\mathbb C}}
\def\R{{\mathbb R}}

\def\11{{\mathbb I}}

\def\Jbar{{\widetilde J}}

\def\delbar{{\overline \partial}}

          \def\cD{\mathcal D}
          \def\cF{\mathcal F}\def\cG{\mathcal G}
          \def\cL{\mathcal L}
          \def\cN{\mathcal N}\def\cO{\mathcal O}\def\cP{\mathcal P}
          
          \def\cU{\mathcal U}\def\cW{\mathcal W}\def\c\Mliou{\mathcal \Mliou}

          \def\FF{\mathbb F}\def\HH{\mathbb H}
          
          \def\NN{\mathbb N}
          \def\RR{\mathbb R}
          \def\\Mliou\Mliou{\mathbb \Mliou}

          \def\s\Mliou{\mathsf \Mliou}

          \def\b\Mliou{\mathbf \Mliou}

          \def\f\Mliou{\mathfrak \Mliou}

\def\Jbar{{\widetilde J}}

\def\delbar{{\overline \partial}}

\def\del{\partial}

\def\b{\beta}
\def\c{\chi}

\def\f{\phi}

\def\s{\sigma}

\def\CL{{\mathcal L}}

\def\CN{{\mathcal N}}

\def\darr#1{\raise1.5ex\hbox{$\leftrightarrow$}
\mkern-16.5mu #1}

\def\roughly#1{\raise.3ex\hbox{$#1$\kern-.75em
\lower1ex\hbox{$\sim$}}}

\def\opname#1{\mathop{\kern0pt{\rm #1}}\nolimits}
\def\Re{\opname{Re}}

\def\End{\opname{End}}
\def\dim{\opname{dim}}

\def\dist{\opname{dist}}

\def\supp{\operatorname{supp}}

\def\Dev{\operatorname{Dev}}

\def\leng{\operatorname{leng}}

\def\End{\operatorname{End}}

\def\Aut{\operatorname{Aut}}
\def\coker{\operatorname{Coker}}

\def\Cont{\operatorname{Cont}}

\def\Sing{\operatorname{Sing}}

\def\Diff{\operatorname{Diff}}
\def\Image{\operatorname{Image}}
\def\ev{\operatorname{ev}}
\def\Int{\operatorname{Int}}

\def\ben{\begin{enumerate}}
\def\een{\end{enumerate}}
\def\be{\begin{equation}}
\def\ee{\end{equation}}
\def\bea{\begin{eqnarray}}
\def\eea{\end{eqnarray}}
\def\beastar{\begin{eqnarray*}}
\def\eeastar{\end{eqnarray*}}
\def\bc{\begin{center}}
\def\ec{\end{center}}

\renewcommand{\b}{\beta}

\def\Hoch{{\tt Hoch}}

\def\Cont{\operatorname{Cont}}

\def\Sing{\operatorname{Sing}}

\def\Ham{\operatorname{Ham}}

\def\Graph{\operatorname{Graph}}
\def\id{\text{\rm id}}

\begin{document}

%
%
\title[Liouville sectors with corners]{Presymplectic characterization of 
Liouville sectors with corners, and its monoidality}

\author{Yong-Geun Oh}
\address{Center for Geometry and Physics, Institute for Basic Science (IBS),
77 Cheongam-ro, Nam-gu, Pohang-si, Gyeongsangbuk-do, Korea 790-784
\& POSTECH, Gyeongsangbuk-do, Korea}
\email{yongoh1@postech.ac.kr}
\thanks{This work is supported by the IBS project \# IBS-R003-D1}

\begin{abstract}
We provide a presymplectic characterization of Liouville sectors introduced
by Ganatra-Pardon-Shende in \cite{gps,gps2}
in terms of the characteristic foliation of the boundary,
which we call \emph{Liouville $\sigma$-sectors}.
We extend this definition to the case with corners
using the presymplectic geometry of null foliations of the coisotropic
intersections of \emph{transverse coisotropic collection} of hypersurfaces which
appear in the definition of Liouville sectors with corners. We show that \emph{the set of
Liouville $\sigma$-sectors with corners canonically forms a monoid}
which provides a natural framework of considering the K\"unneth-type functors
in the wrapped Fukaya category.
We identify its automorphism group which enables one to give a natural definition of
\emph{bundles of Liouville sectors}. As a byproduct, we
affirmatively answer to a question raised  in \cite[Question 2.6]{gps},
which asks about the optimality of their definition of Liouville sectors in \cite{gps}.
\end{abstract}

\maketitle

\tableofcontents

\section{Introduction}

Ganatra-Pardon-Shende introduced a flexible framework of \emph{Liouville sectors} (with corners)
and established the local-to-global principle of wrapped Fukaya categories in a series of papers
\cite{gps,gps2,gps-microlocal}.

In \cite{oh-tanaka:liouville-bundles},
Tanaka and the present author constructed an unwrapped Floer theory
for \emph{bundles} of Liouville manifolds and Liouville sectors. The output was a collection
of unwrapped Fukaya categories associated to fibers of a Liouville bundle of Liouvile sectors,
along with  a compatibility between two natural constructions of continuation maps.
This set-up enabled them to make the construction of Floer-theoretic invariants
of smooth group actions on Liouville manifolds, and they exploited these constructions
in \cite{oh-tanaka:actions,oh-tanaka:localizations} to construct homotopically coherent actions of Lie groups on wrapped Fukaya categories, thereby proving a version of a conjecture from Teleman's 2014 ICM address.

\subsection{Presymplectic characterization of Liouville sectors}

The original definition of Liouville sectors given in \cite{gps,gps2} makes it somewhat clumsy
to identify the structure group of a bundle of Liouville sectors with corners, and to define \emph{the bundle} 
of Liouville sectors with \emph{corners} as in
\cite{oh-tanaka:liouville-bundles}: This is partly because not every defining condition 
related to the \emph{presymplectic geometry of the boundary} is manifestly invariant under the 
action of \emph{Liouville diffeomorphisms}. This was the starting point of current investigation.
In this paper, we introduce a more intrinsic 
but equivalent definition of Liouville sector which skirts this issue: We say it is more
intrinsic in that our definition is closer to one in the sense of $G$-structures.
(See \cite{chern} or \cite[Chapter VII]{sternberg} for a general introduction to $G$-structures.)

\begin{remark}
It may be worthwhile to mention that in the original definition of 
Liouville sectors from \cite{gps,gps2} consideration of the product of Liouville sectors 
is somewhat clumsy and nontrivial which affects the discussion of K\"unneth-type functors. 
However it follows from our definition of $\sigma$-sectors that the product of two Liouville 
$\sigma$-sector canonically becomes a \emph{Liouville $\sigma$-sectors with corners}. 
(See Proposition \ref{prop:product} for the proof.) 
We refer readers to the discussion around \cite[Lemma 2.21]{gps} and  \cite[Section 6]{gps2},
and to \cite{oh:gradient-sectorial} for some relevant discussion on the construction of 
monoidal property of wrapped Fukaya category.
\end{remark}

We start with our discussion of $M$ for the case without corners.

Let $(M, \omega)$ be a symplectic manifold with boundary, which we assume \emph{tame}
in the standard sense in symplectic geometry, say, from \cite{sikorav:tame}.
The boundary $\del M$ (or more generally any coisotropic submanifold $H$) then
carries a natural structure of a \emph{presymplectic manifold} in the sense that
the restriction two form
$$
\omega_\del: = \iota^*\omega
$$
has constant nullity. (See \cite{gotay}, \cite{oh-park} for some detailed explanation on presymplectic
manifolds.) Here $\iota: \del M \to M$ is the inclusion map.

\begin{notation}[$\cD_{\del M}$, $\cN_{\del M}$ and $\pi: \del M \to \cN_{\del M}$]
We denote the characteristic distribution of $(\del M,\omega_\del)$ by
$$
\cD_{\del M} = \ker \omega_\del.
$$
With a slight abuse of notation, we also denote by $\cD_{\del M}$ the associated integrable foliation,
and let $\pi_{\del M}: \del M \to \cN_{\del M}$ be its leaf map.
\end{notation}

Now consider a Liouville manifold $(M,\lambda)$ with boundary and denote by
$$
(\del_\infty M, \xi_\infty)
$$
its ideal boundary as a contact manifold equipped with the contact \emph{distribution} $\xi_\infty$
canonically induced by the Liouville form $\lambda$. (See \cite{giroux}. We recall that
there is no contact \emph{form} on $\del_\infty M$ canonically induced from $\lambda$.)

We will assume \emph{$M$ is $C^3$-tame} in the sense of \cite{choi-oh:quasiisometry} 
which will be needed for the study of the question \cite[Question 2.6]{gps}
 in Section \ref{sec:GPS-question}.
Throughout this paper, by ``near infinity,'' we mean ``on the complement on some compact subset 
of $M$.''

\begin{defn}[Liouville $\sigma$-sectors]\label{defn:liouville-sector-intro}
We say a Liouville manifold with boundary $(M,\lambda)$ is a {\em Liouville $\sigma$-sector} if the following holds:
\begin{enumerate}[(a)]
\item\label{item. contact boundary} The Liouville vector field $Z$ of $M$ is tangent to $\del M$ near infinity.
\item\label{item. convex corner} $\del_\infty M \cap \del M$ is the boundary of $\del_\infty M$, and is convex (as a hypersurface of the contact manifold $\del_\infty M$).
\item\label{item. foliation section} The canonical projection map $\pi:\del M \to \cN_{\del M}$ (to the leaf space of
the characteristic foliation) admits a continuous section, and has fibers abstractly homeomorphic to $\RR$.
\end{enumerate}
\end{defn}
The condition (\ref{item. foliation section})
in this definition is the difference from that of the \emph{Liouville sector} of \cite{gps} and
is responsible for our naming of \emph{Liouville $\sigma$-sectors} where $\sigma$ stands
either for `section' or for `sectional'. It can be replaced by the contractibility of fibers.
(See Corollary \ref{cor:contractible-fiber}.) We will also show in Corollary \ref{cor:trivial-CD} that
the line bundle $\cD_{\del M}$  appearing in this definition is trivial.

\begin{remark}\label{rem:presymplectic}
\begin{enumerate}
\item
In the point of $G$-structures, the choice of a section corresponds to a reduction of
the structure group from $\Diff(\R)$ to $\Diff(\R,\{0\})$ of the $\R$-bundle associated to the null foliation.
\item It is worthwhile to mention that
the presymplectic structure on $(\del M, \omega_\del)$ uniquely determines a symplectic structure on the germ of a neighborhood
 up to symplectic diffeomorphism. (See  \cite{gotay}.) Our definition of Liouville $\sigma$-sectors \emph{with corners}
is much based on Gotay's coisotropic embedding theorem of presymplectic manifolds
\cite{gotay}, applied to a germ of neighborhoods of the boundary $\del M$ or more generally of coisotropic
submanifolds of $(M,d\lambda)$.
\item The condition \eqref{item. foliation section}
depends only on the presymplectic geometry of $(\del M, d\lambda_\del)$ with $\lambda_\del = i_{\del M}^*\lambda$
while the conditions \eqref{item. contact boundary} and \eqref{item. convex corner}
depend on the Liouville geometry at infinity of the ideal contact boundary $\del_\infty M$.
The two geometries are connected by the global topological triviality
of the characteristic foliation implied by \eqref{item. foliation section}.
(See Theorem \ref{thm:GPS-question-intro}.)
\end{enumerate}
\end{remark}

Note that a Liouville ($\sigma$-)sector $\Mliou$ is a smooth manifold (possibly with non-compact corners) and the Liouville
flow determines a well-defined contact manifold $\del_\infty \Mliou$ ``near infinity'' (possibly with boundary).  We will informally write
	\be
	\del_\infty \Mliou \cap \del \Mliou = \del (\del_\infty M)
	\ee
to mean the boundary of $\del_\infty \Mliou$ and call it the \emph{ceiling corner} of the Liouville sector.
(When $\del_\infty M$ has corners, ``boundary'' means the union of all boundary strata.)

\begin{theorem}[Theorem \ref{thm:equivalence-H} for $H = \del M$] \label{thm:equivalence-intro}
Under the above definition of Liouville $\sigma$-sector,  the following holds:
\begin{enumerate}[(1)]
\item\label{item.equivalence cL manifold-intro} $\cN_{\del \Mliou}$ carries
the structure of Hausdorff smooth manifold  such that $\pi: \del M \to \cN_{\del \Mliou}$
is a smooth submersion.
\item The given continuous section $\sigma$ of $\pi: \del M \to \cN_{\del M}$
can be $C^0$-approximated by a smooth section $\sigma^{\text{\rm sm}}$
as close as we want.
\item\label{item. equivalence symplectic structure-intro} 
$\cN_{\del \Mliou}$ carries a canonical symplectic structure denoted by $\omega_{\cN_{\del M}}$ 
as a coisotropic reduction of $\del \Mliou \subset \Mliou$: We set $F: = \text{\rm Image }\sigma^{\text{\rm sm}}$. 
Then there is a diffeomorphism
$\Psi: \del \Mliou \to F \times \RR$ and a commutative diagram
\be \label{eq:diagram-intro}
\xymatrix{\del \Mliou \ar[d]^\pi \ar[r]^{\Psi} & F \times \RR \ar [d]^{\pi_F}\\
\cN_{\del \Mliou} \ar[r]^{\psi} & F
}
\ee
with  $\pi_F$ the canonical projection such that the aforementioned smooth section 
$\sigma^{\text{\rm sm}}$ satisfies 
$$
(\sigma^{\text{\rm sm}})^*\omega_\del = \omega_{\cN_{\del M}}.
$$ 
\item $(\cN_{\del M},\omega_{\cN_{\del M}})$ carries a canonical Liouville one-form $\lambda_{\cN_{\del M}}$:
The map $\psi$ is a Liouville diffeomorphism between
$(\cN_{\del M},\lambda_{\cN_{\del M}})$ and $(F, \lambda|_F)$ with the Liouville form
$ \lambda|_F$ on $F$, which is given by $\psi(\ell) = \sigma(\ell)$
for $\ell \in \cN_{\del M}$.
\end{enumerate}
\end{theorem}
The existence result of a \emph{smooth} section $\sigma^{\text{\rm sm}}$ is a kind of
a \emph{smoothing result} of the given continuous section $\sigma: \cN_H \to H$.
In the literature, we could not locate such a smoothing result of a section of the leaf space 
projection of the foliation, and so provide its full proof in Subsection \ref{subsec:smoothing} for our 
current circumstance. We refer to Section \ref{sec:intrinsic} for the precise description on the dependence
of various structures and maps on the choice of section $\sigma$.

\begin{remark} Other than the existence of the contact vector field transverse to the contact distribution,
which is the defining property of the convexity of hypersurfaces,
the \emph{contact geometry of ideal boundary $\del_\infty M$} does not enter in the proof of this theorem:
It is mainly about the \emph{presymplectic geometry of coisotropic submanifold $\del M$}, which
makes our affirmative answer to the question \cite[Question 2.6]{gps} plausible.
See Remarks \ref{rem:dividing-set}, \ref{rem:GPS-question} below for a further elaboration.
\end{remark}

The following can be also derived in the course of proving the above theorem.
(In fact the argument deriving this proposition is nearly identical to
that of the proof of \cite[Lemma 2.5]{gps}.)

\begin{prop}\label{prop:equivalence-intro}
Let $(M,\lambda)$ be a Liouville $\sigma$-sector. Then
\begin{enumerate}
\item Each choice of smooth section $\sigma$ of $\pi$ and a constant $0 < \alpha \leq 1$ canonically provides a smooth function
$I: \del M \to \RR$ such that $Z(I) = \alpha I$,
\item There is a germ of neighborhood $\nbhd(\del M)$ (unique up to a symplectomorphism fixing
$\del M$) on which the natural extension of $I$, still denoted by $I$, admits a function 
$R: \nbhd(\del M) \to \RR$ satisfying $\{R,I\} = 1$ and vanishing along $\del M$.
\end{enumerate}
\end{prop}
%
%

\subsection{Interpolation of presymplectic  and Liouville geometry at infinity}

Another interesting consequence is the following affirmative answer to a question raised by
Ganatra-Pardon-Shende.

\begin{theorem}[Theorem \ref{thm:GPS-question}; Question 2.6 \cite{gps}]
\label{thm:GPS-question-intro} Suppose $M$ is a Liouville manifold-with-boundary such that
\begin{enumerate}
\item the Liouville vector field is tangent to $\del M$ near infinity, and
\item there is a diffeomorphism $\del M = F \times \RR$ sending the characteristic foliation to the foliation by
leaves $\RR \times \{p\}$.
\end{enumerate}
Then $\del_\infty M \cap \del M$ is convex in $\del_\infty M$. In particular $M$ is a Liouville sector
in the sense of \cite{gps}.
\end{theorem}
We mention that $F$ itself naturally becomes a Liouville manifold. (See Section \ref{subsec:liouville-structure-leaf-space}
for the proof.) 

The main task then is to construct a contact vector field transverse to the
ceiling corner
$$
\del_\infty M \cap \del M =: F_\infty
$$
in the contact manifold $\del_\infty M$. 
We make our construction of the aforementioned contact vector field as a consequence of the following
refinement of Gotay's neighborhood normal form theorem for $\del M \subset M$ when
$\del M \cong F\times \R$ with the given hypotheses. 

\begin{prop}\label{prop:normal-form-intro}
Let $u + \sqrt{-1} v$ be the standard coordinates of $\C$ satisfying $v = t\circ \text{\rm pr}$.
Put
$$
R = u \circ \pi_\C \circ \widetilde \Psi, \quad I = v \circ \pi_\C\circ \widetilde \Psi
$$
on $F \times \C$. We denote by $\Psi: \del M \to F \times \{0\} \times \R$ a diffeomorphism given by
the hypothesis in the theorem.

Then there are neighborhoods $U$ of $\del M \cong F \times \R$ and $V = F \times (-\delta, 0] \times \R$ of 
$F \times \{0\} \times \R \subset F \times \C$
for some $\delta > 0$, and a deformation of $\Psi$, still denoted by $\Psi$,
which extends to a diffeomorphism pair 
$$
(\widetilde \Psi,\Psi): (U,\del M) \to (V, F \times \{0\} \times \R)
$$ 
satisfying
\be\label{eq:omegaV1}
\widetilde \Psi^*\lambda = \widetilde \pi_F^*\lambda_F  -I\,  dR, \quad 
\widetilde \Psi_*(Z) = Z_F \oplus I \frac{\del}{\del I}
\ee
on $\{I > C\} \cap V'$ for a sufficiently large $C> 0$ where $Z_F$ is the Liouville vector field of
the Liouville manifold $F$.
In particular we have $F\cong \del M \cap \del_\infty M = \del(\del_\infty M)$, which is convex in $\del_\infty M$.
\end{prop}
An important ingredient of the proof 
is some stability theorem proved in Appendix \ref{sec:stability-sectors} of Liouville sectors which 
extends the one proved in \cite[Theorem 9.2]{oh:sectorial}.

The following equivalence theorem is an immediate corollary of Theorem \ref{thm:GPS-question-intro}.

\begin{theorem}\label{thm:GPS-equivalence-intro} Let $(M,\lambda)$ be a Liouville manifold
with boundary.  Suppose the Liouville vector field $Z$ of $\lambda$ is tangent to $\del M$
near infinity. Then the followings are equivalent:
\begin{enumerate}
\item $(M,\lambda)$ is a Liouville sector in the sense of \cite{gps}.
\item $(M,\lambda)$ is a Liouville $\sigma$-sector.
\item There is a diffeomorphism $\del M = F \times \RR$ sending the characteristic foliation to the foliation by
leaves $\RR \times \{p\}$.
\end{enumerate}
\end{theorem}

\subsection{Transverse coisotropic collections and Liouville $\sigma$-sectors with corners}

The definition of Liouville $\sigma$-sector can be extended to the case with corners.
Here we start with giving another equivalent definition
of that of the \emph{sectorial hypersurface} from \cite[Definitions 9.2 \& 9.14]{gps2}.
Our  definition is intrinsic in that it utilizes only the canonical
presymplectic geometry of null foliation of the hypersurface in the symplectic manifold
$(M,\omega)$, which is \emph{coisotropic}. Now the existence of
the defining data of function $I$ or of the diffeomorphism $\del M \to F \times \R$
appearing in the definition of Liouville sectors in \cite{gps}
is a `property' of Liouville $\sigma$-sector in our definition.

We start with giving the aforementioned equivalent definitions.

\begin{defn}[$\sigma$-sectorial hypersurface]\label{defn:sectorial hypersurface}
Let $(M, \lambda)$ be a Liouville manifold with boundary (without corners).
Let $H \subset M$ be a \emph{cooriented} smooth hypersurface such that its completion $\overline H$ has the union
$$
(\del_\infty M \cap \overline H)\cup (\overline H \cap \del M) = : \del_\infty H \cup \del \overline H
$$
as its (topological) boundary.
$H$ is a \emph{$\sigma$-sectorial hypersurface} if it satisfies the following:
\begin{enumerate}
\item $Z$ is tangent to $H$ near infinity,
\item $H_\infty (= \del_\infty H) = \del_\infty M \cap H \subset \del_\infty M$
is a convex hypersurface of the contact manifold
$\del_\infty M$,
\item The canonical projection map $\pi:H \to \cN_H$ has a continuous section and
each of its fiber is homeomorphic to $\RR$.
\end{enumerate}
\end{defn}

 The definition of Liouville $\sigma$-sectors with corners
strongly relies on the general intrinsic geometry of the transverse coisotropic collection.
Study of this geometry
in turn strongly relies on the coisotropic calculus and Gotay's coisotropic embedding
theorem of general \emph{presymplectic manifolds} \cite{gotay}.

\begin{defn}[Transverse coisotropic collection]\label{defn:transverse coisotropic collection-intro}
Let $(M,\lambda)$ be a Liouville manifold with corners.
Let $H_1, \ldots, H_m \subset M$ be a collection of cooriented hypersurfaces $Z$-invariant near infinity,
that satisfies
\begin{enumerate}
\item\label{item. transverse intersection sectorial-intro} The $H_i$ transversely intersect,
\item\label{item. coisotropic sectorial-intro} All pairwise intersections $H_i \cap H_j$ are coisotropic.
\end{enumerate}
\end{defn}

Denote the associated codimension $m$ corner by
$$
C = H_1 \cap \cdots \cap H_m
$$
and by $\cN_C$ the leaf space of the null-foliation of the coisotropic submanifold $C$.
Then we prove in Subsection \ref{subsec:integrable} that for each choice of sections $\sigma=\{\sigma_1, \cdots, \sigma_m\}$,
\begin{itemize}
\item there is a natural fiberwise $\RR^m$-action on $C$
which is a simultaneous linearization of the characteristic flows of the sectorial hypersurfaces $H_i$'s.
\item each fiber is diffeomorphic to $\RR^m$ utilizing the standard construction of action-angle variables
in the integrable system.
\end{itemize}
(See \cite{arnold:mechanics} and Corollary \ref{cor:contractible-fiber} for the relevant discussion.)
This leads us to the final definition of Liouville $\sigma$-sectors with corners.

\begin{defn}[Liouville $\sigma$-sectors with corners]\label{defn:intrinsic-corners-intro}
Let $M$ be a manifold with corners equipped with
a Liouville one-form $\lambda$. We call $(M,\lambda)$ a \emph{Liouville $\sigma$-sector with corners}
if at each corner $\delta$ of $\del M$, the corner can be expressed as
$$
C_\delta := H_{\delta,1} \cap \cdots \cap H_{\delta,m}
$$
for a collection $\{H_{\delta,1}, \cdots, H_{\delta,m}\}$ such that
\begin{enumerate}
\item it is a  transversely coisotropic,
\item each fiber of the canonical projection
$$
\pi_{C_\delta}: C_\delta \to \cN_{C_\delta}
$$
is contractible.
\end{enumerate}
We call such a corner a \emph{$\sigma$-sectorial corner of codimension $m$}.
\end{defn}

We will show that each choice of $\sigma$ will canonically provide an equivariant splitting data
$$
(F, \{(R_i,I_i)\}_{i=1}^m), \quad d\lambda = \omega_F \oplus \sum_{i=1}^m dR_i \wedge dI_i
$$
on $\nbhd(C_\delta) \cong F \times \C_{\Re \geq 0}^m$
for $\sigma$-sectorial corners that is equipped with the Hamiltonian $\RR^m$-action
whose moment map is precisely the coordinate projection
$$
\nbhd(C) \to \R_{\geq 0}^m; \quad x \mapsto (R_1(x), \ldots, R_m(x)).
$$
(See Theorem \ref{thm:splitting-data-corners} for the precise statement.)

We also prove the following equivalence result.

\begin{theorem} Definition \ref{defn:intrinsic-corners-intro} is equivalent to that
of Liouville sectors with corners from \cite{gps2}.
\end{theorem}

We refer to Definition \ref{defn:sectorial-collection} for the comparison between
Definition \ref{defn:intrinsic-corners-intro} and the definition of Liouville sectors with
corners from \cite{gps2}. The following is straightforward from our definition.
(Compare this with the discussion on the product arund
 Lemma 2.21 \cite{gps} and in Section 6 \cite{gps2}.)

\begin{prop}[Proposition \ref{prop:product}]
The set of Liouville $\sigma$-sectors with corners forms a monoid: for any two
Liouville sectors  with corners  $M_1, \, M_2$ the product $M_1 \times M_2$
 is canonically a LIouville $\sigma$-sector with corners.
\end{prop}
This monoidal property has been used to construct a monoidal property of
wrapped Fukaya category generated by \emph{(gradient) sectorial Lagrangians}
in \cite{oh:sectorial,oh:gradient-sectorial}.

\subsection{Automorphism group of Liouville $\sigma$-sectors with corners}

Thanks to Theorem \ref{thm:equivalence-intro} or Theorem \ref{thm:GPS-equivalence-intro},
our definition of Liouville $\sigma$-sectors with corners enables us to
give a natural notion of Liouville automorphisms of Liouville sectors (with corners) from \cite{gps,gps2}
which is similar to the case without boundary.

We start with the following observation that every symplectic diffeomorphism of
$(M,\del M)$ induces a presympletic diffeomorphism on $\del M$
and hence preserves the characteristic foliation of $\del M$.

This  enables
 us to define the ``structure'' of Liouville $\sigma$-sectors (Definition \ref{defn:geometric-structure}),
and  to identity its automorphism group $\aut(M,\lambda)$ in the same way as
for the Liouville manifold case.

\begin{defn}[Automorphisms group $\aut(M,\lambda)$] Let $(M,\lambda)$ be
a Liouville $\sigma$-sector, possibly with corners.
We call a diffeomorphism $\phi: (M,\del M) \to (M,\del M)$
a Liouville automorphism if $\phi$ satisfies the following:
$$
\phi^*\lambda = \lambda + df
$$
for a compactly supported function $f: M \to \RR$.
We denote by $\aut(M,\lambda)$ the set of automorphisms of $(M,\lambda)$.
\end{defn}
Obviously $\aut(M,\lambda)$ forms a topological group which is a subgroup of $\symp(M,d\lambda)$,
the group of symplectic diffeomorphisms of $(M,d\lambda)$.

\begin{remark}\label{rem:stratawise-presymplectic} 
The above discussion on the automorphism can be
naturally extended to the case of with corners.
Recall that a manifold with corners $X$ is (pre)symplectic if there is a stratawise
(pre)symplectic form $\omega$, i.e., a collection of (pre)symplectic forms
$$
\{\omega_\alpha\}_{\alpha \in I}
$$
that is compatible under the canonical inclusion map of strata
$$
\iota_{\alpha\beta}: X_\alpha \hookrightarrow X_\beta, \quad \alpha < \beta
$$
i.e., $\omega_\alpha = \iota_{\alpha\beta}^*\omega_\beta$. Here $I$ is the
POSET that indexes the strata of the stratified manifold $X$. By definition, a diffeomorphism
between two manifolds with corners preserves dimensions of the strata.
\end{remark}

\medskip

Finally we would like to mention that different geometric nature of $(\del_\infty M, \xi_\infty)$ and $(\del M, \lambda_\del)$
is partially responsible for the difficulty, as
manifested in its construction given in \cite{oh:sectorial}, of the construction of
a pseudoconvex pair $(\psi, J)$ in a neighborhood
$$
\nbhd(\del_\infty M \cup \del M)
$$
such that the almost complex structures $J$ is amenable to the (strong) maximum principle)
for the (perturbed) pseudoholomorphic maps into the Liouville sectors.
We anticipate that together with the local nature of the  maximum principle proof 
of $C^0$-estimates from \cite{oh:sectorial} and its natural monoidality of  
Liouville $\sigma$-sectors will facilitate the study of K\"unneth-type functors and 
simplicial descents of wrapped Fukaya categories.
(See \cite{gps-microlocal}, \cite{oh-tanaka:actions}, \cite{asplund} and others for the relevant study.)

The current paper is the Part I of the arXiv posting arXiv:2110.11726(v1)-(v3) of the title
``Monoid of Liouville sectors with corners and its intrinsic characterization''. The paper is
now split into two, Part I becoming the current paper and Part II split away
 to a separate paper \cite{oh:gradient-sectorial}.

\bigskip

{\bf Acknowledgments:}
The present work is supported by the IBS project IBS-R003-D1. We would like to thank Hiro Lee Tanaka
for his collaboration on the study of Liouville sectors and for useful comments on the
preliminary draft of the present work. We would like to express our sincere thanks to
the unknown referee for pointing out several mistakes in our presentation
which have led us to revising the proof given in Section \ref{sec:GPS-question} and
significantly improving the overall presentation of the paper.
\bigskip

\noindent{\bf Conventions:}
\begin{itemize}
\item Hamiltonian vector field $X_H$: $X_H \rfloor \omega = dH$,
\item Canonical one-form $\theta_0$ on $T^*Q$: $\theta_0 = \sum_{i=1}^n p_i dq_i$,
\item Canonical symplectic form $\omega_0$ on $T^*Q$: $\omega_0 =  d(-\theta) = \sum_{i=1}^n {dq_i \wedge dp_i}$,
\item Liouville one-form on $(T^*Q, \omega_0)$: $\lambda = -\theta= -\sum_{i=1}^n p_i dq_i$,
\item Symplectization $SC$ of contact manifold $(C,\theta)$: $SC = C \times \R$ with $\omega = d(e^s \pi^*\theta)$.
Here note that \emph{we write the $\RR$-factor after the $C$-factor}.
\item Contact Hamiltonian: The contact Hamiltonian of contact vector field $X$ on a contact manifold
$(M,\theta)$ is given by $-\theta(X)$. (See \cite{oh:contacton-Legendrian-bdy} for the same convention
adopted in the general framework of contact dynamics.)
\end{itemize}

\bigskip

\noindent{\bf Notations:}
\begin{itemize}
\item $\overline M$: the completion of $M$ which is $M \coprod \del_\infty M$.
\item $DM$: the union $\del_\infty M \cup \del M$ in $\overline M$.
\item $F_\infty: = \del_\infty M \cap \del M$: the ideal boundary of $\del M$.
\item $\del_\infty M = \del_\infty^{\text{\rm Liou}}M$: the ideal boundary of a Liouville manifold $M$ (or sector).
\item $\aut(M,\lambda)$: The group of Liouville diffeomorphisms of Liouville $\sigma$-sector $(M,\lambda)$.
\item $\omega_\del = d\lambda_\del$ : The induced presymplectic form on $\del M$ with $\lambda_\del := \iota^*\lambda$.
\item $\aut(M,\lambda_\del)$: The group of pre-Liouville diffeomorphisms of
exact presymplectic manifolds $(M,d\lambda_\del)$.
\item $H$ : a $\sigma$-sectorial hypersurface $H \subset M$.
\item $H_\infty = \del_\infty M \cap H$: the ideal boundary $H$.
\item Constants $N$ and $C$: We consistently use the letter $N$ to write the level of
symplectization radial function $s$ and the letter $C$ for the level of the characteristic flow
of the sectorial hypersurface or for the $\R$-coordinate in the product $F \times \R$.
\end{itemize}

\section{Preliminaries}

We start with the case without corners but with nonempty boundary $\del M$, postponing
the study of the case with corners till Section \ref{sec:sectors-with-corners}.

For the comparison, we recall the definition of Liouville sectors in \cite{gps}.
In fact we will consider the definition of sectorial hypersurfaces in \cite[Definition 9.2]{gps2}
and restrict that to the sectorial boundary of a Liouville domain.

To facilitate our exposition, we utilize Giroux's notion of the \emph{ideal completion}
of the Liouville domain $(W,\lambda)$.
\begin{defn}[Ideal completion $\overline M$ \cite{giroux}] \label{defn:giroux}
\begin{enumerate}
\item
An \emph{ideal Liouville domain} $(W,\omega)$ is a domain endowed with an ideal Liouville structure $\omega$.
\item
The \emph{ideal Liouville structure} is an exact symplectic form on $\Int W$ admitting a primitive
$\beta$ such that: For some (and then any) function $u:W \to \RR_{\geq 0}$ with regular level set
$\del_\infty W = \{u=0\}$, the product $u \beta$ extends to a smooth one-form $\lambda$ 
on $W$ which induces a contact form on $\del W$.
\item When a Liouville manifold $(M,\beta)$ is Liouville isomorphic to
$(\Int W,\beta)$, we call $W$ the ideal completion of $M$ and denote it by $\overline M$.
\end{enumerate}
\end{defn}
\begin{remark}
Firstly, this definition provides a natural topology and smooth
structure on the completion $\overline M$ and a Liouville structure on $M(=\Int W)$ as
an open Liouville manifold. Secondly it also provides a natural
class of Liouville diffeomorphisms on $M$ as the restriction of
diffeomorphisms of $\overline M = W$. (See \cite{giroux}.)
\end{remark}

For a (noncompact) Liouville manifold $(M,\lambda)$ (without boundary)
its ideal boundary, denoted by $\del_\infty M$,
is defined to be the set of asymptotic rays of Liouville vector field $Z$.
Then the \emph{ideal completion} is the coproduct
$$
\overline M = M \coprod \del_\infty M
$$
equipped with the obvious topology. We refer readers to \cite{giroux} for complete details.
For readers' convenience, we provide some summary thereof  in Appendix \ref{sec:giroux}
that are to be used later in the study of Theorem \ref{thm:GPS-question-intro}.

\subsection{Liouville manifolds with boundary and orientations}
\label{subsc:liouville-with-bdy}

When $(M,\lambda)$ is a Liouville sector with boundary $\del M$,
its ideal boundary is still well-defined by the $Z$-invariance requirement near infinity put on $\del M$
in the definition of Liouville sectors \cite{gps} and so is its completion $\overline M$. Then we have the formula for the topological boundary
$$
\del \overline M = \del_\infty M \cup \del M.
$$
To ease our exposition, we often abuse our notation
$$
DM: = \del_\infty M \cup \del M
$$
for the coproduct $\del_\infty M \coprod \del M$ after the present section,
as long as there is no danger of confusion. Likewise
we also abuse the notation like
$$
\del_\infty M \cap H : = \del_\infty M \cap \overline H
$$
for ideal boundary of $\sigma$-sectorial hypersurface $H$
where the intersection is actually taken as a subset of  $\overline M$. 
For the simplicity of notation, we will also use
\be\label{eq:H-infty}
H_\infty: = \del_\infty M \cap \overline H
\ee
similarly as we denoted
$F_\infty = \del_\infty M \cap \del M$ when $H = \del M$.

\subsubsection{Null foliation}

We  recall  the well-known fact that each hypersurface $H \subset \Mliou$
in a symplectic manifold $(\Mliou,\omega)$
carries the canonical characteristic foliation $\cD$. The definition of this foliation is based on the
fact that any hypersurface $S$ of $(M,\omega)$ is a \emph{coisotropic} submanifold in that
\begin{enumerate}
\item We have
$$
(T_x H)^{\omega_x} \subset T_x H,
$$
for any $x \in H$, where $(T_x H)^{\omega_x}$ is the $\omega_x$-orthogonal complement
$$
(T_x H)^{\omega_x}: = \{v \in T_x M \mid \omega_x(v,w) = 0 \, \forall w \in T_xH\}.
$$
\item Let $\iota_H: H \to M$ be the inclusion map and
$$
\ker \iota_H^*\omega_x:= \{v \in T_x H \mid \omega_x(v, w) = 0 \, \forall w \in T_xH\}
$$
has constant rank 1 for all $x \in H$.
\end{enumerate}
Then we denote $\cD = \ker \iota_H^*\omega$ which defines a 1-dimensional (integrable) distribution of $H$, and
call it the characteristic distribution or the null distribution of $H$.
We denote by $\cN_H$ the leaf space of
the associated foliation. It is also well-known that $\cD$ carries a transverse symplectic structure
which induces one on the leaf space
\be\label{eq:NH}
\cN_{H}: = H /\sim
\ee
chart-wise. With slight abuse of notation, we will also denote by $\cD$ the associated foliation.
Of course, the quotient topology of a leaf space may not be Hausdorff in general.
We will show that under the conditions laid out in Definition~\ref{defn:liouville-sector-intro}, the aforementioned transverse symplectic form, as well as its smooth structure, descends to
the leaf space.

We denote the ideal boundary of $H$ (reltaive to $Z$) by $\del_\infty H=: H_\infty$.
Then
$$
H_\infty = \del_\infty \Mliou \cap \overline H.
$$
At each point $x \in \overline H \cap \nbhd(\del_\infty M) \supset H_\infty$, we have a natural exact sequence
	\be\label{eq:D-orientation}
0 \to \cD_x \to T_x H  \to T_x H/\cD_x \to 0.
	\ee
The quotient carries a canonical symplectic bilinear form and so carries a natural
symplectic orientation.

\begin{choice}[Orientation of $\cD$]\label{choice:orientation-CD}
 Let $H \subset M$ be a proper $\sigma$-sectorial
hypersurface. Make a choice of orientation on the trivial line bundle $\cD \to H$.
\end{choice}

\begin{defn}[Presymplectic orientation on $H$]\label{defn:presymplectic-or}
Let $\cD \to H$ be given an orientation $o_{\cD}$
on a neighborhood of $H_\infty$ in $\del_\infty M$.
We call the orientation on $TH|_{H\cap \nbhd(\del_\infty M)}$ given by the direct sum orientation
$$
T_x H|_{H_\infty} = (T_x H/\cD_x) \oplus \cD_x, \quad x \in H\cap \nbhd(\del_\infty M)
$$
the \emph{presymplectic orientation} of $H$ relative to $o_{\cD}$.
\end{defn}

\begin{example}[{$F^\pm_\infty$ on $T^*[0,1]$}] Now consider the case of
the cotangent bundle $M = T^*[0,1]$ of the closed interval $[0,1]$
equipped with the Liouville form
\be\label{eq:Liouville-form}
\lambda = - p\, dq.
\ee
(This is the negative of the standard Liouville one-form $pdq$ in the cotangent bundle.)
The standard orientation of the interval induces a diffeomorphism $\Mliou \cong [0,1]_q \times \RR_p$
which carries the symplectic orientation induced by the symplectic form
$$
dq \wedge dp.
$$
(We alert the readers that this is the negative of the convention $dp \wedge dq$ used by \cite{gps}.)
The boundary $\del \Mliou \cong \{0,1\} \times \RR_p$  has 2 connected components.
\emph{The characteristic foliation's orientation is compatible with the vector field $\frac{\del}{\del p}$.}
Note that the Liouville vector field of the Liouville form \eqref{eq:Liouville-form}
on $T^*[0,1] \cong [0,1]_q \times \RR_p$ is given by the Euler vector field
\be\label{eq:Liouville-vector-field}
\vec E:= p\frac{\del}{\del p}
\ee
on $T^*M$ which vanishes at $p = 0$. So each leaf $\{q\} \times \RR_p$
of the foliation consists of 3 different orbit sets of the Liouville vector field
$$
\RR_+ = (0,\infty), \quad \{0\},  \quad \RR_- = (-\infty, 0).
$$
We may identify $\del_\infty \Mliou$ with two disjoint copies of $[0,1]$ at ``$p= \pm \infty$.'' $F_\infty$ consists of four points, which we will denote by $(0,\pm\infty)$ and $(1,\pm \infty)$ again using the informal notation allowing $p$ to attain $\pm \infty$. Under this notation, we have that
	\be
	F^+_\infty = \{(0,-\infty) , (1,\infty)\},
	\qquad\text{and}\qquad
	F^-_\infty = \{(0,\infty), (1,-\infty)\}.
	\ee
\end{example}

\begin{example}[$\dim Q \geq 2$]
More generally, let $Q = Q^n$ be a connected $n$-manifold with boundary and let $M = T^*Q$. The inclusion $T(\del Q) \into TQ$ induces a quotient map $T^*Q|_{\del Q} \to T^*(\del Q)$ of bundles on $\del Q$; the kernel induces the characteristic foliation on
$$
T^*Q|_{\del Q} = \del \Mliou.
$$
 Informally: At a point $(q,p) \in \del M$, the oriented vector defining the characteristic foliation is the symplectic dual to an inward vector normal to $\del Q$. For example, identifying $Q$ near $\del Q$ with the
right half plane with final coordinate $p_n$, in standard Darboux coordinate
$(q,p)$, the characteristic foliation is generated by
${\frac {\del}{\del p_n}}$.
\end{example}

\subsection{Convexity of $H_\infty =\del_\infty M \cap H$ and contact vector field}

By applying the notion of $\sigma$-sectorial hypersurface from Definition \ref{defn:sectorial hypersurface} to the boundary $\del M \subset M$,
we introduce the following definition. This is the counterpart of
the definition of \emph{sectorial hypersurface} given in \cite[Definition 9.2]{gps2}.

\begin{defn}[Liouville $\sigma$-sector]\label{defn:liouville-sigma-sector}
Let $M$ be a noncompact manifold with boundary such that its completion $\overline M$ has
(topological) boundary given by the union
$$
\del_\infty M \cup \del M = DM
$$
and $\del_\infty M \cap \del M$ is the codimension two corner of $\overline M$. $M$ is called a
\emph{Liouville $\sigma$-sector} if its boundary $\del M \subset M$ is a
$\sigma$-sectorial hypersurface in the sense of Definition \ref{defn:sectorial hypersurface}.
\end{defn}

To avoid some confusion with the corners in $\del M$, we call the intersection
$$
\del_\infty M \cap \del \overline M
$$
the \emph{ceiling corner}.  This is the corner of the ideal completion $\overline M$ of $M$ of codimension 2.
(We will call the genuine corners of $M$ the \emph{sectorial corners} in Section \ref{sec:sectors-with-corners}
when we consider the Liouville sectors with corners.)

Recall that $\del_\infty M$ is naturally oriented as the ideal boundary of symplectic manifold
$M$ with $Z$ pointing outward along $\del_\infty M$.

We take a contact-type hypersurface
$S_0 \subset M$  and identify a neighborhood $\nbhd(\del_\infty M)$ with
the half $S_0 \times [0,\infty)$ of the symplectization of the contact manifold
$(S_0, \iota_{S_0}^*\lambda)$.
We denote
\be\label{eq:H0}
H_0 = S_0 \cap H.
\ee
Then considering the Liouville embedding $S_0 \times [0,\infty) \hookrightarrow M$,
we can decompose $M$ into
$$
M = (M \setminus \nbhd(\del_\infty M) )\cup \nbhd(\del_\infty M)
$$
so that
\begin{itemize}
\item  $Z = \frac{\del}{\del s}$ for the symplectization form $d(e^s \pi^* \iota_{S_0}^*\lambda)$
of the contact manifold $(S_0, \iota_{S_0}^*\lambda)$
on $S_0 \times [0,\infty)$,
\item we may identify the one-form $\iota_{S_0}^*\lambda$ as
a contact form of $\del_\infty M$ by the natural diffeomorphism $S_0 \cong \del_\infty M$
induced by this Liouville embedding $S_0 \times [0,\infty) \hookrightarrow M$.
\end{itemize}

By the convexity hypothesis of $H_\infty:= H \cap \del_\infty \Mliou$ in $\del_\infty \Mliou$,
there exists a contact vector field $\eta$ of the contact structure $(\del_\infty M,\xi_\infty)$
on a neighborhood of $H_\infty$ in $\del_\infty M$
that is transverse to $H_\infty$.

Since there are different sign conventions in the literature in defining the contact Hamiltonian
associated to a contact vector field, we set our sign convention as follows
by adopting the one used by the present author in \cite{oh:contacton-Legendrian-bdy} and
its sequels, which also coincides with that of \cite{dMV}.

\begin{defn}[Contact Hamiltonian]\label{defn:contact-Hamiltonian} We call the function
$$
h: = - \theta(\eta)
$$
the \emph{contact Hamiltonian} associated to the contact vector field $\eta$.
\end{defn}

\begin{remark}\label{rem:dividing-set}
It is well-known that a choice of contact vector field $\eta$ transverse to $H_\infty$ in $\del_\infty M$,
gives rise to a decomposition of $H_\infty$ into
\be\label{eq:Hinfty-decompose}
H_\infty = H_\infty^+ \sqcup \Gamma_\eta \sqcup H_\infty^-
\ee
where $H_\infty^\pm$ and $\Gamma_\eta$ are defined by
$$
H_\infty^\pm = \{x \in H_\infty \mid  \pm \theta(\eta(x)) > 0\}, \quad
 \Gamma_\eta = \{x \in H_\infty \mid \theta(\eta(x)) = 0\}.
$$
(Recall that $\Gamma_\eta$ is called the \emph{dividing set} of $\eta$ on $H_\infty$.
See \cite{giroux} for a general study of convex hypersurface.)
Other than the existence of the contact vector field transverse to the contact distribution,
which is the defining property of the convexity of hypersurfaces,
this \emph{contact geometry of ideal boundary $\del_\infty M$} does not enter
in our study of \emph{presymplectic geometry of coisotropic submanifold, $\del M$}, which
makes our affirmative anwser to the question \cite[Question 2.6]{gps} plausible.
See Remark \ref{rem:GPS-question} below for a further elaboration.
\end{remark}

\section{Sectional characterization of sectorial hypersurfaces}
\label{sec:intrinsic}

Let $H \subset M$ be a $\sigma$-sectorial hypersurface of a Liouville $\sigma$-sector $(M,\lambda)$.
Equip the leaf space $\cN_{H}$ with the quotient topology induced by the projection $\pi = \pi_H: H \to \cN_{H}$.  The main goal of this section is to equip this quotient space with
a canonical Liouville structure induced from that of $M$.

\subsection{The leaf space is a topological manifold}
\label{subsec:topological}

Before providing a smooth atlas on $\cN_{H}$, our first order of business is to prove
the existence of topological manifold structure thereon. This is the most technical step
towards the goal of the section as common in the study of general topology argument.
The proof of this proposition occupies the rest of this subsection.

\begin{theorem}\label{thm:NH} Let $H$ be a $\sigma$-sectorial hypersurface.
The leaf space $\cN_{H}$ is a topological manifold. (In particular, $\cN_{H}$ is
second countable and Hausdorff.)
\end{theorem}

We start with the following  lemma.

\begin{lemma}\label{lem:G} There exists a neighborhood $\nbhd (\del_\infty \Mliou \cap H)$
of the ceiling corner $\del_\infty \Mliou \cap H$ in $M$ and a smooth function
\be\label{eq:G}
G: \nbhd (\del_\infty \Mliou \cap H) \to [0,\infty)
\ee
on $\nbhd (\del_\infty \Mliou \cap H)$ of $M$ that has the following properties:
\begin{enumerate}
\item $Z[G] = G$,
\item its Hamiltonian vector field $X_G$  is transverse to $H$ and
 represents the given coorientation of $H$ at each point
$x \in H \cap \nbhd (\del_\infty \Mliou \cap H)$.
\end{enumerate}
\end{lemma}
\begin{proof} By the defining data of Liouville $\sigma$-sectors, we have
\begin{itemize}
\item $H_\infty$ is convex in $\del_\infty M$,
\item  $Z$ is tangent to $H$ near infinity.
\end{itemize}
the second requirement enables us to choose a contact-type hypersurface $S_0$
far out close to $\del_\infty M$
so that $S_0 \pitchfork H$. Write the smooth hypersurface $H_0: = S_0 \cap H$ of $H$.

We take a symplectization neighborhood of $\del_\infty M$ obtained by the Liouville embedding
\be\label{eq:phiZS0}
\phi_{Z;S_0}: S_0 \times [0, \infty) \hookrightarrow M
\ee
defined by $\phi_{Z;S_0}(y,t): = \phi_Z^t(y)$. We denote by $s$
the associated radial function defined by $s(y,t): = t$. Then we have the
splitting
$$
TM|_{S_0 \times [0,\infty)}  \cong TS_0 \oplus \R \left\{\frac{\del}{\del s}\right\}
$$
and satisfies
\be \label{eq:s-S0-Z}
s^{-1}(0) = S_0, \quad Z = \frac{\del}{\del s}, \quad S_0 \cong \del_\infty M.
\ee
We also have the contact form $\theta \cong \iota|_{S_0}^*\lambda$ on $S_0$
so that we can express the Liouville form as
$$
\lambda = e^s \pi^* \theta
$$
on a neighborhood $\nbhd(\del_\infty M)$.

Using  the convexity hypothesis
of $H_\infty \subset \del_\infty M$, we can take
 a contact vector field $\eta$ on a neighborhood of $H_\infty$ in $\del_\infty M$ such that
$\eta \pitchfork H_\infty$. Take its contact Hamiltonian
 $h = - \theta(\eta)$ on a neighborhood of $H_\infty$ in $\del_\infty M$.
  (Recall the sign convention  from Definition \ref{defn:contact-Hamiltonian}
   adopted in the present paper.)
By considering the function $\pi^*h$ on a neighborhood of $H_\infty$
in $M$, we take the associated homogeneous Hamiltonian function
on the symplectization in a neighborhood of $H_\infty$ in $M$, which we  denote it by
$$
G := e^s \pi^* h
$$
which is defined on a neighborhood $H_\infty = H \cap \del_\infty M$
in $M$, say, on
$$
V \times [0,\infty) \subset s^{-1}([0, \infty)) \subset M,
$$
where $ V\subset \del_\infty M$ is an open neighborhood of $H_\infty$ in $\del_\infty M$.
Through the symplectization end Liouville embedding $S_0 \times [0,\infty) \hookrightarrow
M$, we may identify the function $h: H_\infty \to \R$ with $\pi^*h|_{\{s=0\}}: H \cap S_0 \times \{0\} \to \R$.
Then the Hamiltonian vector field $X_G|_{H \cap \{s \geq0\}}$ represents
the coorientation of $H$ compatible with the one on $H_\infty \subset \del_\infty M$
given by $\eta$.

Clearly it satisfies $Z[G] = G$ since $Z = \frac{\del}{\del s}$ thereon. This finishes the proof.
\end{proof}

\begin{remark}\label{rem:GPS-question} We would like to mention that 
a priori the characteristic foliation of a general hypersurface satisfying that $\pi: H \to \CN_H$ 
satisfying Condition (3) in Definition \ref{defn:sectorial hypersurface} could be very wild without 
the other conditions (1), (2). One of the consequences of the convexity of
$H_\infty$ in $\del_\infty M$ is the presence of the function $G$ on $\nbhd(\del_\infty M \cap H)$ which gives rise to
the taming of the behavior of the characteristic foliation of $H$
in a neighborhood $\nbhd(\del_\infty \Mliou) \cap H$. Indeed, such a taming
is also a sufficient condition for $H_\infty$ to be convex, which is precisely what
\cite[Qestion 2.6]{gps} is asking about.  We provide its affirmative answer
in Theorem \ref{thm:GPS-question} of the present paper.
\end{remark}

We fix a Riemannian metric $g$ on $M$ that is
$Z$-invariant near infinity $\del_\infty M$. More explicitly we require the metric to satisfy
\begin{itemize}
\item Near $H$, we require it to have the form
$$
g = g_H \oplus dv^2
$$
on the neighborhood $H \times (-\epsilon, \epsilon) \hookrightarrow M$ where $v$ is
the coordinate of $(-\epsilon,\epsilon)$.
\item Near $\del_\infty M$ on the symplectization end,  we require the metric to satisfy
$$
g = g_{S_0} \oplus ds^2
$$
on $S_0 \times [0, \infty)$ where $g_{S_0}$ is any Riemannian metric on $S_0$,
recalling $Z = \frac{\del}{\del s}$ on this region.
\item Near $H_\infty  = H \cap \del_\infty M$, we require that the above two choices
are compatible in that $g_H|_{H_\infty} = g_{S_0}|_{H_\infty}$ and has the form
$$
g = g_{H_\infty} \oplus du^2 \oplus ds^2.
$$
\end{itemize}

In addition, using the coorientation hypothesis on $H \subset M$, we fix a coorientation.
(For the case of $H = \del M$, we use the canonical outward coorientation.) Then we choose the
aforementioned contact vector field $\eta$ so that it defines the same coorientation as that of
the coorientation on $H \subset M$ induced by $X_G|_H$.
With the above Riemannian metric equipped with the neighborhood of $H$,
we require
\be\label{eq:symplectic-normal}
d\lambda\left(\cD_H,\frac{\del}{\del v}\right) > 0
\ee
with $\cD_H$ equipped with the one given in Definition \ref{choice:orientation-CD}:
Note that we have the exact sequence of symplectic vector bundle
\be\label{eq:exact-sequence-sympl}
0 \to \text{\rm span}\left\{\cD_H, \frac{\del}{\del v} \right\} \to TM|_H \to 
TM \Big\slash\text{\rm span}\left\{\cD_H, \frac{\del}{\del v} \right\} \to 0
\ee
where we have
$$
TM\Big\slash\text{\rm span}\left\{\cD_H, \frac{\del}{\del v} \right\} \cong TH/ \cD_H.
$$

\begin{prop}\label{prop:W}
There exists a unique vector field $W$ on $H$ that satisfies the following:
\begin{enumerate}
\item $g(W,W)\equiv 1$ and in particular $W$ is nowhere vanishing,
\item $W$ is tangent to the foliation $\cD$, and
\item The choice of $W$ is compatible with the orientation~\eqref{eq:D-orientation} of the leaves
and satisfies $d\lambda\left(W,\frac{\del}{\del v}\right) > 0$.
\end{enumerate}
\end{prop}
\begin{proof} We first recall that the Liouville vector field $Z$ is tangent to $H$ near infinity.
We define the radial coordinate $s$ as in the proof of Lemma \ref{lem:G}.

We start with defining the vector field $W \in \cD \subset TH$ on $H$ 
along the hypersurface
$s^{-1}(N) = S_0 \times \{N\} \cong S_0$ for a sufficiently large $N > 0$ in the given symplectization end.
We can express it as the sum
$$
W= Y' +  a \frac{\del}{\del s},
$$
for some function $a = a(y)$ on $S_0$, and $Y'$ tangent to $s^{-1}(N) \cap H$
for  all $N \geq 0$. Then we have
$$
0 < d\lambda \left(Y' +  a \frac{\del}{\del s},\frac{\del}{\del v}\right)
$$
and
\beastar
d\lambda \left(Y' +  a \frac{\del}{\del s},X\right) & = & 0 \quad \text{for all $ X\in TH$},\\
d\lambda \left(Y' +  a \frac{\del}{\del s},\frac{\del}{\del s}\right) & = & 0 \quad
\text{on $s^{-1}(N) \cap H$}
\eeastar
The second equation also implies $d\lambda \left(Y', \frac{\del}{\del s}\right) = 0$
since $Z = \frac{\del}{\del s}$ is tangent to $H$ for all sufficiently large $N > 0$.

Using the property that $Z$ is tangent to $H$ near infinity, we will  choose
$W$ near infinity, say for $s \geq N$ for a sufficiently large $N> 0$ so that
\be\label{eq:Wys}
W(y,s) := d\phi_Z(W(y,N)) = Y'(y) \oplus  a(y) \frac{\del}{\del s}
\ee
i.e., it is just the $s$-translation of the initial vector $W(y,0)$.
By normalizing $W$ to $W/|W|$, we may assume that $W$ has unit norm which makes
its choice unique among the vector fields tangent to $\cD$ in the orientation
given in Choice \ref{choice:orientation-CD}.

Next we would like to extend the vector field $W$  to everywhere on $H$ that still satisfies
the standing requirements (1) - (3). For this purpose, we consider equation for $W \in TH$
in the orientation from Choice \ref{choice:orientation-CD} to satisfy
\be\label{eq:W-todefine}
W \in \cD \subset TH, \quad d\lambda(W, TH) = 0, \quad d\lambda\left(W, \frac{\del}{\del v}\right) > 0.
\ee
By further requiring $|W| = 1$, the equation is uniquely solvable at
each point of $y \in H$.  This finishes the proof of
Proposition \ref{prop:W}.
\end{proof}

In the course of proving the above proposition, we have also proved the following.

\begin{cor}\label{cor:trivial-CD} \begin{enumerate}
\item The line bundle $\cD \to H$ is trivial.
\item Denote by $R: \nbhd(H) \to \R$ the defining function $R : = u$ on a neighborhood of $H$ in $M$.
Then $W = \frac{X_R}{|X_R|_g}$.
\end{enumerate}
\end{cor}
\begin{proof} Statement (1) is obvious since $W$ is nowhere vanishing section of the line bundle $\cD \to H$.
For Statement (2), we note that $X_R$ satisfies $d\lambda(X_R, TH) = dR(TH) dt(TH) \equiv 0$ and hence is
tangent to $\cD$. Furthermore we have
$$
d\lambda\left(X_R,\frac{\del}{\del u}\right) \equiv 1 > 0
$$
Then by the aforementioned uniqueness, we derive $W = \frac{X_R}{|X_R|_g}$.
\end{proof}
This corollary will be useful for the later study of intrinsic characterization
of Liouville sectors \emph{with corners}.
(Of course this is a tautological property with the original definition
of Liouville sectors from \cite{gps}.)

\begin{defn}[Leaf-generating vector field $W$ of $\cD_H$]\label{defn:W}
We call the above constructed vector field $W$ on $H$ a \emph{leaf-generating
vector field of $\cD_H$}.
\end{defn}

The next lemma states that
the leaf space $\cN_H$ is Hausdorff with respect to the quotient topology of $\pi: H \to \cN_H$.
This should be classical which  can be derived  from the property of the quotient topology and the existence 
of continuous section  $\sigma_{\text{\rm ref}}: \cN_H \to H$.  For readers' convenience,  
we give its proof in Appendix \ref{sec:Hausdorff}. 

\begin{lemma}\label{lem:Hausdorff}
The space $\cN_H$ equipped with the quotient topology of $\pi: H \to \cN_H$ is 
Hausdorff.
\end{lemma}

The next lemma shows that the presence of continuous section implies
 the triviality of the fibration $\pi: H \to \cN_H$. 

\begin{lemma}\label{lemma. M forward backward} Take a continuous section $\sigma_{\text{\rm ref}}: \cN_{H} \to H$ guaranteed by Definition~\ref{defn:liouville-sector-intro}. We write
\be\label{eq:F-ref}
F_{\text{\rm ref}}: = \Image \sigma_{\text{\rm ref}} \subset H.
\ee
Then the flow map
\be\label{eq:flow-map}
\Phi_{\text{\rm ref}}: F_{\text{\rm ref}} \times \RR \to H; \quad \Phi_{\text{\rm ref}}(y,t) =
\phi_{W}^t(\sigma_{\text{\rm ref}}(\pi(x))
\ee
is a homeomorphism.
\end{lemma}
\begin{proof} 

We will first show
\begin{enumerate}
\item
Any trajectory of $W$ eventually exits from any given compact subset $K \subset M$
both forward and backward.
\item  Moreover every leaf is a flow orbit of $W$ and vice versa.
\end{enumerate}
It is a standard fact that each leaf is second countable because the manifold $\Mliou$ is assumed to
be second countable. (This rules out the possibility for a leaf becomes a 
`Long line' \cite[pp. 71-72]{steen-seebach}.)
Note that since $W$ is regular, each leaf of $H$ of the characteristic
foliation is a flow line of the regular vector field $W$. (See \cite[Section 2.1]{candel-conlon}.)
Furthermore no leaf can be a point.
\emph{By the condition stated in Definition \ref{defn:liouville-sector-intro} (c), $W$ cannot have a nontrivial
periodic orbit either}. Therefore each flow trajectory $t \mapsto \Phi_{\text{\rm ref}}^t(y)$
in $H$ defined on $\RR$ is one-to-one, and hence $\Phi_{\text{\rm ref}}$ is a one-one map.

Furthermore there is a uniquely defined $T \in \RR$ such that
$\phi_{W}^T(\sigma_{\text{\rm ref}}(\pi(x)) = x$ for each $x \in H$.
We define a function $T: \Mliou \to \RR$ by
\be\label{eq:T}
T(x): = \text{``the reaching time of the flow of $W$ issued at $\sigma_{\text{\rm ref}}(\pi(x))$''.}
\ee
\begin{lemma}\label{lem:continuity-T}
The  function $T$ is continuous.
\end{lemma}
\begin{proof} Let $x \in M$ and set $x_0: = \sigma_{\text{\rm ref}}(\pi(x))$. Choose a foliation chart 
$(U_\alpha, \varphi_\alpha)$ with $\varphi_\alpha = (y_1, \ldots, y_{2n-1}, t)$ with $x_0 = (0,0)$ so that 
$W = \frac{\del}{\del t}$. It follows from the Hausdorff property of $\cN_H$, we may assume that
the restriction $\sigma_{\text{\rm ref}}$ to 
$\pi(U_\alpha)$ defines a (local) section of $\pi$. Since the map $\{t = 0\}$ is also the image 
of continuous section on the same domain $\pi(U_\alpha)$, we can express 
$$
F_{\text{\rm ref}}\cap U_\alpha = \{(y, t) \in U_\alpha \mid t = \sigma'(y)\}
$$
 with $y = (y_1, \cdots, y_{2n-1})$
for some continuous one-to-one map $\sigma: \{t = 0\} \cap U_\alpha \to U_\alpha$.

Write $T_x: = T(x)$ and consider the pair
$$
(U_\beta, \varphi_\beta) = (\phi_W^{T_x}(U_\alpha), \varphi_\alpha \circ  \phi_W^{T_x})
$$
which defines a foliation chart on $U_\beta$ at $x$. Write $\varphi_\beta = (y', t')$.
Then $\phi_W^{T_x} \circ \sigma_{\rm ref}$ defines another local continuous section on $\pi(U_\alpha)$,
and we have
\beastar
\phi_W^{T_x}(F_{\text{\rm ref}} \cap U_\alpha) & = &
\left( \Image \phi_W^{T_x} \circ \sigma_{\text{\rm ref} }\right)
\cap U_\beta \\
& = & \left \{(y',t') \in U_\beta \, \Big| \, t' = \sigma\left( \left( \phi_W^{T_x}\right)^{-1}(x')\right) \right\}.
\eeastar
By definition of $(U_\beta,\varphi_\beta)$, we have the relation
$$
T_{x'} = T_x + t \left( (\phi_W^{T_x})^{-1}(x')\right) - t(x)
$$
for all $x' \in U_\beta$. This formula clearly shows that $t'$ is a continuous function on $U_\beta$, and 
satisfies
$$
T(x') - T(x) = T_{x'}  - T_x = t \left( (\phi_W^{T_x})^{-1}(x')\right) - t(x)
$$
for all $x' \in U_\beta$. Since $x' \mapsto t\left((\phi_W^{T_x})^{-1}(x')\right)$ is continuous 
at any point $x'$ where the map is
defined and $ t\left(\phi_W^{T_x})^{-1}(x) \right)= t(x)$,
this explicit formula shows that $T$ is continuous at $x$. Since continuity is local,
this proves continuity of $T$.
\end{proof}
Therefore the inverse image $T^{-1}(-C,C)$ is an open subset and $H$ is an increasing union
$$
H = \bigcup_{C \in \NN} T^{-1}(-C,C)
$$
of open subset $T^{-1}(-C,C)$.

Let $K \subset H$ be any compact subset. Then $K \subset T^{-1}(-C,C)$ for some $C > 0$.
Since $|W| = 1$ and $W$ is tangent to the leaf $\ell_y$ through $y$ of the characteristic foliation, 
any point $y \in F_{\text{\rm ref}}$ we have
$$
\phi_W^{\pm(2\max\{C,|T(y)|\})}(y) \cap K = \emptyset.
$$
This proves the aforementioned claims.

Then, combining this with the aforementioned completeness, we can define another map
\be\label{eq:Psi-ref}
\Psi_{\text{\rm ref}}: H \to F_{\text{\rm ref}} \times \RR; \quad \Psi_{\text{\rm ref}}(x)
 = (\sigma_{\text{\rm ref}}(\pi(x)), T(x)).
\ee
By construction, $\Psi_{\text{\rm ref}}$ is continous and satisfies
$$
\Psi_{\text{\rm ref}} \circ \Phi_{\text{\rm ref}} = id|_H, \quad
\Phi_{\text{\rm ref}} \circ \Psi_{\text{\rm ref}} = id|_{F_{\text{\rm ref}} \times \R}
$$
This finishes the proof of Lemma~\ref{lemma. M forward backward}.
\end{proof}

Consider the leaf map $\pi_{\text{\rm ref}}: F_{\text{\rm ref}} \to \CN_H$ where
$F_{\text{\rm ref}}$ equipped with the subspace topology of $H$ and $\CN_H$ is the quotient
topology of the projection $\pi:H \to \CN_H$.

\begin{cor}\label{cor:piN homeo} The leaf map $\pi_{\text{\rm ref}}: F_{\text{\rm ref}} \to \CN_H$
is a homeomorphism.
\end{cor}
\begin{proof}
Since $\sigma_{\text{\rm ref}}: \cN_H \to H$ is a section, we have
$\pi_{\text{\rm ref}} \circ \sigma_{\text{\rm ref}} = id_{\cN_H}$ which shows
$\pi_{\text{\rm ref}}$ is surjective.

On the other hand, if $y_1 \neq y_2$ in $F_{\text{\rm ref}}$, then $\ell_{y_1} \neq
\ell_{y_2}$ since otherwise we would have
$$
y_1 = \sigma_{\text{\rm ref}}(\ell_{y_1}) = \sigma_{\text{\rm ref}}(\ell_{y_2}) = y_2
$$
which is a contradiction. This shows that $\pi_{\text{\rm ref}}$ is a bijective
continuous map.

By construction, the map
$$
\ell \mapsto \sigma_{\text{\rm ref}}(\ell); \, \CN_H \to F_{\text{\rm ref}}
$$
defines a continuous map which also satisfies
$\pi_{\text{\rm ref}} \circ \sigma_{\text{\rm ref}} = id|_{\CN_H}$,
and $\sigma_{\text{\rm ref}}\circ \pi_{\text{\rm ref}} = id|_{F_{\text{\rm ref}}}$.
This proves that the map $\sigma_{\text{\rm ref}}$ is a continuous inverse of
$\pi_{\text{\rm ref}}$.  Therefore the map $\pi_{\text{\rm ref}}$ is a homeomorphism.
\end{proof}

Now we go back to the proof of Theorem \ref{thm:NH}.

\begin{proof}[Wrap-up of the proof of Theorem  \ref{thm:NH}]
First we show the following.

\begin{lemma}\label{lem:NH-manifold}
$F_{\text{\rm ref}}$ with the subspace topology of $H$
is  locally Euclidean (and in particular, locally compact).
\end{lemma}
\begin{proof} 
To see the locally Euclidean property of $F_{\text{\rm ref}}$, let $x_0 \in F_{\text{\rm ref}}$
be any given point. We have only to note that \eqref{eq:Psi-ref} induces a homeomorphism
$$
U/\sim \,\,  \to F_{\text{\rm ref}}\cap U
$$
for a sufficiently small foliation chart $U$ containing $x_0$ where $\sim$ is the orbit equivalence with respect to
$W$. Since $U/\sim$ is homeomorphic to $\R^{2n-1}$, so is $F_{\text{\rm ref}}\cap U$. This proves that
 $F_{\text{\rm ref}}\cap U$ is locally Euclidean.
\end{proof}

Now combination of Corollary \ref{cor:piN homeo} and Lemma~\ref{lem:NH-manifold}
 finish the proof of Theorem \ref{thm:NH}.
\end{proof}

\subsection{Smooth structure on the leaf space}
\label{subsec:smooth-structure}

When the leaf space is Hausdorff and locally Euclidean, the well-known construction
of coisotropic reduction (or symplectic reduction) applies to prove existence of
 the symplectic structure on the leaf space \emph{once the smooth structure
 on the leave space is equipped.} (See \cite{abraham-marsden} for example.)
Since we also need to construct the map $\Psi$ appearing in the statement of Theorem
\ref{thm:equivalence-intro} and will also use the details of the proof later,
we provide the full details of the existence proofs of both structures below along the way
partly for readers' convenience.

The goal of this section is to prove the first item of Theorem~\ref{thm:equivalence-intro}.
We start with the following proposition whose proof will occupy entirety of this subsection.

\begin{prop}\label{prop:leaf-space-structure}
The leaf space $\cN_{H}$ carries a canonical smooth manifold structure such that
\begin{enumerate}
\item
$\pi: H \to \cN_H$ is a smooth submersion, and
\item there is a smooth diffeomorphism
$\Psi: H \to \cN_H \times \R$ which makes the following diagram commute
\be\label{eq:Psi-diagram}
\xymatrix{ H \ar[dr]_{\pi_H} \ar[rr]^{\Psi} && \ar[dl]^{\pi_1}\cN_H \times \R  \\
& \cN_H
}
\ee
\end{enumerate}
\end{prop}

We follow the standard notation of~\cite{candel-conlon} in our discussion of foliations.
It follows from a well-known result in foliation theory that the foliation $\cF$ is
determined by its holonomy cocycle
$\gamma = \{\gamma_{\alpha\beta}\}_{\alpha, \beta \in \mathfrak U}$ with
$$
\gamma_{\alpha\beta}: y_\beta(U_\alpha \cap U_\beta) \to y_\alpha(U_\alpha \cap U_\beta).
$$
arising from the transverse coordinate map $y_\alpha: U_\alpha \to \FF^{2n-2} = \RR^{2n-2}$ or $\HH^{2n}$.

Each $y_\alpha$ is a submersion and $\gamma_{\alpha\beta}$ is given by $y_\alpha = y_\alpha(y_\beta)$
in coordinates. (See e.g., \cite[Definition 1.2.12]{candel-conlon}.)
Furthermore for the null foliation $\cF$ of the coisotropic submanifold $H$, we can choose
a foliated chart $\cU = \{(U_\alpha,\varphi_\alpha)\}_{\alpha \in \mathfrak U}$ so that the associated cocycle elements $\gamma_{\alpha\beta}$ become symplectic, i.e., the foliation $\cF$ carries a transverse symplectic structure.
We refer readers to the proof of Proposition \ref{prop:leaf-space-structure} below for the details.

\begin{remark} When $H$ has corners, the foliated chart $B = B_\tau \times B_{\pitchfork}$
means that the
\emph{tangential factor} $B_\tau$ of the foliated chart has no boundary but the \emph{transverse factor}
$B_\pitchfork$ has a boundary. (See e.g., \cite[Definition 1.1.18]{candel-conlon} for the definition.)
\end{remark}

We will first show that the above holonomy cocycle naturally descends to a smooth atlas on $\cN_{H}$
\emph{under the defining condition of $\sigma$-sectorial hypersurface above,
especially in the presence of a
continuous section of the projection $\pi_H:H \to \cN_{H}$.}

For this purpose, we consider a coherent regular foliated atlas $\{\varphi_{\alpha}: U_\alpha \to \RR^{2n-1}\}$,
and its associated foliation cocycle $\gamma = \{\gamma_{\alpha\beta}\}$
(see e.g., \cite[Section 1.2.A]{candel-conlon}).

By considering a refinement  $\{U_{\alpha'}\}$ of the given
covering, we can choose a collection of foliated charts
$\varphi_{\alpha'}: U_{\alpha'} \to  \RR^{2n-2} \times \RR$ of the form
\be\label{eq:yi'H}
( y_1, \ldots, y_{2n-2},t)
\ee
whose transverse coordinate $(y_1, \ldots, y_{2n-2})$ satisfies
\be\label{eq:yi'H-chart}
dt(W) \equiv 1.
\ee
We take a maximal such collection which we denote by
\be\label{eq:cO'}
\cO' = \{(\varphi_{\alpha'},U_{\alpha'})\}.
\ee
By the definition of transverse coordinates $(y_1,\cdots, y_{2n-2})$ of the foliated chart, it follows that
the collection thereof defines a {\em smooth} atlas of $\cN_{H}$. We write the resulting
atlas of $\cN_{H}$ by
\be\label{eq:CO'}
[\cO']:= \{[\varphi_{\alpha'}]: [U_{\alpha'}]\ \to \RR^{2n-2}\}.
\ee
\begin{lemma}\label{lemma. pi is submersion}
The projection map $\pi: H \to \cN_{H}$ is a
smooth submersion.
\end{lemma}
\begin{proof} To show smoothness of $\pi$, we will show that for any smooth function $f: \cN_{H} \to \RR$
the composition $f\circ \pi$ is smooth.
For this purpose, at any point $x$, we consider the foliated chart
$\varphi_\alpha: U_\alpha \to \RR^{2n-1}$ given above in \eqref{eq:yi'H}.

Let $f: \cN_{H}  \to \RR$ be any smooth function on $\cN_{H}$.
With respect to the aforementioned foliated atlas of $H$,
we will show that $f \circ \pi$ is smooth at
every point $x \in H$.

\emph{If $x$ is contained in $U_{\alpha'}$,}
we have
$$
(f\circ \pi) \circ (\varphi_{\alpha'})^{-1}(y_1, \cdots, y_{2n-2},t)
= f \circ [\varphi_{\alpha'}]^{-1}(y_1,\cdots, y_{2n-2})
$$
The right hand side is smooth in the variables $y_1, \cdots, y_{2n-2}$ by the
hypothesis on $f$, and does not depend on $t$-variable. This in particular implies that
the left hand map $(f\circ \pi) \circ (\varphi_{\alpha'})^{-1}$ is smooth at $x$.

\emph{Otherwise,} let $(\varphi_\beta, U_\beta)$ be a foliation chart at $x$. We take a flow map
$\phi_{W}^T$ satisfying $y: = \phi_{W}^T(x) \in U_\beta'$ for some
chart $(\varphi_{\beta'}, U_{\beta'}) \in \cO'$ at $y$ given by
$$
\left(U_{\beta'} = \phi_{W}^T(U_{\beta}), \quad \varphi_{\beta'} = \varphi_\beta \circ (\phi_{W}^T)^{-1}\right)
$$
which is contained in $\cO'$ by the maximality of the collection $\cO'$.

Therefore the map $(f\circ \pi) \circ \varphi_{\beta'}^{-1}$ is smooth at $y = \phi_{W}^T(x) \in U_{\beta'}$.
We can factorize $f \circ \pi$ into
$$
f\circ \pi = \left((f\circ \pi) \circ \varphi_{\beta'}^{-1}\right) \circ
\left((\varphi_{\beta'} \circ \phi_{W}^T|_{U_\beta})\right)
$$
which is a composition of two smooth maps and so smooth at $x$.
This implies $f\circ \pi$ is smooth at $x$.
This finishes the proof of smoothness $\pi \circ f$ for all smooth
function $f: \cN_{H} \to \RR$ and hence proves that  $\pi$ is smooth.

Submersivity of $\pi$ is
obvious by the above construction.
\end{proof}

\subsection{Construction of a smooth section: smoothing}
\label{subsec:smoothing}

Finally, we would like to improve the existence of continuous section of $\pi: H \to \cN_H$ to
a smooth one $\sigma^{\text{\rm sm}}: \cN_H \to H$. For this purpose, we apply the
`standard mollifer smoothing and a partition of unity'. However  a priori the set of 
sections of the fibration $H \to \cN_H$ is not a linear space which prevents us from
directly implementing the smoothing of the sections.  

The first  order of business for our purpose is to reduce the problem of smoothing 
to that of smoothing a section of certain smooth line bundle. 
For this purpose, we need to choose a
collection of the atlas of foliated atlases of $\cN_H$ that is compatible with the 
flow of the leaf-generating vector field $W$ on $H$.

For the simplicity of notation and exposition, we write the maximal atlas $\cO'$ 
chosen in the previous subsection back as 
$\cO$ and the atlas of $\cN_H$ given in \eqref{eq:CO'} without prime.

Thanks to the property \eqref{eq:yi'H-chart},  the transition map 
$$
\varphi_{{\alpha\beta}} = \varphi_\alpha \circ \varphi_\beta^{-1}: \varphi_\beta(U_\alpha \cap U_\beta)\to \varphi_\alpha(U_\alpha \cap U_\beta)
$$
associated to the covering $\cO$  has the form
\be\label{eq:varphialphabeta}
\varphi_{\alpha\beta}(y,t) = (\psi_{\alpha\beta}(y),t + c_{\alpha\beta}(y))
\ee
where $y = (y_1, \cdots, y_{2n-2})$ on $\varphi_\alpha(U_\alpha \cap U_\beta)$
for some smooth functions $c_{\beta\alpha}$ and $\psi_{\beta\alpha}$
on $\varphi^\alpha(V_\alpha \cap V_\beta)$.

A direct translation of the cocycle condition of $\{\varphi_{\alpha\beta}\}$ gives 
rise to the following identites for $c_{\alpha\beta}$.

\begin{lemma} \label{lem:c-alphabeta} Let $\psi_{\alpha\beta} $ be  the transition map 
for the transverse coordinate charts of $\varphi_\alpha$ which is given by
 $$
\psi_\alpha \circ \psi_\beta^{-1}: \psi_\beta([U_\alpha] \cap [U_\beta])
 \to \psi_\alpha([U_\alpha] \cap [U_\beta]).
 $$
Then the collection $c_{\alpha\beta}$ satisfies
\be\label{eq:c-alphagamma}
c_{\alpha\gamma} = c_{\beta\gamma} + c_{\alpha\beta} \circ \psi_{\beta\gamma}
\ee
 In particular, $c_{\alpha\alpha} \equiv 0$ for all $\alpha$.
\end{lemma}

The rest of this subsection will be occupied by the proof of the following.
\begin{prop} \label{prop:smoothing}
There exists  a smooth section $\sigma^{\text{\rm sm}}: \cN_H \to H$ and a
diffeomorphism $\Psi: H \to \cN_H \times \R$ such that
\be\label{eq:smooth-sigma}
\sigma^{\text{\rm sm}}(\ell) = \Psi^{-1}(\ell,0).
\ee
which makes the  diagram  \eqref{eq:Psi-diagram} commute.
\end{prop}

We first provide some general discussion on the coordinate representation of
sections of $\pi:H \to \cN_H$.
Let $\sigma: \cN_H \to H$ be a continuous section of $\pi$ and
$T_\sigma: H \to \R$ be the continuous function associated to $\sigma$
given in \eqref{eq:T}. Then we have
\be\label{eq:Fref}
F_{\text{\rm ref}} := \Image \sigma = T_\sigma^{-1}(0)
\ee
and a homeomorphism $\Psi_\sigma: H \to \cN_H \times \R$ of the type
$$
\Psi_\sigma(x) = (\pi_H(x), T_\sigma(x))
$$
whose inverse $\Phi_\sigma: \cN_H \times \R \to H$ is given by the flow map
$$
\Phi_\sigma(\ell,t) = \phi_{W}^t(\sigma(\ell))
$$
such that $T_\sigma(\phi_{W}^t(x)) = t$ for all $x \in F_{\text{\rm ref}}$.

We take a collection  $\{(U_{\alpha},\varphi^\alpha)\}$ 
with $\varphi^\alpha: U_\alpha \to \R^{2n-1}$
of foliated charts of $H$ that covers $F_{\text{\rm ref}} = T^{-1}(0)$
each element of which is centered at a point in $F_{\text{\rm ref}}$.
We write
$$
\varphi^\alpha = (y^\alpha, t^\alpha) = ( y_1^\alpha, \cdots, y_{2n-2}^\alpha, t^\alpha).
$$
Thanks to the requirement \eqref{eq:yi'H-chart}, we must have
\be\label{eq:talphabeta}
t^\alpha = t^\beta + c_{\alpha\beta}(y^\beta)
\ee
on $U_\alpha \cap U_\beta$. (See \eqref{eq:varphialphabeta}.)

 Let $\sigma$ be the given continuous section.
On each such a chart $(U_{\alpha}, \varphi_\alpha)$ with 
$\varphi_\alpha = (y^\alpha, t^\alpha)$,
the level set $T_\sigma^{-1}(0)$ of the continuous function $T_\sigma$
can be locally represented as
$$
F_{\text{\rm ref}} \cap U_{\alpha} = \{x \in  U_{\alpha} \mid
t^\alpha = f_\alpha(y^\alpha)\}, \quad y^\alpha = (y_1, \cdots, y_{2n-2}) \in V_\alpha \subset \R^{2n-2}
 $$
for some continuous function $f_\alpha = f_\alpha(y_1, \ldots, y_{2n-2})$
that satisfies
\be\label{eq:falpha}
\begin{cases}
 T\circ \varphi_{\alpha\beta}^{-1}(y,t) = t -f_\alpha(y),\\
f_\alpha(0,\cdots, 0) = 0
\end{cases}.
\ee
The transverse  coordinates $(V_\alpha, \psi_\alpha)$
induce a smooth chart on $[U_\alpha] \subset \cN_H$, and the function $f_\alpha$
induces a continuous function $f_\alpha'$ thereon.  Note that the section $\sigma$
can be expressed  in terms of its local representatives 
$\{\sigma_\alpha: = \sigma|_{[U_\alpha]}\}$: we require them to satisfy
$$
\varphi_\alpha(\sigma_\alpha(\ell))  =  (\psi^\alpha(\ell), f_\alpha(\ell))
$$
in terms of the coordinate charts $([U_\alpha], \psi^\alpha)$ of $\cN_H$ and
$(U_\alpha, (y^\alpha,t_\alpha))$ of $H$.
It follows from the above discussion that
to define a global section out of the collection $\{\sigma^\alpha\}$, the collection should
satisfy 
 \be\label{eq:g-compatibility-alphabeta}
 g_\alpha \circ \psi_{\alpha\beta}=  g_\beta + c_{\alpha\beta}
 \ee
by \eqref{eq:talphabeta}.
 
 We summarize the above discussion into the following.

 \begin{lemma}\label{lem:trivializing}
  A section of $\pi_H: H \to \cN_H$ is characterized by the collection
 of maps $\{g_\alpha\}$ and $\{c_{\alpha\beta}\}$ with
 $g_\alpha: \psi_\alpha([U_\alpha]) \to \R$, $ \, c_{\alpha\beta}: \psi_\alpha([U_\alpha] \cap [U_\beta]) \to \R$
 that satisfy
 \eqref{eq:g-compatibility-alphabeta},  or equivalently
 \be\label{eq:g-beta}
 g_\beta
 =  g_\alpha  \circ \psi_{\alpha\beta} - c_{\alpha\beta}
 \ee
 on $\psi_\beta([U_\alpha] \cap [U_\beta])$ and vice versa.
 \end{lemma}
 \begin{proof} For the proof of \eqref{eq:g-beta},
 we apply Lemma \ref{lem:c-alphabeta} to \eqref{eq:g-compatibility-alphabeta}
 and get
 $$
 c_{\beta\alpha} \circ \psi_{\beta\alpha}^{-1}
 = c_{\beta\alpha} \circ \psi_{\alpha\beta} = c_{\beta\beta} - c_{\alpha\beta}= - c_{\alpha\beta}.
$$
 Then we rewrite \eqref{eq:g-compatibility-alphabeta} into
\beastar
 g_\beta
& =  & (g_\alpha  + c_{\beta\alpha}) \circ \psi_{\beta\alpha}^{-1}
= (g_\alpha  + c_{\beta\alpha}) \circ \psi_{\alpha\beta} \\
& = & g_\alpha  \circ \psi_{\alpha\beta}  + c_{\beta\alpha}\circ \psi_{\alpha\beta}\\
& = & g_\alpha \circ \psi_{\alpha\beta}- c_{\alpha\beta}
\eeastar
which finishes the proof.
\end{proof}

 By exponentiating \eqref{eq:g-beta}, we get
 $
 e^{g_\alpha} \circ \psi_{\alpha\beta} = e^{c_{\alpha\beta}} e^{g_\beta}
 $
 which is equivalent to
 \be\label{eq:salpha}
  e^{g_\alpha} \circ \psi_\alpha = e^{c_{\alpha\beta} \circ \psi_\beta} e^{g_\beta} \circ \psi_\beta.
 \ee
If we set  $s_\alpha = e^{g_\alpha}\circ \psi_\alpha$ and 
$
g_{\alpha\beta} = e^{c_{\alpha\beta} \circ \psi_\beta},
$
the equation becomes $s_\alpha = g_{\alpha\beta} s_\beta$ on $[U_\alpha] \cap [U_\beta]$.

\begin{lemma} The collection $\{g_{\alpha\beta}\}$ is a $\R_+$-valued smooth
cocycle.
\end{lemma}
\begin{proof} By definition of $c_{\alpha\beta}$, it is a smooth function.
The equation \eqref{eq:c-alphagamma} is equivalent to
$$
c_{\alpha\gamma} \circ \psi_\gamma = c_{\beta\gamma} \circ \psi_\gamma 
+ c_{\alpha\beta}\circ \psi_\beta.
$$
By exponentiating this equation, we obtain
$$
g_{\alpha\gamma} = g_{\beta\gamma} g_{\alpha\beta} =   g_{\alpha\beta}g_{\beta\gamma}.
$$
Furthermore since $c_{\alpha\alpha} = 0$, we have $g_{\alpha\alpha} = 1$.
This finishes the proof.
\end{proof}

This shows that the collection $\{g_{\alpha\beta}\}$ defines a real oriented \emph{smooth}
line bundle  on $\cN_H$, and $\{s_\alpha\}$ associated to the 
local representatives $\{f_\alpha\}$ of the given section $\sigma$
defines a nowhere vanishing \emph{continuous} section thereof.

\begin{remark} This line bundle can be also described as follows. The presence of 
leaf-generating vector field $W$ equips each leaf with the structure of an
oriented 1 dimensional real affine space. A choice of section of $\pi: H \to \cN_H$
then it  identifies each leaf with the real line $\R$.
Then the bundle is nothing but the tautological line bundle of $\cN_H$.
\end{remark}

We denote this  \emph{smooth} oriented line bundle by $\CL$. Lemma \ref{lem:trivializing}
shows that this collection also provides $\CL$ with
 a trivialzing cover and hence defines a \emph{smooth} trivialization
$$
\CL \to \cN_H \times \R.
$$
 We summarize the above discussion into the following.
 
 \begin{lemma}  Consider the collections $\{g_{\alpha\beta}\}$ and
  $\{s_\alpha\}$ defined by  
   $$
 g_{\alpha\beta} = e^{c_{\alpha\beta}\circ \psi_\beta}, \quad
  s_\alpha = e^{f_\alpha\circ \psi_\alpha}
  $$
of continuous $\R_+$-valued functions respectively. 
Then the collection  $\{s_\alpha\}$ defines  a nowhere vanishing continuous section
 of  the \emph{smooth} oriented line bundle $\CL$. We denote the associated
 global section of $\cL$ by $s_\sigma$.
 \end{lemma}
  
We are now ready to complete the proof of Proposition \ref{prop:smoothing}.

\begin{proof}[Wrap-up of the proof of Proposition \ref{prop:smoothing}]  

 We would like to construct a \emph{smooth} section or the collection $\{g_\alpha\}$
 satisfying \eqref{eq:g-beta}, knowing the existence of this continuous section $\sigma$.
 For this purpose, we have only to find a smooth approximation of the section $s_\sigma$ of
 the line bundle $\CL$, which is a standard process by taking the mollifier smoothing whose
 details is now in order. 
 
 We denote by $s_{\sigma;\alpha}$ the local representative of $s_\sigma$ determined by $e^{f_\alpha} \circ \psi_\alpha$, i.e., we will characterize the section $s_\sigma$ 
 by the collection $\{s_\alpha: [U_\alpha] \to \R \}$ that satisfy
 $$
s_\alpha = g_{\alpha\beta}\,  s_\beta.
$$  
For this purpose, without loss of any generality,
we assume $\psi_\alpha([U_\alpha]) = I^{2n-2}$ with $I = (-1,1)$
for all $\alpha$, and take a family of mollifier $\{\rho_\delta\}_{\delta > 0}$ supported 
in $I^{2n-2}$. We then take the collection $\{s_\alpha\}$  by setting 
$$
s_\alpha = h_\alpha^\delta\circ   \psi_\alpha
$$
for the mollifier smoothing of  the functions $\{e^{f_\alpha}\}$ which are defined by
$$
h_\alpha^\delta = e^{f_\alpha} * \rho_\delta
$$
for all $\alpha$. Here $*$ is the
standard convolution product defined by
$$
a* b(x) : = \int_{\R^{2n-2}} a(x-y) b(y)\, dy
$$
for two real-valued functions $a,  \, b: \R^{2n-2} \to \R$.  Then we take the sum
$$
s^{\text{\rm sm}}: = \sum_\alpha \chi_\alpha s_{\sigma;\alpha}
$$
for a partitions of unity $\{\chi_\alpha\}$ subordinate to $\{[U_\alpha]\}$
which defines a global \emph{smooth} section of $\CL$.

It follows from the general property of the mollifier smoothing 
that  $h^\delta_\alpha \to e^{f_\alpha}$  as $\delta \to 0$ in compact open topology or in $C^0$
topology. This is easy to check (or see \cite{gelfand-shilov} for example).
Therefore $h_\alpha^\delta$ is nowhere vanishing 
for a sufficiently small $\delta = \delta_\alpha > 0$, and so
 we can take the logarithm $g_\alpha = \log h_\alpha^{\delta_\alpha}$ 
so that $h_\alpha^\delta = e^{g_\alpha}$ unambiguously.

Reading back
 the above \emph{explicit} correspondence between a section of $H\to \cN_H$ and a nowhere-vanishing
section of the line bundle $\cL$, we conclude that
the collection $\{g_\alpha \circ \psi_\alpha\}$ associated to $\{[U_\alpha]\}$
represents a smooth section of the projection $\pi: H \to \cN_H$. 
We denote by $\sigma^{\text{\rm sm}}$ the corresponding smooth section.

Now we consider the flow map of the vector field $W$ 
$$
\Phi_H^{\sigma^{\text{\rm sm}}}: \cN_H \times \R \to H
$$
given by $\Phi_H^{\sigma^{\text{\rm sm}}}(\ell,t) = \phi_{W}^t(\sigma^{\text{\rm sm}}(\ell))$,
and define the map
$\Psi: H \to \cN_H \times \R$ to be its inverse
\be\label{eq:smooth-Psi}
\Psi(x) = (\pi_H(x), T_{\sigma^{\text{\rm sm}}}(x)).
\ee
By construction, $\Psi$ now
satisfies all the properties required in Proposition \ref{prop:leaf-space-structure}.
This finally completes the proof of Proposition \ref{prop:leaf-space-structure}.
\end{proof}

This will finish the proof of the diagram \eqref{eq:diagram-intro}
required in the proof of Theorem \ref{thm:equivalence-intro}.

\subsection{Symplectic structure on the leaf space}
\label{subsec:symplectic-structure}

Now  we turn to the construction of symplectic structure.
Using Proposition \ref{prop:leaf-space-structure}, we fix a smooth
section $\sigma^{\text{\rm sm}} : \cN_H \to H$ and write $F: = \Image \sigma^{\text{\rm sm}}$.

When we choose the above used coherent atlas, we can choose them
so that the associated cocycle $\gamma_{\alpha\beta}$
becomes symplectic by requiring the chart $(U_\alpha, \varphi_\alpha)$ also to satisfy
the defining equation
\be\label{eq:reduced-form}
(y^\alpha)^*\omega_0 = \iota_{H}^*\omega, \quad \omega = d\lambda
\ee
of the general coisotropic reduction (see \cite[Theorem 5.3.23]{abraham-marsden} for example)
where $\iota_{H}: H \to M$ is the inclusion map and
$\omega_0$ is the standard symplectic from on $\RR^{2n-2}$. (See also \cite{gotay}, \cite{oh-park}.)
By using such a foliated chart satisfying \eqref{eq:reduced-form},
the associated holonomy cycles define symplectic atlas and so a symplectic structure on $\cN_{H}$,
\emph{when the holonomy is trivial as in our case where we assume the presence of smooth section.}
This will then finish construction of reduced symplectic structures on $\cN_H$.
(We refer to \cite[Section 5]{oh-park} for a detailed discussion on the construction of
transverse symplectic structure for the null foliation of general coisotropic submanifolds.)

An immediate corollary of the above construction of diffeomorphism
$\Psi: H \to \cN_H \times \R$ is that  any Liouville $\sigma$-sector
is a Liouville sector in the sense of \cite{gps}.

\begin{remark}
On the other hand, the converse is almost a tautological statement
in that \cite[Lemma 2.5]{gps} shows that any of their three defining conditions
given in \cite[Definition 2.4]{gps} is equivalent to
the condition
\begin{itemize}
\item There exists
a diffeomorphism $\Psi: H \to F \times \R$ making \eqref{eq:Psi-diagram} commute
\end{itemize}
Once this is in our disposal, $\Psi$ induces a diffeomorphism
$
[\Psi]: \cN_H \to F.
$
Therefore we can choose a continuous section
$
\sigma_{\text{\rm ref}}: \cN_H \to H
$
required for the definition of
$\sigma$-sectorial hypersurface to be
$$
\sigma_{\text{\rm ref}}(\ell): = [\Psi]^{-1}(\ell), \quad \ell \in \cN_H.
$$
\end{remark}

Now we wrap up the proof of Theorem~\ref{thm:equivalence-intro} as the special case
$H = \del M$ of the following theorem. We will postpone the proof of Statement (3) till
the next subsection.

\begin{theorem} \label{thm:equivalence-H} 
Under the above definition of $\sigma$-sectorial hypersurface $H \subset M$,  the following holds:
\begin{enumerate}[(1)]
\item\label{item.equivalence cL manifold-intro} $\cN_{H}$ carries
the structure of Hausdorff smooth manifold  such that $\pi: H \to \cN_{H}$
is a smooth submersion.
\item There exists a smooth section $\sigma^{\text{\rm sm}}$ of $\pi: H \to \cN_H$
which can be $C^0$-approximated to the given continuous section $\sigma$ as close as we want.
\item\label{item. equivalence symplectic structure-intro} 
$\cN_{H}$ carries a canonical symplectic structure denoted by $\omega_{\cN_H}$ as a coisotropic reduction of $H \subset \Mliou$: We set $F: = \text{\rm Image }\sigma^{\text{\rm sm}}$. Then there is a diffeomorphism
$\Psi: H \to F \times \RR$ and a commutative diagram
\be \label{eq:diagram-intro}
\xymatrix{H \ar[d]^\pi \ar[r]^{\Psi} & F \times \RR \ar [d]^{\pi_F}\\
\cN_{H} \ar[r]^{\psi} & F
}
\ee
such that $\pi$ is a smooth map which admits a smooth section $\sigma: \cN_{H} \to H$
for which $\sigma$ satisfies $\sigma^*\omega_\del = \omega_{\cN_{\del M}}$, and $\pi_F$ is the canonical
projection.
\item $(\cN_H,\omega_{\cN_H})$ carries a canonical Liouville one-form $\lambda_{\cN_H}$:
The map $\psi$ is a Liouville diffeomorphism between
$(\cN_H, \lambda_{\cN_H})$ and the $(F, \lambda|_F)$ with the Liouville form
$ \lambda|_F$ on $F$, which is given by $\psi(\ell) = \sigma(\ell)$
for $\ell \in \cN_H$.
\end{enumerate}
\end{theorem}

\subsection{Induced Liouville structure on the leaf space}
\label{subsec:liouville-structure-leaf-space}

Finally we prove Statement (4) of Theorem \ref{thm:equivalence-H} by extracting
some consequences on the above constructed symplectic structure on $\cN_H$
derived from the given property of the characteristic foliation $\cD$ near infinity.
Recall the definitions  $F = \Image \sigma^{\text{\rm sm}}$ and
the smooth flow map $\Phi_H := (\Psi_H^{\text{\rm sm}})^{-1}$
\be\label{eq:PhiH}
\Phi_H:  F  \times [0,\infty) \to H
\ee
where $\Psi_H^{\text{\rm sm}}$ is given in \eqref{eq:smooth-Psi}.
By the convexity hypothesis on $H_\infty$,
we have a contact vector field $\eta$ on $\del_\infty M$ that is transverse to $H_\infty$.

\begin{lemma}\label{lem:cNH-exact} The symplectic manifold $(\cN_H,\omega_{\cN_H})$ is exact.
\end{lemma}
\begin{proof} Note that $F = \Image \sigma^{\text{\rm sm}}$
is a symplectic submanifold of $M$ and
the symplectic form $d\lambda$ induces an exact symplectic form
$d(\iota_F^*\lambda) = \iota_F^*d\lambda $
for the inclusion map
$$
\iota_F: F  \hookrightarrow H \hookrightarrow (M,\lambda).
$$
Therefore it follows from \eqref{eq:reduced-form} and $\pi_F^*\circ \sigma^{\text{\rm sm}}  = id_{\cN_H}$
\beastar
\omega_{\cN_H} & = & (\pi_{F }\circ \sigma^{\text{\rm sm}})^*\omega_{\cN_H} = (\sigma^{\text{\rm sm}} )^*(\pi_{F }^*\omega_{\cN_H})\\
& = & (\sigma^{\text{\rm sm}} )^*(\iota_{F }^*d\lambda) = (\sigma^{\text{\rm sm}} )^*d(\iota_F^*\lambda)
= d((\sigma^{\text{\rm sm}} )^*\iota_F^*\lambda) \\
& = & d((\iota_F \circ \sigma^{\text{\rm sm}} )^*\lambda)\eeastar
which proves exactness of $\omega_{\cN_H}$: Here the third equality follows from
the defining equation \eqref{eq:reduced-form} and the equalities
$$
\pi_{F } = \pi_H \circ \Phi_H , \quad \iota_{F } = \iota_H \circ \Phi_H
$$
with the map $\Phi_H $ given in \eqref{eq:PhiH}.
\end{proof}

This leads us to the following reduced Liouville structure on $\cN_H$.
\begin{defn}[Reduced Liouville structure] We call the primitive $\lambda_{\cN_H}$ of $\omega_{\cN_{\del H}}$ defined as
above the canonical Liouville structure on $(\cN_H,\omega_{\cN_H})$.
\end{defn}

\section{Geometry of transverse coisotropic collections}
\label{sec:sectors-with-corners}

Recall that~\cite{gps2} requires the following properties
 on the boundary strata when studying Liouville sectors with corners:

\newenvironment{sectorial-collection}{
	  \renewcommand*{\theenumi}{(S\arabic{enumi})}
	  \renewcommand*{\labelenumi}{(S\arabic{enumi})}
	  \enumerate
	}{
	  \endenumerate
}
\begin{defn}[Definition 9.2 \& Lemma 9.4 \& Definition 9.14 \cite{gps2}]\label{defn:sectorial-collection}
A \emph{sectorial collection} is a collection of $m$ hypersurfaces $H_1, \ldots, H_m \subset M$, cylindrical near infinity, such that:
\begin{sectorial-collection}
\item\label{item.  transverse intersection sectorial} The $H_i$ transversely intersect,
\item\label{item. coisotropic sectorial} All pairwise intersections $H_i \cap H_j$ are coisotropic, and
\item\label{item. I} There exist functions $I_i: \nbhd(\del M) \to \RR$, linear near infinity, satisfying the following on the characteristic foliations $\cD_i$ of $H_i$:
\be\label{eq:function-Ii}
dI_i|_{\cD_i} \neq 0, \, dI_i|_{\cD_j} = 0 \quad \text{\rm for } i \neq j, \quad \{I_i, I_j\} = 0.
\ee
\end{sectorial-collection}
A Liouville sector $(M,\lambda)$ with corners is a Liouville manifold-with-corners whose codimension one boundary strata form a sectorial collection.
\end{defn}


We will introduce another definition of sectorial collection
by replacing Condition \ref{item. I} in the spirit of Definition \ref{defn:sectorial hypersurface}.

For this purpose, we need some preparations. We start with introducing the following definition

\begin{defn}[Transverse coisotropic collection]\label{defn:transverse coisotropic collection}
Let $(M,\lambda)$ be a Liouville manifold with boundary and corners.
Let $H_1, \ldots, H_m \subset M$ be a collection of hypersurfaces cylindrical near infinity,
that satisfies Conditions \ref{item.  transverse intersection sectorial},
\ref{item. coisotropic sectorial} of Definition \ref{defn:sectorial-collection}.
\end{defn}

In the remaining section, we first study the underlying geometry and
prove a general structure theorem of such a collection.
In the next section, based on the theorem, we will provide an intrinsic characterization
of the sectorial collection and Liouville sectors with corners above
purely in terms of geometry of coisotropic submanifolds. We call
the resulting structure the structure of \emph{Liouville $\sigma$-sectors with corners}.

\subsection{Gotay's coisotropic embedding theorem of presymplectic manifolds}
\label{subsec:neighborhoods}

For a finer study of the neighborhood structure of the sectorial corner $C$,
we first recall below some basic properties of the coisotropic
submanifolds and the coisotropic embedding theorem of Gotay \cite{gotay}.
See also \cite{weinstein-cbms}, \cite{oh-park} for relevant material on the geometry of
coisotropic submanifolds. We will mostly adopt the notations used in \cite{gotay},
\cite[Section 3]{oh-park}.

Let $(Y,\omega_Y)$ be any presymplectic manifold.
The null distribution on $Y$ is the vector bundle
$$
E: = (TY)^{\omega_Y} \subset TY, \, \quad E_y = \ker \omega_Y|_y.
$$
This distribution is integrable since $\omega_Y$ is closed.
We call the corresponding foliation the {\it null foliation} on
$Y$ and denote it by
$$
\cF = \cF_Y.
$$
(Then $E$ is nothing but the total space of
the foliation tangent bundle $T\cF$.)
We now consider the dual bundle $\pi: E^* \to Y$ which is the foliation cotangent bundle
$$
E^* = T^*\cF.
$$
The tangent bundle $TE^*$ of the total space $E^*$ has its restriction to
the zero section $Y \hookrightarrow E^*$; this restriction carries a canonical decomposition
$$
TE^*|_Y \cong TY \oplus E^*.
$$

\begin{example}
A typical example of a presymplectic manifold is given by
$$
(Y,\omega_Y) = (H, \omega_H), \quad \omega_H: = \iota_H^*\omega
$$
arising from any coisotropic submanifold $H \subset^{\iota_H} (X,\omega)$. Then $E = \cD_H$,
the null distribution of $(H,\omega_H)$.
It is easy to check that the isomorphism
$$
TX \to T^*X
$$
maps $TY^\omega$ to the conormal $N^*Y \subset T^*X$, and induces
an isomorphism between $NY = (TX)|_Y/TY$ and $E^*$.
\end{example}

Gotay \cite{gotay} takes a transverse symplectic subbundle $G$ of $TY$ and associates to each  splitting
\be\label{eq:splitting}
\Gamma: \quad TY = G \oplus E, \quad E = T \cF
\ee
the zero section map
$$
\Phi_\Gamma: Y \hookrightarrow T^*\cF = E^*
$$
as a coisotropic embedding with respect to a `canonical' two-form $\omega_{E^*}$ on $E^*$ which restricts to a
symplectic form on a neighborhood of the zero section of $E^*$ such that
$$
\omega_Y = \Phi_\Gamma^*\omega_{E^*}.
$$
\begin{remark}
When $\omega_Y = 0$, Gotay's embedding theorem reduces to the well-known Weinstein's
neighborhood theorem of Lagrangian submanifolds $L$ in which case $E^* = T^*L$ with $Y = L$.
\end{remark}

We now describe the last symplectic form closely following \cite{gotay}.

We denote the aforementioned neighborhood by
$$
V \subset T^*\cF = E^*.
$$
Using the splitting $\Gamma$, which may be regarded as an `Ehresmann connection' of the `fibration'
$$
T \cF \to Y \to \cN_Y,
$$
we can explicitly write down a symplectic form $\omega_{E^*}$ as follows.

First note that as a vector bundle, we have a natural splitting
$$
TE^*|_Y \cong TY \oplus E^* \cong G \oplus E \oplus T^*\cF
$$
on $Y$, which can be extended to a neighborhood $V$ of the zero section $Y \subset E^*$ via the
`connection of the fibration' $T^*\cF \to Y$. (We refer readers to \cite{oh-park} for
a complete discussion on this.)

We denote
$$
p_\Gamma : TY \to T\cF
$$
the (fiberwise) projection to $E=T\cF$ over $Y$ with respect to the splitting
\eqref{eq:splitting}. We have the bundle map
$$
TE^* \stackrel{T\pi}\longrightarrow TY \stackrel{p_\Gamma}
\longrightarrow E
$$
over $Y$.
\begin{defn}[Canonical one-form $\theta_\Gamma$ on $E^*$]
Let $\zeta \in E^*$ and $\xi \in T_\zeta E^*$.
We define the one form $\theta_\Gamma$ on $E^*$ whose value is to be the linear functional
$$
\theta_{\Gamma}|_\zeta \in T_\zeta^*E^*
$$
at $\zeta$ that is determined by its value
\be\label{eq:thetag}
\theta_\Gamma|_{\zeta}(\xi): = \zeta(p_\Gamma \circ T\pi(\xi))
\ee
against $\xi \in T_\zeta(T^*\cF)$.
\end{defn}
(We remark that this is reduced to the canonical Liouville one-form $\theta$ on the cotangent bundle $T^*L$ for the case
of Lagrangian submanifold $L$ in which case $\omega_Y = 0$ and the splitting is trivial and not needed.)

Then we define the closed (indeed exact) two form on $E^* = T^*\cF$ by
$$
- d\theta_\Gamma.
$$
Together with the pull-back form $\pi^*\omega_Y$, we consider the closed two-form $\omega_{E^*,\Gamma}$ defined by
\be\label{eq:omega*}
\omega_{E^*,\Gamma}:= \pi^*\omega_Y - d\theta_\Gamma
\ee
on $E^* = T^*\cF$.
It is easy to see that $\omega_{E^*,\Gamma}$  is non-degenerate in a neighborhood $V \subset E^* $ of the zero
section (See the coordinate expression \cite[Equation (6.6)]{oh-park}
of $d\theta_\Gamma$ and $\omega_V$.)

\begin{defn}[Gotay's symplectic form \cite{gotay}]\label{defn:gotay's-two-form}
We denote the restriction of
$\omega_{E^*,\Gamma}$ to $V$ by $\omega_V$, i.e.,
$$
\omega_V: = (\pi^*\omega_Y - d\theta_\Gamma)|_V.
$$
We call this two-form \emph{Gotay's symplectic form} on $V \subset E^*$.
\end{defn}

The following theorem ends the description of
Gotay's normal form for the neighborhood of a coisotropic submanifold $C \subset (M,\omega)$
of any symplectic manifold $(M,\omega)$ as a neighborhood $V$ of the zero section of
$T^*\cF_C$ of its null foliation $\cF_C$ on $C$ equipped with the symplectic form.

\begin{theorem}[See \cite{gotay,oh-park}]\label{thm:normal-form} Let $Y \subset (X,\omega_X)$ be any coisotropic submanifold.
Fix a splitting $\Gamma$ in \eqref{eq:splitting}. Then
there is a neighborhood $\nbhd(Y):= U \subset X$ and a diffeomorphism
$$
\Phi_\Gamma : U \to V \subset E^*
$$
such that the following hold:
\begin{enumerate}
\item $\omega_X = \Phi_\Gamma^*\omega_{E^*,\Gamma}$ on $U \subset X$.
\item For two different choices, $\Gamma$ and $\Gamma'$, of splitting of $TY$,
the associated two forms $\omega_{E^*,\Gamma}$ and $\omega_{E^*,\Gamma'}$ are
diffeomorphic relative to the zero section $Y \subset E^*$,
on a possibly smaller neighborhood $V' \subset E^*$ of $Y$.
\end{enumerate}
\end{theorem}
\begin{proof} The first statement is proved in \cite{gotay}. Statement (2) is then proved
in \cite[Theorem 10.1]{oh-park}.
\end{proof}

We have the natural projection map
\be\label{eq:projection-to-Y}
\widetilde \pi_Y: \nbhd(Y) \to Y
\ee
defined by
\be\label{eq:piY}
\widetilde \pi_Y := \pi_{E^*}\circ \Phi_\Gamma \circ \iota_Y,
\ee
for the inclusion map $\iota_Y: Y \hookrightarrow \nbhd(Y)=:U \subset X$,
which is induced by restricting the canonical projection $E^* \to Y$ to
the neighborhood $V \subset E^*$ of the zero section $Y$.
In particular, we have
$$
\ker d_x\pi_Y = E_x = \cD_Y|_x.
$$
\subsection{Structure of the null foliations of $\sigma$-sectorial corners}

We apply the discussion in the previous subsection to general transverse coisotropic collection
$$
\{H_1, \cdots, H_m\}.
$$
For any given subset $I \subset \{1, \cdots, m\}$, we denote
$$
H_I = \bigcap_{i \in I} H_i
$$
and $\pi_{H_I}: H_I \to \cN_{H_I}$ be the canonical projection. We also denote
the full intersection by
$$
C= \bigcap_{i=1}^m H_i.
$$
Furthermore, by the  transverse intersection property of
the coisotropic collection, we can choose the collection $\{\sigma_{C,1}, \ldots, \sigma_{C,m}\}$
to have the complete intersection property in that their images form a collection of  transverse intersection.
More precisely, we fix the following choice of smooth sections for a finer study of
the neighborhood structure of further constructions we will perform
\begin{choice}[Choice of sections $\sigma_i: \cN_{H_i} \to H_i$]\label{choice:sigma}
For each $i = 1, \ldots, m$, we choose a smooth section
$$
\sigma_i: \cN_{H_i} \to H_i
$$
for each $i = 1, \ldots, m$. Denote the set of sections $\sigma_i: \cN_{H_i} \to H_i$ by
\be\label{eq:sigma-collection}
\sigma = \{\sigma_1, \ldots, \sigma_m\}.
\ee
\end{choice}

Recall from Section \ref{sec:intrinsic} that for each $i$ a choice of smooth section
$$
\sigma_i: \cN_{H_i} \to H_i
$$
provides the trivialization map
$$
\Psi_i^{\sigma_i}: H_i \to \cN_{H_i} \times \RR, \quad \Psi_i^{\sigma_i}(x) = (\pi_{H_i}(x), t_i^{\sigma_i}(x))
$$
given in \eqref{eq:Psi-diagram}. We choose each $\sigma_i$  to be $\sigma_i = \sigma_{H_i}$
as defined in \eqref{eq:smooth-sigma}
For the given choice of $\sigma = \{\sigma_1, \ldots, \sigma_m\}$, we
collectively write
\be\label{eq:Psi-i-sigma}
\Psi_i^\sigma: = \Psi_i^{\sigma_i}, \quad i = 1, \ldots, m.
\ee

The following theorem is the generalization of Theorem \ref{thm:equivalence-intro}
whose proof also extends the one used in Section \ref{sec:intrinsic} to the case with corners.
The main task for this extension is to establish compatibility of the null foliations
of various coisotropic intersections arising
from taking a sub-collection $I \subset \{1, \ldots, m\}$: This compatibility condition and construction of
relevant strata is in the same
spirit as the combinatorial construction of a toric variety out of its associated fan.
(See \cite{fulton:toric} for example.)

\begin{theorem}\label{thm:Z-projectable} Let $(M,\lambda)$ be a Liouville $\sigma$-sector with corners,
and let $Z$ be the Liouville vector field of $(M,\lambda)$. Let
$$
\sigma = \{\sigma_1, \cdots, \sigma_m\}
$$
be a collection of smooth sections $\sigma_i: \cN_{H_i} \to H_i$ for $i=1, \ldots, m$.
Then the leaf space $\cN_{C}$ carries a canonical structure $\lambda_{\cN_{C}}$ of a Liouville manifold
with boundary and corners.
\end{theorem}

We also define the function $t_i^{C,\sigma}: C \to \RR$ to be the restriction
\be\label{eq:tiC}
t_i^{C,\sigma}= t_i^{\sigma_i}|_C
\ee
where $t_i^{\sigma_i}$ is the function appearing in \eqref{eq:yi'H}.
The collection $\sigma = \{\sigma_i\}$ also induces a surjective map $\Psi_C: C \to \cN_C \times \RR^m$,
\be\label{eq:PsiC}
\Psi_C^\sigma(x): = \left(\pi_C(x), \left(t_1^{C,\sigma}(x),\ldots, t_m^{C,\sigma}(x)\right)\right)
\ee
which is also smooth with respect to the induced smooth structure on $\cN_C$.
(The functions $t_i^{C,\sigma}$ correspond to $t_i$ appearing in \cite[Section 49]{arnold:mechanics}
in the discussion following below.)

\begin{prop}\label{prop:PsiC} There is an $\RR^m$-action on $C$ that is free, proper and discontinuous
and such that
$C$ is foliated by the $\RR^m$-orbits. In particular the map
$$
\Psi_C^\sigma: C \to \cN_C \times \RR^m
$$
is an $\RR^m$-equivariant diffeomorphism with respect to the $\RR^m$-action on $C$ and
that of linear translations on $\RR^m$.
\end{prop}
\begin{proof}

Let $(s_1, \ldots, s_m)$ be the standard coordinates of $\RR^m$.
We set
\be\label{eq:Zi}
Z_i : = (\Psi_C^\sigma)^*\left(\vec 0_{\cN_C} \oplus \frac{\del}{\del s_i}\right).
\ee
Then $Z_i \in \cD_C$, and $[Z_i,Z_j] = 0$ since $[\frac{\del}{\del s_i},\frac{\del}{\del s_j}] = 0$.
On $C$, we also have
$$
t_j^{C,\sigma}(Z_i) = d(s_j \circ \Psi_C^\sigma)\left((\Psi_C^\sigma)^*\left(\vec 0_{\cN_C} \oplus \frac{\del}{\del s_i}\right)\right)
= ds_j(\frac{\del}{\del s_i}) = \delta_{ij}.
$$
In particular $Z_i$ is tangent to all level sets of $t_i^{C,\sigma}$ with $j \neq i$, and is transverse
to the level sets of $t_i^{C,\sigma}$ for each $i$.
%
%
%
%

The so-constructed global frame $\{Z_1, \cdots, Z_m\}$ of $TC$ on $C$ are commuting vector fields.
Therefore we have an $\RR^m$-action on $C$ induced by the flows of commuting vector fields $\{Z_1, \cdots, Z_m\}$.

\begin{lemma}\label{lem:discrete-isotropy} This $\RR^m$-action is also proper and discontinuous.
In particular, its isotropy subgroup is a discrete subgroup of $\RR^m$.
\end{lemma}
\begin{proof}
The Liouville vector field $Z$ is tangent to every $H_i$ near infinity.
Since $Z$ is tangent to $H_i$ for all $i$ near infinity, the flag
$$
H_1 \cap \cdots \cap H_m \subset H_1 \cap \cdots \cap H_{m-1} \subset \cdots \subset H_1
$$
is $Z$-invariant near infinity, and in particular we have
$$
Z \in TC
$$
near infinity of $C$. Since $Z[s] = 1$, $Z$ is also transverse to $s^{-1}(r)$ for all
sufficiently large $r > 0$. Therefore the $\RR^m$-action induces a free $\RR^m/\RR$-action
on the set $\del_\infty C = \del_\infty M \cap C$ of asymptotic Liouville rays tangent to
$C$. Since the latter set is compact, it follows that the $\RR^m/\RR$-action is
proper and discontinuous. Since the flow of $Z$ or the $\RR$-action induced by $Z$
moves the level of $s$ by 1 as time varies by 1, we conclude
 that the $\RR^m$-action on $C$ is proper and discontinuous.

Once the action is proved to be proper and discontinuous, the second statement of the lemma
follows e.g. from the proof in \cite[Section 49, Lemma 3]{arnold:mechanics}, to which we refer.
This finishes the proof.
\end{proof}

With Lemma \ref{lem:discrete-isotropy} in our disposal, the standard argument in the construction of action-angle coordinates
proves that each orbit of the $\RR^m$-action is homeomorphic to $\RR^{n_1} \times T^{n_2}$ for some $n_1, \, n_2$ with
$n_1 + n_2 = n$.  (See \cite[Section 49, Lemma 3]{arnold:mechanics} and its proof.)

Now we immediately conclude the following

\begin{cor}\label{cor:contractible-fiber} Suppose $\pi_C: C \to \cN_C$ has contractible fibers. Then
\begin{enumerate}
\item The $\RR^m$-action is free and
its fiber is naturally diffeomorphic to $\RR^m$, i.e., it is a principle $\RR^m$ bundle over $\cN_C$.
\item The map $\Psi$ is an $\RR^m$-equivariant diffeomorphism with respect to the translations of
$\RR^m$.
\end{enumerate}
\end{cor}

The inverse of $\Psi_C^\sigma$ denoted by
\be\label{eq:PhiC}
\Phi_C^\sigma: \cN_C \times \RR^m \to C
\ee
is also easy to explicitly write down as follows. First we note
$$
t_i^{C,\sigma}(\sigma_{C,i}(\pi_C(x))) = 0
$$
for all $i=1,\ldots, m$ by the definitions of $\sigma_{C,i}$ and $t_i^{C,\sigma}$.
Now let a point
$$
(\ell, (t_1, \ldots, t_m)) \in \cN_C \times \RR^m
$$
be given. Then there is a unique point $x \in C$ satisfying
\be\label{eq:piC}
\begin{cases}
\pi_C(x) = \ell\\
 x = \bigcap_{i=1}^n (t_i^{C,\sigma})^{-1}(t_i).
 \end{cases}
\ee
(See \eqref{eq:tiC} for the definition of $t_i^{C,\sigma}$ and Proposition \ref{prop:W}
for the definition of the vector field $Z_i'$ respectively.)
Then we define $\Phi_C^\sigma(\ell, (t_1, \ldots, t_m))$ to be this unique point. It is easy to
check from definition that $\Phi_C^\sigma$ is indeed the inverse of $\Psi_C^\sigma$. This finishes the proof of
Proposition \ref{prop:PsiC}.
\end{proof}

By applying the above proof and Proposition \ref{prop:PsiC} to any
sub-collection $I \subset \{1, \cdots, m\}$ including the full collection itself,
we also obtain the following stronger form of Theorem \ref{thm:Z-projectable}
\begin{theorem}\label{thm:stronger-Z-projectable}
Let $I \subset \{1, \cdots, m\}$ be any sub-collection, and define
$$
H_I = \bigcap_{i \in I} H_i.
$$
Assume $\pi_{H_I}: H_I \to \cN_{H_I}$ has contractible fibers.
Let $\lambda_{\cN_{H_I}}$ be the canonical induced Liouville form as before. Then
the following hold:
\begin{enumerate}
\item There is an $\RR^{|I|}$-action on $H_I$ that is free, proper and discontinuous
and such that
$H_I$ is foliated by the $\RR^{|I|}$-orbits. In particular the map
$$
\Psi_{H_I}^\sigma: H_I \to \cN_{H_I} \times \RR^{|I|}
$$
is an $\RR^{|I|}$-equivariant diffeomorphism with respect to the $\RR^{|I|}$-action on $H_I$ and
that of linear translations on $\RR^{|I|}$.
\item The leaf space $\cN_{H_I}$
carries a canonical structure of Liouville manifold with boundary and corners.
\end{enumerate}
\end{theorem}

By applying the above to the full collection $C = H_{\{1,\ldots, m\}}$,
we have finished the proof of Theorem \ref{thm:Z-projectable}.

\subsection{Compatibility of null foliations of transverse coisotropic intersections}

Let $C_\delta = C$ as in the previous section and let $\{\sigma_1, \cdots, \sigma_m\}$
a collection of sections $\sigma_i:\cN_{H_i} \to H_i$ made in Choice \ref{choice:sigma}.
For each subset $I \subset \{1, \cdots, m\}$, we have the following section
$$
\sigma_I: \cN_{H_I} \to H_I
$$
defined by
\be\label{eq:sigma-I}
\sigma_I([\ell]): = \Phi_{H_I}^\sigma([\ell], (0,\cdots, 0)) = (\Psi_{H_I}^\sigma)^{-1}([\ell], (0,\cdots, 0))
\ee
for the diffeomorphism $\Phi_{H_I}$ given in \eqref{eq:PhiC} applied to $C = H_I$.

Then for each pair of subsets $I \subset J $ of $\{1,\cdots, n\}$, we have $H_J \subset H_I$ and the map
$$
\psi_{JI}^\sigma: \cN_{H_J} \to \cN_{H_I}
$$
given by
\be\label{eq:psi-IJ}
\psi_{JI}^\sigma([\ell]): = \pi_{\cN_{H_I}}(\Phi_{H_J}^\sigma([\ell], (0,\cdots,0)).
\ee
%

In particular consider the cases with $I = \{i\}$, $J = \{i, j\}$ and $K = \{i,j,k\}$.
Then we prove the following compatibility of the collection of maps $\psi_{IJ}$:
For each $i \neq j$, we consider the maps
$$
\psi_{ij,i}^\sigma: \cN_{H_i\cap H_j} \to \cN_{H_i}
$$
defined by $\psi_{ij,i}^\sigma := \psi_{\{ij\}\{i\}}$, and the inclusion maps
$$
\iota_{ij,i}: H_i \cap H_j \to H_i.
$$
\begin{prop}\label{prop:sectorial-diagram}
Let $\{H_1,\ldots,H_m\}$ be a collection of hypersurfaces satisfying
only ~\ref{item.  transverse intersection sectorial} and~\ref{item. coisotropic sectorial}.
Then the maps $\psi_{ij,i}^\sigma$ satisfy the following:
\begin{enumerate}
\item They are embeddings.
\item The diagram
\be\label{eq:sectorial-diagram}
\xymatrix{ H_i \cap H_j \ar[r]^{\iota_{ij,i}} \ar[d]_{\pi_{ij}} & H_i \ar[d]^{\pi_i}\\
\cN_{H_i\cap H_j} \ar[r]^{\psi_{ij,i}^\sigma} & \cN_{H_i}
}
\ee
commutes for all pairs $1 \leq i, \, j \leq n$.
\item The diagrams are compatible in the sense that we have
$$
\psi_{ij,i}^\sigma \circ \psi_{ijk,ij}^\sigma = \psi_{ijk,i}^\sigma.
$$
for all triples $1 \leq i,\, j, \, k \leq n$.
\end{enumerate}
\end{prop}
\begin{proof}
We first show that the map $\psi_{ij,i}^\sigma$ is an embedding.
Let $\ell_1, \, \ell_2$ be two leaves of the null-foliation of $H_i \cap H_j$
such that
$$
\ell_1 \cap H_i = \ell_2 \cap H_i.
$$
By definition of leaves, we have only to show that $\ell_i \cap \ell_j \neq \emptyset$.

Let $x \in H_i$ be in the above two common intersection which obviously implies
$$
x \in \ell_1 \cap \ell_2 \subset H_i \cap H_j.
$$
This proves $\psi_{ij,i}^\sigma$ is a one-one map.
Then smoothness and the embedding property of $\psi_{ij,i}^\sigma$
follow from the definition of smooth structures given on the leaf spaces.

For the commutativity, we first note
\be\label{eq:psi-iji=}
\psi_{ij,i}^\sigma (\pi_{ij}(x)) = \pi_i(\Phi_{ij}^\sigma((\pi_{ij}(x), 0,0)))
\ee
by the definition of the maps $\psi_{ij,i}^\sigma$.
But by the definition \eqref{eq:PhiC} of $\Phi_{ij}^\sigma$, the point
$$
y := \Phi_{ij}^\sigma((\pi_{ij}(x), 0,0))
$$
is the intersection point
$$
y \in \Image \sigma_i \cap \Image \sigma_j.
$$
Since $x \in H_i \cap H_j$, we can express it as
$$
x = \Phi_{ij}^\sigma(\pi_{ij}(x), t_1, t_2)
$$
for some $t_1, \, t_2 \in \RR$. In other words, it is obtained from $y$ by the
characteristic flows of $H_i$ and $H_j$ by definition of $\Phi_{ij}^\sigma$ in \eqref{eq:PhiC}. In particular, we have
$$
\pi_i(\iota_{ij,i}(x)) = \pi_i(y).
$$
On the other hand, the definition of the null foliation of $\cN_{H_i}$ implies
\be\label{eq:pi-i=}
\pi_i(y) = \psi_{ij,i}^\sigma (\pi_{ij}(x))
\ee
for all $x \in H_i \cap H_j$. Combining the last two equalities and commutativity of the
diagram $\pi_i \circ \iota_{ij,i} = \psi_{ij,i}^\sigma \circ \pi_{ij}$,
we have proved the commutativity of \eqref{eq:sectorial-diagram}.

Finally we show that $\psi_{ij,i}^\sigma$ is a symplectic map. Consider the pull-back
$$
\omega_{ij}^\sigma: = (\psi_{ij,i}^\sigma)^*(\omega_{\cN_{H_i}}).
$$
We will show that $\omega_{ij}^\sigma$ satisfies the defining property
$$
\pi_{H_i \cap H_j}^*\omega_{ij}^\sigma = \iota_{H_i \cap H_j}^*\omega, \quad \omega = d\lambda
$$
of the reduced form on $\cN_{H_i \cap H_j}$ under the coisotropic reduction on
the coisotropic submanifolds $H_i \cap H_j \subset M$.  We compute
\beastar
\pi_{H_i \cap H_j}^*\omega_{ij}^\sigma & = & \pi_{H_i \cap H_j}^*(\psi_{ij,i}^\sigma)^*(\omega_{\cN_{H_i}})\\
& = & (\psi_{ij,i}^\sigma \circ \pi_{H_i \cap H_j})^*(\omega_{\cN_{H_i}}) \\
& = & (\pi_{H_i}\circ \iota_{H_i\cap H_j,H_i})^* \omega_{\cN_{H_i}} \\
& = & (\iota_{H_i\cap H_j,H_i})^*(\pi_{H_i}^*\omega_{\cN_{H_i}}) \\
& = &  (\iota_{H_i\cap H_j,H_i})^*(\iota_{H_i}^*\omega)
= \iota_{H_i\cap H_j}^*\omega
\eeastar
where we use the defining condition of the reduced form $\omega_{\cN_{H_i}}$ of $\omega_{\del H_i}$
$$
\pi_{H_i}^*\omega_{\cN_{H_i}} = \iota_{H_i}^*\omega
$$
for the penultimate equality. Therefore we have proved
$$
\pi_{H_i \cap H_j}^*\omega_{ij}^\sigma = \iota_{H_i \cap H_j}^*\omega.
$$
This shows that the form $\omega_{ij}^\sigma$ satisfies the defining equation \eqref{eq:reduced-form}
of the reduced form $\omega_{H_i \cap H_j}$. Then by the uniqueness of the reduced form,
we have derived
$$
\omega_{ij}^\sigma = \omega_{H_i \cap H_j}.
$$
This proves $(\psi_{ij,i}^\sigma)^*\omega_{H_i} = \omega_{H_i \cap H_j}$, which finishes the proof of
Statement (1).

Statement (2) also follows by a similar argument this time from the naturality of the \emph{coisotropic reduction
by stages}: Consider $H_i, \, H_j, \, H_k$ in the given coisotropic collection and consider the two flags
\be\label{eq:flag1}
H_i \cap H_j \cap H_k \subset H_i \cap H_j \subset H_i
\ee
and
\be\label{eq:flag2}
H_i \cap H_j \cap H_k \subset  H_i.
\ee
The composition $\psi_{ij,i}^\sigma \circ \psi_{ijk,ij}^\sigma$ is the map obtained by
the coisotropic reductions in two stages and $\psi_{ijk,i}^\sigma$ is the one obtained by
the one stage reduction performed in the proof of Statement 1 with the replacement of the pair
$(H_i \cap H_j, H_i)$ by $(H_i\cap H_j \cap H_k, H_i)$. Then by the naturality of
the coisotropic reduction, we have proved Statement 2. This finishes the proof of the proposition.
\end{proof}

The following is an immediate corollary of the above proposition and its proof.
(See Remark \ref{rem:stratawise-presymplectic} for the relevant remark on the stratified
presymplectic manifolds.)

\begin{cor} The collection of maps
$$
\{\psi_{I}\}_{I \subset \{1, \ldots, m\}}
$$
are compatible in that the leaf space $\cN_{H_I}$ carries the structure of symplectic manifold
with boundary and corners.
\end{cor}

\section{Liouville $\sigma$-sectors and canonical splitting data}

Let  $\{H_1, \cdots, H_m\}$ be a transverse coisotropic collection
as in Definition \ref{defn:transverse coisotropic collection}.
We denote their intersection by
$$
C = H_1 \cap \cdots \cap H_m
$$
as before,
which is a  coisotropic submanifold of codimension $m$ associated thereto.

\subsection{Definition of Liouville $\sigma$-sectors with corners}

Denote
by $\iota_{CH_i}: C \to H_i$ the inclusion map, and $\sigma = \{\sigma_1, \ldots, \sigma_m\}$ be
the collection as before.
This induces the diagram
\be\label{eq:sectorial-diagram-C}
\xymatrix{C \ar[r]^{\iota_{CH_i}} \ar[d]_{\pi_{C}} & H_i \ar[d]^{\pi_i}\\
\cN_{C} \ar[r]^{\psi_{CH_i}^\sigma} & \cN_{H_i}
}
\ee
for all $i$ which are compatible in the sense of Statement (2) of Proposition \ref{prop:sectorial-diagram}.
In fact, we have
\be\label{eq:cNC}
\cD_C = \cD_{H_1}|_C + \cD_{H_2}|_C + \cdots + \cD_{H_m}|_C
\ee
which canonically induces the leaf map $\psi_{CH_i}^\sigma$ in the bottom arrow
that makes the diagram commute.

With these preparations, we are finally ready to provide the sectional characterization of
Liouville sectors with corners.

\begin{defn}[Liouville $\sigma$-sectors with corners]\label{defn:intrinsic-corners}
Let $M$ be a manifold with corners equipped with
a Liouville one-form $\lambda$. We call $(M,\lambda)$ a \emph{Liouville $\sigma$-sector with corners}
if at each sectorial corner $\delta$ of $\del M$, the corner can be expressed as
$$
C_\delta := H_{\delta,1} \cap \cdots \cap H_{\delta,m}
$$
for a transverse coisotropic collection
$$
\{H_{\delta,1}, \cdots, H_{\delta,m}\}
$$
of $\sigma$-sectorial hypersurfaces such that fibers of the map
$$
\pi_{C_\delta}: C_\delta \to \cN_{C_\delta}
$$
are contractible. We call such a corner $C_\delta$ a \emph{$\sigma$-sectorial corner of codimension $m$}.
\end{defn}

In the remaining section, we will derive the consequences of this definition.
The following monoidal property is apparent from our definition of Liouville $\sigma$-sectors.

\begin{prop}\label{prop:product} The set of Liouville $\sigma$-sectors with corners
is a monoid.
\end{prop}
\begin{proof} Let $M_1$ and $M_2$ be Liouville $\sigma$-sectors with corners.
For the simplicity of exposition, we assume $M_i$ without corners. The general case
follows by similar arguments.

Recall that the set of manifolds with corners naturally forms a monoid and so
$M_1 \times M_2$ is a manifold with corners with its boundary and corners given by
\beastar
\del(M_1 \times M_2) & = & \del M_1 \times M_2 \cup M_1 \times \del M_2= : H_1 \cup H_2 \\
C_2 & = & \del M_1 \times \del M_2 = H_1 \cap H_2.
\eeastar
Obviously both are coisotropic submanifolds of codimension 1 and 2 respectively,
and $C_2$ is a manifold without boundary. The required transversality hypothesis 
trivially holds.

It remains to show the property of their null foliations. 
Let $\cF_i$ be the associated characteristic foliation of $\del M_i$.
Then the characteristic distribution of $\del(M_1\times M_2)$ is given by
$$
\cD_1 \oplus \{0\}, \quad \{0\} \oplus \cD_2
$$
on $\del M_1 \times M_2$ (resp. $M_1 \times \del M_2$) whose
leaf $\cF_{\del(M_1 \times M_2),(x,y)}$ is 
 given by $(\cF_1)_x \times \{y\}$ (resp. $\{x\} \times (\cF_2)_y$) which 
 is obviously trivial.  Furthermore the required section $\sigma_{\del(M_1\times M_2)}$
is given by the map $\sigma_1 \times \id_Y$ and $\id_X \times \sigma_2$, respectively.
 
 On the corner $C_2$, its null distribution is given by
 $$
 \cD_{M_1 \times M_2;(x,y)} = \cD_{1,x} \oplus \cD_{2,y}
 $$
 whose associated leaves are $\cF_{1,x} \times \cF_{2,x}$ which are 
 clearly contractible if so are $\cF_{1,x}$ and $\cF_{2,x}$. The product of 
Liouville sectors with corners can be treated similarly whose details are omitted.
 This finishes the proof.
 \end{proof}

\subsection{Integrable systems and canonical splitting data}
\label{subsec:integrable}

By applying Theorem \ref{thm:normal-form} to the coisotropic submanifold $C$,
we will obtain a neighborhood $\nbhd(C) \subset M$ and the projection
$$
\widetilde \pi_C: \nbhd(C) \to C.
$$
\begin{choice}[Splitting $\Gamma_C^\sigma$]\label{choice:GC}
Let $\sigma=\{\sigma_1,\cdots, \sigma_m\}$ be a choice of sections of
transverse coisotropic collection $\{H_1, \cdots, H_m\}$.
Then we associate the splitting
\be\label{eq:C-splitting}
\Gamma = \Gamma_C^\sigma : \quad TC = G_C^\sigma \oplus \cD_C
\ee
thereto given by the transverse symplectic subspace
\be\label{eq:GC}
G_C^\sigma |_x : = (d\Psi_C^\sigma|_x)^{-1}(T_{\pi_C(x)}\cN_C \oplus \{0\}_{\RR^m}).
\ee
\end{choice}
Applying Theorem \ref{thm:normal-form}, we obtain a diffeomorphism
$$
\Psi_\Gamma^\sigma: \nbhd(C) \to V \subset E^* = T^*\cF_C
$$
where $\cF_C$ is the null foliation of $C$.
Furthermore the pushforward of symplectic form $d\lambda$ on $U$
is given by the canonical Gotay's symplectic form on $V \subset E^*$
$$
(\Psi_\Gamma^\sigma)_*(d\lambda) = \pi^*\omega_C - d\theta_\Gamma
$$
for the presymplectic form $\omega_C = \iota_C^*(d\lambda)$ on $C$.
(See Theorem \ref{thm:normal-form}.)

Note that we have
$$
\cD_C|_x = \text{\rm span}_\RR\{Z_1(x), \cdots, Z_m(x)\}
$$
by definition of $Z_i$ above.
Therefore the aforementioned $\RR^m$-action induces an
$\RR^m$-equivariant bundle isomorphism
$$
\cD_C \cong C \times \RR^m
$$
over $C$.
(This isomorphism does not depend on the choice of $\sigma$ but depends only
on the Liouville geometry of $\nbhd(C \cap \del_\infty M)$. )

Then we have made the aforementioned splitting
$
TC = G_C^\sigma \oplus \cD_C
$
given in \eqref{eq:GC} $\RR^m$-equivariant. In other words,
for each group element $\mathsf t = (t_1, \dots, t_m) \in \RR^m$, we have the equality
$$
d \mathsf t (G_x^\sigma ) = G_{\mathsf t\cdot x}^\sigma.
$$
For a fixed $\alpha> 0$, we put
\be\label{eq:definition-Ii}
I_i^\sigma = \pm e^{\alpha t_i^{C,\sigma}}
\ee
which then satisfies $dI_i^\sigma(Z_i) = \alpha\, I_i^\sigma$ on $C$.

Noting that the induced $\RR^m$-action on $TC$ preserves the subbundle
$$
T \cF_C = \cD_C \subset TC,
$$
the canonically induced action on $T^*C$ also preserves the subbundle
$$
\cD_C^\perp \subset T^*C
$$
for which we have the isomorphism
$$
T^*\cF \cong \cD_C^{\perp}.
$$
Therefore the $\RR^m$-action on $C$ can be lifted to $T^*\cF$ which is the restriction of
the canonical induced action on $T^*C$ of the one on $C$.
\begin{lemma} We can
lift the vector fields $Z_j$'s to $Z'_j$ on $T^*\cF$ which are the generators of
the induced $\RR^m$-action such that
\begin{enumerate}
\item  $Z'_j|_C = Z_j$,
\item The collection $\{Z'_j\}$ are commuting.
\end{enumerate}
\end{lemma}
\begin{proof} Let $\phi_{Z_j}^t$ be the flow of $Z_j$ on $C$.
Since the $\RR^m$-action is abelian, the vector fields $Z_j$'s
are pairwise commuting. Then the lifting $Z'_j$ is
nothing but the vector field generating the isotopy of canonical derivative maps
$$
((d\phi_{Z_j}^t)^*)^{-1}: T^*C \to T^*C
$$
on $T^*C$. Since the flows $\phi_{Z_j}^t$ are commuting, their derivatives are also commuting.
Then obviously their dual flows $((d\phi_{Z_j}^t)^*)^{-1}$ on $T^*C$ are also commuting
and hence $Z'_j$'s too. The first condition also follows since any derivative maps zero vector to
a zero vector. This finishes the proof.
\end{proof}

We now define
$$
\widetilde I_i^\sigma = I_i^\sigma \circ \pi_{T^*\cF}.
$$
Then $\{d\widetilde I_1^\sigma,\cdots, d\widetilde I_m^\sigma\}$ are
linearly independent on a neighborhood of the zero section of $T^*\cF$
if we choose the neighborhood small enough. This is because $\{dI_1^\sigma, \ldots, dI_m^\sigma\}$
are linearly independent on $C$.
By suitably adjusting the parametrization $t_i^{C,\sigma}$ of the $\RR^m$-action, we can make
the equation
\be\label{eq:tilde-dIiZj}
d\widetilde I_i^\sigma(Z'_j) = \alpha \delta_{ij} \widetilde I_i^\sigma
\ee
hold.

\emph{This is precisely the situation of
completely integrable system to which we can apply the standard construction of
action-angle coordinates.} (See  \cite[Section 49]{arnold:mechanics} for example.)
Therefore, regarding $\{\widetilde I_1^\sigma, \ldots, \widetilde I_m^\sigma\}$ as the
(fiberwise) angle coordinates,
we can find a unique choice of (fiberwise) action coordinates
$$
\{\widetilde R_1^\sigma, \cdots, \widetilde R_m^\sigma\}
$$
over $\cN_C$ satisfying
$$
\{\widetilde R_i^\sigma, \widetilde I_j^\sigma\} = \delta_{ij}, \quad \widetilde R_i^\sigma
\circ \Phi_C^\sigma|_{H_j} = 0
$$
on a neighborhood $V \subset T^*\cF_C$ of the zero section $0_{T^*\cF_C} \cong C$.
Now we define the pull-back functions
$$
R_i^\sigma: = \widetilde R_i^\sigma \circ \Phi_C^\sigma, \quad I_j^\sigma: = \widetilde I_j^\sigma \circ \Phi_C^\sigma
$$
on $U= \nbhd(C)$. We also pull-back the vector fields $Z'_j$ to $\nbhd(C)$ by $\Phi_C^\sigma$ and denote them
by $Z_j$.
(Note that the notations $I_j^\sigma$ and $Z_j$ are consistent in that their restrictions to
$C$ are nothing but the above already given $I_j^\sigma$ or $Z_j$ respectively on $C$.)
Furthermore, we have the relationship
$$
Z_j = X_{R_j^\sigma}.
$$
(See the definition \eqref{eq:Zi} of $Z_i$ on $C$.)

Then we have
$$
\{R_i^\sigma,R_j^\sigma\} = \omega(X_{R_i^\sigma},X_{R_j^\sigma})= \omega(Z_i,Z_j) = 0
$$
on $\nbhd(C)$. Since
$
Z_i : = X_{R_i^\sigma},
$
we have
\be\label{eq:moment-map}
Z_i \rfloor \omega = dR_i^\sigma
\ee
on $U = \nbhd^Z(C)$.  This is precisely the defining equation of the moment map
$\phi_{G,C}^\sigma: \nbhd(C) \to \frak g^* \cong \RR^m$ with $G = \RR^m$
given by
$$
\phi_{G,C}^\sigma(x) = (R_1^\sigma(x), \cdots, R_m^\sigma(x))
$$
for the above $G=\RR^m$-action.
Recall that the hypersurfaces $H_i$ are $Z$-invariant near infinity.
Therefore we can choose the neighborhood $\nbhd(C)$ so that it is $Z$-invariant near infinity.
Then by the requirement put on the Liouville vector field $Z$
which is pointing outward along $\del M$,  we can choose the whole neighborhood $\nbhd(C)$
$Z$-invariant. Together with the normalization condition of $R_i$'s
$$
R_i^\sigma|_{H_i} = \widetilde R_i \circ \Phi_C^\sigma|_{H_i} = 0,
$$
it also implies $R_i^\sigma \geq 0$ on $\nbhd(C)$ for all $i$.
We now take the neighborhood $U \subset M$ to be this $Z$-invariant neighborhood
$$
U = \nbhd^Z(C).
$$

The content of the above discussion can be summarized into the following
intrinsic derivation of the splitting data.
\begin{theorem}[$\sigma$-Splitting data]\label{thm:splitting-data-corners}
Let $C \subset \del M$ be a sectorial corner of codimension $n$ associated to
the sectorial coisotropic collection $\{H_1, \ldots, H_m\}$ on $\del M$. Then for each
choice
$$
\sigma = \{\sigma_1, \cdots, \sigma_m\}
$$
of sections $\sigma_i: \cN_{H_i} \to H_i$ of $\pi_{H_i}$ for $i = 1, \cdots, n$,
there is a diffeomorphism
$$
\Psi_C^\sigma: \nbhd^Z(C) \cap \del M \to F \times \RR^m
$$
and
$$
\psi_C^\sigma: \cN_{C} \to F_C^\sigma
$$
such that
\begin{enumerate}
\item $F_C^\sigma = \text{\rm Image } \sigma_1 \cap \cdots \cap \text{\rm Image } \sigma_m$,
\item $(\Psi_C^\sigma)_*\omega_\del = \pi_F^*\omega_F$,
\item The following diagram
\be\label{eq:split-coisotropic-diagram}
\xymatrix{ \del M|_{C} \ar[r]^{\Psi_C^\sigma} \ar[d]^{\pi_{\del M}} & F_C^\sigma \times \RR^m \ar[d]_{\pi_{F_C^\sigma}}\\
\cN_{\del M|_{C}} \ar[r]_{\psi_C^\sigma} & F_C^\sigma. &
}
\ee
commutes for the map
$$
\Psi_C^\sigma = (\sigma_C \circ \pi_{F_C},(I_1^\sigma, \cdots, I_m)).
$$
\item The $G$-action with $G = \RR^m$ has the moment map $\phi_{G,C}^\sigma: \nbhd_\epsilon^Z(C) \to \RR^m$
given by
$$
\phi_{G,C}^\sigma = (R_1^\sigma, \cdots, R_m^\sigma)
$$
for a collection of Poisson-commuting $R_i$'s satisfying the simultaneous normalization condition
$$
R_i^\sigma|_{H_i} = 0, \quad R_i^\sigma \geq 0
$$
for all $i$ on $\nbhd^Z(C)$.
\item The map $\nbhd(\del M) \to F_C \times \C^m_{\text{\rm Re} \geq 0}$ is given by the formula
\be\label{eq:tilde-PsiC}
\widetilde \Psi_C^\sigma(x) = \left(\sigma_C(\pi_{F_C}(x)), R_1^\sigma(x) + \sqrt{-1} I_1^\sigma(x), \ldots, R_m^\sigma(x) + \sqrt{-1}I_m^\sigma(x)\right).
\ee
such that
\be\label{eq:tilde-Psi*-omega}
(\widetilde \Psi_C^\sigma)_*\omega = \pi_F^*\omega_{F_C} + \sum_{i=1}^m dR_i^\sigma \wedge dI_i^\sigma.
\ee
\end{enumerate}
We call these data a \emph{$\sigma$-splitting data} of $\nbhd(C)$ associated to
the choice $\sigma = \{\sigma_1, \cdots, \sigma_m\}$ of sections $\sigma_i: \cN_{H_i} \to H_i$.
\end{theorem}

We also gather the following consequences of the above discussion separately. The first one,
in particular,  states that Proposition \ref{prop:equivalence-intro} still holds for
the Liouville $\sigma$-sectors with corners.

\begin{theorem}\label{thm:equivalence-sector}
\begin{enumerate}
\item Each Liouville $\sigma$-sector with corners is
a Liouville sector in the sense of Definition \ref{defn:sectorial-collection}.
\item The leaf space $\cN_{C_\delta}$ carries a natural structure of manifold with corners
at each sectorial corner $\delta$
such that the map $\pi_{C_\delta}: \del M \to \cN_{C_\delta}$ is a morphism of manifolds with corners.
\end{enumerate}
\end{theorem}
\begin{proof}
We have already constructed a diffeomorphism
$$
\Psi_\delta^\sigma : \del M|_{C_\delta} \to F_\delta^\sigma \times \RR^m
$$
given by
$$
\Psi_\delta^\sigma (x) = (\pi_{F_\delta^\sigma}(x), I_1^\sigma(x), \ldots, I_m^\sigma(x)).
$$
Each $I_i^\sigma$ defined on $\del M$ is extended to the function
$\widetilde I_i^\sigma \circ \Phi_{C_\delta}^\sigma$ on
a symplectic neighborhood $U_\delta: = \nbhd^Z(C_\delta) \subset M$
via Gotay's coisotropic neighborhood map
$$
\Phi_{C_\delta}^\sigma: \nbhd(C_\delta) \hookrightarrow T^*\cF_{C_\delta}
$$
where the function $\widetilde I_i^\sigma$ is canonically defined on a neighborhood
$$
V \subset E^* =T^*\cF_{C_\delta}.
$$
This diffeomorphism $\Phi_{C_\delta}$ onto $V_\delta \subset T^*\cF$ also induces a
splitting of the tangent bundle $TC_\delta$
$$
\Gamma_{C_\delta}^\sigma: \quad TC_\delta = G_\delta^\sigma \oplus T\cF_{C_\delta} = G_\delta^\sigma \oplus \cD_{C_\delta}
$$
such that $G_\delta^\sigma$ is a transverse symplectic subbundle of $TC_\delta$
$$
G_\delta^\sigma|_x: = d\Psi^{-1}\left(T_{\pi_{F_\delta^\sigma}(x)}F_\delta^\sigma \oplus \{0\}\right)
$$
at each $x \in C_\delta$. Theorem \ref{thm:splitting-data-corners} then finishes the construction of
the data laid out in Definition \ref{defn:sectorial-collection}.

For the proof of Statement (2), we start with the observation that for each $H = H_i$
the canonical smooth structure on $\cN_H$ carries the
natural structure of a manifold with boundary and corners through a choice of smooth section
made in Choice \ref{choice:sigma},
whose existence relies on the defining hypothesis of $\sigma$-sectorial hypersurfaces
that the projection map $\pi_H: H \to \cN_H$ admits a continuous section.
For each choice of smooth section,
by the same construction as in Subsection \ref{subsec:liouville-structure-leaf-space},
we have a symplectic structure $(\cN_H,\omega_{\cN_H})$, and
a smooth map $\sigma_\infty: \cN_H \to \del_\infty M$ which is a symplectic diffeomorphism
onto the convex hypersurface $F_\infty$ of the contact manifold $(\del_\infty M,\xi)$.
For two different choices of splittings, the resulting structures are diffeomorphic.

Finally it remains to verify the property of $\cN_C$ carrying the structure of
 the Liouville manifolds with corners.
But this immediately follows from the compatibility result, Proposition \ref{prop:sectorial-diagram}:
The moment map $\phi_{G,\delta}^\sigma: \nbhd^Z(C_\delta) \to \RR^m_+$ provides
local description of the codimension $k$-corner of $\cN_{C_\delta}$.
This finishes the proof.
\end{proof}

\section{Triviality of characteristic folitation implies convexity at infinity}
\label{sec:GPS-question}

As an application of our arguments used to derive the canonical splitting data,
we can now provide the affirmative answer to a
question raised by Ganatra-Pardon-Shende in \cite{gps}.

\begin{theorem}[Question 2.6 \cite{gps}]\label{thm:GPS-question}
Suppose $(M,\lambda)$ is a Liouville manifold-with-boundary that satisfies the following:
\begin{enumerate}
\item Its Liouville vector field $Z$ is tangent to $\del M$ near infinity.
\item There is a diffeomorphism $\del M = F \times \RR$ sending the characteristic foliation to the foliation by
leaves $\RR \times \{p\}$.
\end{enumerate}
Then $\del_\infty M \cap \del M$ is convex in $\del_\infty M$, and hence $M$ is a Liouville sector,
\end{theorem}

The proof will be divided into three parts: we first examine the presymplectic geometry
component of the proof, and then explain how the triviality of characteristic foliation
simplifies Gotay's normal form theorem and finally
 combine the discussions with that of the Liouville geometry.

\subsection{Presymplectic geometry of $\del M$}

Denote by  $\iota_{\del M}: \del M \to M$ the inclusion map. Then
the one-form $\lambda_\del: = \iota_{\del M}^*\lambda$
induces the structure of presymplectic manifold
$$
(\del M, d\lambda_\del).
$$
By definition, $\cD_{\del M} = \ker d\lambda_\del$. Denote by $\Psi: \del M \to F \times \R$
the diffeomorphism entering in Condition (2) of the hypothesis. We denote by
\be\label{eq:projections}
 \pi_F: F \times \R \to F, \quad v: F \times \R \to \R
\ee
the natural projections to $F$ and $\R$ respectively.

Then the hypothesis implies that we have
a commutative diagram
\be\label{eq:hypothesis-diagram}
\xymatrix{\del M \ar[d]^{\pi_{\del M}} \ar[r]^\Psi & F \times \R \ar[d]^{\pi_F} \\
\cN_{\del M} \ar[r]^{\psi} & F
}
\ee
where $\psi:= [\Psi]: \cN_{\del M} \to F$ the obvious quotient map, which becomes a
diffeomorphism. In particular, Condition (2) implies that the foliation is a fibration
and the induced smooth structure $\cN_{\del M}$ from the presymplectic
structure is nothing but the pull-back of that of $F$. Furthermore we can take  the pull-back
\be\label{eq:leaf-generatingX} 
X = \Psi_*\left(\frac{\del}{\del v}\right)
\ee
as the leaf generating vector field.
Obviously the map $\sigma: \cN_{\del M} \to \del M$ defined by
\be\label{eq:section-sigma}
\sigma(\ell): = \Psi^{-1}(\psi(\ell),0)
\ee
defines a continuous section of $\pi_{\del M}: \del M \to \cN_{\del M}$, 
one of the defining data of Liouville $\sigma$-sectors. This section is in fact already smooth
with respect to the aforementioned smooth structure equipped with $\cN_{\del M}$.

Next, by Condition (1), we have
$$
\del_\infty M \cap \del M = \del_\infty(\del M).
$$
Therefore it remains to show convexity of $\del_\infty M \cap \del M$ in $\del_\infty M$, i.e., that there exists
a contact vector field defined on $\nbhd(\del_\infty M \cap \del M) \subset \del_\infty M$
that is transverse to the hypersurface $\del(\del_\infty M)$.
We denote the reduced symplectic form on $\cN_{\del M}$ of the presymplectic form 
$d\lambda_{\cN_{\del M}}$ by $\omega_{\cN_{\del M}}$

Next we prove

\begin{lemma}\label{lem:X} Suppose that $Z$ is tangent to $\del M$ outside a compact subset 
$K \subset  \del M$. Consider the pull-back $\lambda_\del:= \iota_{\del M}^*\lambda$ whose differential
$d\lambda_\del$ is  a presymplectic form on $\del M$.
Let $X$ be a vector field  tangent to $\ker \iota_{\del M}^*\lambda = \cD_{\del M}$
on $\del M \setminus K$. Then we have
$\cL_X \lambda_\del = 0$ thereon.
\end{lemma}
\begin{proof}
Since $X$ spans the characteristic distribution of $(\del M,\lambda_\del)$, we have
$$
X \rfloor d\lambda_\del = 0
$$
on $\del M$. On the other hand, since $Z$ is tangent to $\del M\setminus K$
and $X \in \ker (\omega_\del = d\lambda_\del)$, we also have
\be\label{eq:lambdadelX=0}
\lambda_\del(X)= d\lambda_\del(Z,X) = 0
\ee
where the first equality follows by definition of Liouville vector field $Z$. Therefore on
$\del M \setminus K$, we compute
$$
\cL_X \lambda_\del = (d(X \rfloor \lambda) + X \rfloor d\lambda)|_{\del M} = 0
$$
which finishes the proof.
\end{proof}

We push-forward the one-form $\lambda_\del$ on $\del M$ to
$F \times \R$ by $\Psi$, and write
$$
\lambda^{\text{\rm pre}}: = \Psi_*(\lambda_\del)
$$
on $F \times \RR$. As we mentioned in \eqref{eq:leaf-generatingX}, 
we have $\Psi_*X = \frac{\del}{\del v}$.
We have the natural one-form $\lambda_{\cN_{\del M}}$ on $\cN_{\del M}$ 
 induced from $\lambda_{\del}$
characterized by the equation
\be\label{eq:lambda=lambdaF+df}
\lambda_\del = \pi_{\del M}^*\lambda_{\cN_{\del M}}
\ee
which holds at infinity.

\begin{lemma}\label{lem:constant-C} 
Suppose that $Z$ is tangent to $\del M$ on $\del M \setminus K$ for a compact 
subset $K \subset \del M$.
Then there exists a sufficiently large constant $C = C(K) > 0$ such that
$\cL_{\frac{\del}{\del t}} \lambda^{\text{\rm pre}} = 0$
on $F \times \{\log |v| \geq C\}$ and so
$$
\lambda^{\text{\rm pre}} = \pi_F^*\lambda_F
$$
for some one-form $\lambda_F$ on $F\times \{\log |v| \geq C\}$, where $\pi_F: F \times \RR \to F$ is the projection.
In particular we have $\frac{\del}{\del t} \rfloor \lambda^{\text{\rm pre}} = 0$ thereon.
\end{lemma}
\begin{proof}
Since $K$ is compact and $\Psi$ is continous, $\Psi(K)$ is compact and so
there exists a sufficiently large $C' > 1$ such that
$$
\Psi(K) \subset F \times (-C',C').
$$
In particular, we have
$$
F \times \{ \log |v| \geq C\} \subset \Psi(\del M \setminus K)
$$
for a sufficiently large $C := \log C' > 0$.  Then Lemma \ref{lem:X} applied to $F \times \{ \log |v| \geq C\}$
implies
$$
\cL_{\frac{\del}{\del v}} \lambda^{\text{\rm pre}} = 0,
$$
 i.e., $ \lambda^{\text{\rm pre}}$ is $\frac{\del}{\del v}$-invariant and hence
there exists a one-form $\lambda_F$ on $F$ such that $\pi_F^*\lambda_F = \lambda^{\text{\rm pre}}$
on $F \times \{ \log |v| \geq C\}$. The second statement follows from this or directly follows from \eqref{eq:lambdadelX=0}. This finishes the proof.
\end{proof}

For the simplicity of notation, we write 
$$
Y: = F \times \RR
$$
for the resulting presymplectic manifold $(Y,  d\lambda^{\text{\rm pre}})$.
In this case, we have natural identification $\cN_Y = F$, and
the reduced symplectic form on $\cN_Y = F$ is determined by the equation
$$
\psi_*(\omega_{\cN_{\del M}}) = d\lambda_F: = \omega_F
$$
where $\lambda_F$ depends on the behavior of Liouville vector field 
against the characteristic foliation of $Y$ near infinity of $Y$.

\subsection{The standing hypothesis of tameness of $M$}

Recalling that the basic geometric 
assumption put on $M$ is that $M$ is \emph{tame} or of \emph{bounded geometry}. 
We refer to \cite{sikorav:tame} for the standard definition thereof, and
\cite{choi-oh:quasiisometry} for the precise definition of tameness of $J$ relative to $\omega$,
which requires to assume $C^3$-tameness to ensure contractibility of 
tame almost complex structures $J$ on general noncompact symplectic manifolds.
(For the purpose of the present paper, $C^2$-tameness might be enough but we will
assume $C^3$-tameness for extra safety.)
This notion of tameness of a noncompact sympletic manifold $(M,\omega)$ is largely 
about the Riemanninan geometric behavior of the associated tame metrics of the form
$g$ that is quasi-isometric to those of the form $\omega(\cdot, J \cdot)$. This is needed
mainly for the geometric analysis of pseudoholomorphic curves on noncompact manifolds.

More precisely,  we assume that there is a cylindrical metric $g$ in  
$$
\nbhd (\del_\infty M) := [N, \infty) \times \del_\infty M
$$
for some sufficiently large $N> 0$,  so that
\be\label{eq:gcyl}
g_{\text{\rm cyl}}= g_{\del_\infty M} + ds^2, \quad Z = \frac{\del}{\del s}
\ee
and $\omega = g(J\cdot, \cdot)$ with almost complex structure tame to $\omega$
and $g$ has bounded curvature and injectivity radius $\delta = \delta_g > 0$. 
We also assume $g_{\del_\infty M}$ is also $C^3$ tame as a Riemannian manifold 
$(\del_\infty M, g_{\del_\infty M})$ if  $\del_\infty M$   itself  is noncompact.

In terms of the metric $g_{\text{\rm cyl}}$, $Z$ is parallel with respect to
the Levi-Civita connection of $g_{\text{\rm cyl}}$ and $|Z| =1$ 
on $\nbhd(\del_\infty M)$.
Furthermore we have $\omega = d\lambda$, and
by the definition of Liouville one-form, it is uniquely determined by the equation
$$
\lambda = Z \rfloor \omega = \frac{\del}{\del s} \rfloor \omega
$$
from the given exact $\omega$ and $Z$.  It also shows that $|\lambda| \equiv 1$ and  
\be\label{eq:C1-norm-lambda}
\|\nabla \lambda\|_{C^2}, \,\,  \|\nabla J\|_{C^2} \leq C_1
\ee
for some constant $C_1 > 0$ by the assumption of $C^3$-tameness of $M$.

Condition Theorem \ref{thm:GPS-question} (1) and the $C^3$-tameness hypothesis give rise to following

\begin{lemma}\label{lem:collar} Suppose $Z$ is tangent to $\del M$ on $\{s \geq N\}$. Then 
there exists a collar neighborhood $U = \nbhd(\del M)$
of $\del M \subset M$, a sufficiently large constant $C' > 0$ and a proper embedding 
$$
\Phi : F \times (-\delta, 0 ] \times \{ |v| \geq C'\} \to U
$$
 such that
\begin{enumerate}
\item $\Image \Phi \subset U$, and $\Phi (F \times \{0\} \times \{  |v| \geq C'\} )\subset \del M \cap \{s \geq N\}$,
\item  $\Phi$ is  a $C^2$ quasi-isometry onto its image. More precisely, we have
 $$
 \|\Phi\|_{C^2},  \|\Phi^{-1}\|_{C^2} < C_2.
 $$
 \end{enumerate}
\end{lemma}
\begin{proof} Using the fact that $F$ is a Liouville manifold, we decompose 
$$
F = F_0 \cup (F \setminus F_0)
$$ 
so that $F_0$ is compact and that
\be\label{eq:F-decompose}
F \setminus F_0 \cong [0, \infty) \times \del_\infty F, \quad \del_\infty F = \del_\infty M \cap \del M.
\ee
Since $Z$ is tangent to $\del M$ at infinity, we may choose $F_0$ sufficiently large and re-choose $s$ so that 
it satisfies
\be\label{eq:s-at-infty}
s = \log |v|
\ee
near $\del M\cap \del_\infty M$, more specifically on 
$$
\left\{x = (y,u,v) \in M \mid  s(x) \geq N, \, y \in  F \setminus F_0, \, u \in (-\delta , 0], \,  |v| \geq e^N \right\} 
\supset \del_\infty M \cap \del M
$$
and that $Z = \frac{\del}{\del s}$ is a Killing field of $g_{\text{\rm cyl}}$ of \eqref{eq:gcyl}
for a sufficiently large $N > 0$ and sufficiently small $\delta > 0$.

 Then it follows from this adjustment and 
compactness of $F_0$ that we can choose a sufficiently large $C' > 0$ so that
$$
\phi_Z^a \left(\{|v| \geq C'\}, \{|v| \geq C'\} \cap \del M \right) 
\subset \left(\{s \geq N\}, \{s \geq N\} \cap  N_{-\delta < u \leq 0}(\del M)\right) 
$$
for all $a \geq 0$. (In fact, we may choose $C' = e^N + C''$ for some constant $C''> 0$ depending only on $F_0$ and 
$Z$.)

On $\{|v| > C'\}$, we put a metric on 
\be\label{eq:product-metric}
g_F + du^2 + ds^2, \quad s : =  \log|v|, \quad |v| > 1
\ee
which is isometric to \eqref{eq:gcyl} on $\nbhd_\delta(\del M) \cap \{|v| \geq C'\}$.
Now we denote by $X$ the normalized positive leaf generating vector field such that
$X$ with $\|X\|_g = 1$ which is a multiple of $\frac{\del}{\del s}$. In the aforementioned coordinates 
$(u,v) = (u, \pm \log s)$,
we may have $X = v \frac{\del}{\del v}$ everywhere on $\del M \cap \{|v| > C'\}$. 
Consider the inward unit normal
vector $\vec \nu$ which coincides with $- \frac{\del}{\del u}$ of $\del M \cap \{|v| > C'\}$. 
Then we define a map
\be\label{eq:Phi}
\Phi: F \times (-\delta, 0] \times   \{|v| > C'\}  \to M
\ee
given by
$$
\Phi(y,u,v): = \exp^{\del M}_{(y,v)} (u\vec \nu_{(y,v)}).
$$
Then $\Phi|_{F \times \{0\}  \times \{|v| > C'\}}$ is the restriction to $\del M$ 
of $\Phi$. Furthermore by definition of $\Phi$
there exists  some $\delta > 0$ and sufficietly large $C' > 0$ such that 
 \begin{itemize}
 \item the map
\be\label{eq:Phi-V}
\Phi: F \times  (-\delta, 0] \times \{|v| \geq C'\} \to \nbhd(\del_\infty M)
\ee
defines a proper embedding, and
\item its inverse   is an isometry on $(N_g^\delta(\del M) \cap \{|v| \geq C' \}, g)$ 
mapping to
$$
(F \times (-\delta, 0] \times  \{|v| > C'\}, g_F + du^2 + ds^2) 
$$
for some $\delta> 0$ and sufficiently large constant $C' > 0$ where $N_g^\delta(\del M)$ is the $\delta$-neighborhood of $\del M$ with respect to the metric \eqref{eq:product-metric}.
\end{itemize}
The map also extends  the inclusion map
$$
\del M\cap \del_\infty M \hookrightarrow  \del_\infty M.
$$
It follows from the $C^3$-tameness hypothesis, compactness of $F_0$ and
the above adjustment of the radial function $s$ that there exists constant $C_2 > 0$ such that
 $$
 \|\Phi\|_{C^2},  \|\Phi^{-1}\|_{C^2} < C_2
 $$
 i.e., $\Phi$ is a $C^2$ quasi-isometry.
\end{proof}

\subsection{Symplectic thickening of $F \times \R$}

Now let us assume $\del M = F \times \R$ and $F = \del_\infty M \cap \del M$.
We also recall that $F$ itself canonically becomes a Liouville manifold (without boundary).
(See Subsection \ref{subsec:liouville-structure-leaf-space}.)
Therefore $F$ is the interior of the ideal completion 
$$
W = F \sqcup \del_\infty F
$$
which is Liouville isomorphic to a Liouville manifold with cylindrical
in the sense of Giroux \cite{giroux}. (See  Appendix \ref{sec:giroux}.
In the point of view of Definition \ref{defn:giroux} $F = M$ and $W = \overline M$ therein.)

Now we will prove a refinement of Gotay's normal form theorem for the 
\emph{noncompact} presymplectic manifold $(Y, d\lambda^{\text{\rm pre}})$.
In fact, under the present circumstance $\del M \cong F \times \R$ and $C^3$-tameness of $M$, i.e., when
there is a presymplectic diffeomorphism $\Psi: \del M \to F \times \R$ as in the hypothesis
of Theorem \ref{thm:GPS-question}, the map $\Psi$
can be directly thickened to a map $\widetilde \Psi$ utilizing the presence of uniform collar neighborhood constructed  
by the embedding $\Phi$ given in \eqref{eq:Phi} so that $\widetilde \Psi|_{\del M} = \Psi$ 
$$
\omega_V|_{F \times \{0\} \times \R} = (\widetilde \Psi|_{\del M})^*d\lambda
$$
\emph{after re-choosing $\Psi$, if necessary.}

Recall, from the assumption, that there is a diffeomorphism
$\Psi: \del M \to F \times \R$ such that it maps sending the characteristic distribution of
$\del M$ to that of $F \times \R$ which is given by $\{y\} \times \R$.  Because we also regard 
the presymplectic manifold $Y= F \times \R$
as the boundary $\del M$ of tame symplectic manifold $M$, requiring  the $C^3$-tameness hypothesis
 is a natural  continuation of bounded geometry
so that we can arrange the leaf generating vector field of $\del M$
$$
X: = \Phi_*\frac{\del}{\del v}
$$
have its $C^1$-norm $\|X\|_{C^2}$ bounded. Such a requirement has been already
used in the proof of Lemma \ref{lem:collar} .
Then we choose the aforementioned diffeomorphism $\Psi$ to coincide with
\be\label{eq:Psi-final}
\Psi = ( \Phi|_{F \times \{0\} \times \{|v| > C'\}})^{-1}
\ee
on the image $\Phi(F \times \{0\} \times \{|v| > C'\})
\subset \nbhd(\del_\infty M)$ of  the map $\Phi$ given in \eqref{eq:Phi}.

\begin{prop}[Proposition \ref{prop:normal-form}]\label{prop:normal-form}
Let $u + \sqrt{-1} v$ be the standard coordinates of $\C$ satisfying $v = t\circ \text{\rm pr}$.
Put
$$
R = u \circ \pi_\C \circ \widetilde \Psi, \quad I = v \circ \pi_\C\circ \widetilde \Psi
$$
on $F \times \C$. 
Then there are neighborhoods $U$ of $\del M \cong F \times \R$ and $V = F \times (-\delta, 0] \times \R$ of 
$F \times \{0\} \times \R \subset F \times \C$
for some $\delta > 0$, and a deformation of $\Psi$, still denoted by $\Psi$,
which extends to a diffeomorphism pair 
$$
(\widetilde \Psi,\Psi): (U,\del M) \to (V, F \times \{0\} \times \R)
$$ 
satisfying
\be\label{eq:omegaV1}
\widetilde \Psi_*\lambda = \widetilde \pi_F^*\lambda_F  -I\,  dR, \quad 
\widetilde \Psi_*(Z) = Z_F \oplus I \frac{\del}{\del I}
\ee
on $\{I > C'\} \cap V'$ for a sufficiently large $C> 0$ where $Z_F$ is the Liouville vector field of
the Liouville manifold $F$.
In particular we have $F\cong \del M \cap \del_\infty M = \del(\del_\infty M)$, which is convex in $\del_\infty M$.
\end{prop} 
\begin{proof} We consider  the inclusion map
$$
(F\times \{|I |\geq C' \}, d\lambda^{\text{\rm pre}}) \hookrightarrow (F \times \C,
\widetilde \pi_F^*\omega_F + dR \wedge dI)
$$
is a coisotropic embedding of a presymplectic manifold $F \times \{|I| \geq C' \} \subset \del M$. 
We write
\be\label{eq:omegaV}
\omega_V: = \widetilde \pi_F^*\omega_F + dR \wedge dI, \quad 
\lambda_V : =  \widetilde \pi_F^*\lambda_F - I\,  dR
\ee
We pull-back the two-form $\omega = d\lambda$ by the map $\Phi$ and write
$$
\omega_V': =\Phi^*\omega =\Phi^*d\lambda.
$$
Then we have $\omega_V' = \omega_V$ on $F \times \{0\} \times \{|I| \geq C'\} 
\subset F \times R \times \R$ with
\be\label{eq:H-C}
V \cap \{s \geq N\} \supset F \times (-\delta, 0] \times \{|I| \geq C'\} =: H.
\ee
Since $\omega_F = d\lambda_F$, $(\Phi^{-1})_*d\lambda = d\lambda_V$, we have 
$$
d\lambda_V =  \omega_V = \widetilde \pi_F^*d\lambda_F + dR \wedge dI
$$
and hence
$$
d((\Phi^{-1})_*\lambda - \pi_F^*\lambda_F - I dR) = 0.
$$
Since the choice of $\sigma$ made above implies
$$
\pi_F^*\lambda_F =(\Phi^{-1})_*\lambda_\del = \iota_H^*\lambda_V
$$
we have $\iota_H^*(\lambda_V - \pi_F^*\lambda_F - IdR) = 0$ on $\{R=0\}$
recalling $V \hookrightarrow
F \times \C$ is a codimension zero embedding.
In particular the form $\lambda_V - \pi_F^*\lambda_F - I dR$ is
exact on any neighborhood $V$ of $\{R = 0\}$ which deformation retracts to $\{R=0\}$.
Therefore we can write
$$
(\Phi^{-1})_*\lambda - \lambda_V = dh_V
$$
on such a neighborhood $V$ for some smooth function $h_V: V \to \RR$, i.e.,
\be\label{eq:pre-lambdaV}
(\Phi^{-1})_*\lambda = \pi_F^*\lambda_F - I dR + dh_V
\ee
thereon. Since $(\Phi^{-1})_*\lambda = \lambda_V$ on $H$ and $\Phi^{-1}$ is a $C^2$ quasi-isometry,
there exists a constant $C_2 > 0$ such
\be\label{eq:C2}
\|dh_V\|_{C^0} < C_2.
\ee
Since $Z$ is assumed to be tangent to $H$ near infinity, we have
$$
\lambda(X)= d\lambda\left(Z,X\right) = 0.
$$
Obviously we also have $\lambda_V(\frac{\del}{\del I}) =  (\pi_F^*\lambda_F - I dR)(\frac{\del}{\del I}) = 0$.
Therefore we have derived
$$
\frac{\del h_V}{\del I}\Big|_{R=0} = 0
$$
by evaluating \eqref{eq:pre-lambdaV} against $\frac{\del}{\del I}$.

Under this circumstance, the following deformation lemma is 
a generalization of the one proved for the
Liouville manifolds in \cite{oh:sectorial} to the case of Liouville sectors.
For readers' convenience, we give a full proof in Appendix \ref{sec:stability-sectors}
where a more precise statement is also given.

\begin{lemma}[Theorem \ref{thm:stability-sectors}; Compare with Theorem 9.2 \cite{oh:sectorial}]
\label{lem:stability-sectors}
Consider the family
$$
\kappa \mapsto \lambda_\kappa = \lambda_v + \kappa\,  dh_V, \quad \kappa \in [0,1].
$$
Then there exists a diffeomorphism $\varphi_t$ such that 
$$
\varphi_t^*\lambda_t = \lambda_V
$$
with $\supp \varphi_t \subset \supp dh_V$. In particular $\varphi_t|_{\{R = 0\}} = \id$.
\end{lemma}

Now we define $\widetilde \Psi = \varphi_1 \circ (\Phi^{-1})$.
Then we have
$$
\widetilde \Psi_*\lambda = \widetilde \pi^*\lambda_F - R\, dI
$$
for a sufficiently large  constant $C', \, N > 0$.
We then consider the model Liouville sector $(V,\omega_V)$
on $V := F \times (-\delta, 0] \times \R \subset F \times \C$ that is given by 
\be\label{eq:lambdaV}
\lambda_V =  \widetilde \pi_F^*\lambda_F - R\, dI
\ee
and let $Z_V = Z_F + I\, \frac{\del}{\del I}$ be its associated Liouville vector field.
We will compare  these with the pair $\widetilde \Psi_*\lambda$ and $\widetilde \Psi_*Z$.

For this purpose, on $V \subset F \times \R \times \R$, we 
 decompose the vector field $\widetilde \Psi_*Z$  into
\be\label{eq:ZV}
\widetilde \Psi_*Z = X_F + a \frac{\del}{\del R} + b \frac{\del}{\del I}
\ee
for some coefficient functions $a = a(y,R,I), \, b = b(y,R,I)$ for $(y,R,I) \in F \times \C$
in terms of the splitting $TV = TF \oplus T\C$. We compute
$$
\widetilde \Psi_*Z  \rfloor d\lambda_V = X_F \rfloor \widetilde \pi_F^*\omega_F  + a\, dI - b\, dR.
$$
Substituting this and \eqref{eq:lambdaV} into the equation
\be\label{eq:Z-defining1} 
\widetilde \Psi^*Z \rfloor d\lambda_V = \lambda_V
\ee
we obtain
\be\label{eq:Z-defining2}
X_F  \rfloor \widetilde \pi_F^*\omega_F + a\, dI - b\, dR = \widetilde \pi_F^* \lambda_F - I\, dR
\ee
on $(F \times \{|I| \geq C\}) \cap \{s \geq N\} \supset \del_\infty M \cap \del M$.
Comparing the two sides, we have derived
$$
X_F = Z_F, \quad a = 0,  \quad b = I
$$
thereon. We summarize the above discussion into the following 

\begin{lemma} On $F \times \{|I| \geq C'\} \cap \{s \geq N\} \subset V$.
we have
$$
\widetilde \Psi_*Z = Z_F + I\, \frac{\del}{\del I}.
$$
\end{lemma}

In summary, we have constructed a map
$\widetilde \Psi: \nbhd(\del_\infty M \cap \del M) \to F \times \C$
which is a diffeomorphism onto its image, where $F \times \C$
is the Liouville sector equipped with the model structure given by
$$
\lambda_V = \widetilde \pi_F^*\lambda_F - I\, dR, \quad Z_V = Z_F +  I \frac{\del}{\del I}
$$
on $\nbhd(\del_\infty M \cap \del M)$ for which the convexity of $\del_\infty M \cap \del M \hookrightarrow \del_\infty M$ 
is verified by the contact vector field induced by the Hamiltonian vector field $X_I$.

This finishes the proof of Theorem \ref{thm:GPS-question}.
\end{proof}

\section{Structure of Liouville $\sigma$-sectors and their automorphism groups}
\label{sec:automorphism-group}

Our definition of Liouville $\sigma$-sectors with corners enables us to
give a natural notion of automorphisms which is the same as the
case without boundary. 

We first recall the following well-known definition of automorphisms of 
Liouville manifold (without boundary)

\begin{defn}\label{defn:geometric-structure} Let $(M,\lambda)$ be
an Liouville manifold without boundary. We call a diffeomorphism $\phi: M \to M$
a Liouville automorphism if $\phi$ satisfies
$$
\phi^*\lambda = \lambda + df
$$
for a compactly supported function $f: M \to \RR$.
We denote by $\aut(M)$ the set of automorphisms of $(M,\lambda)$.
\end{defn}

Now we would like
extend this definition of automorphisms to the case of Liouville $\sigma$-sectors.
The extension is not completely obvious because not every defining condition involving the 
the \emph{presymplectic geometry} is manifestly preserved  under the action of Liouville diffeomorphisms,
especially for the case of Liouville $\sigma$-sectors with corners. 
(In our opinion, the same applies to the original definition of Liouville sectors with corners from
\cite{gps2} in a different way.)

For this purpose, we need some preparations by examining the universal geometric structures
inherent on the boundary $\del M$ of a Liouville manifold with boundary and corners.

\subsection{Some presymplectic geometry of $\del M$}

We start with the observation that $(\del M, \omega_{\del M})$
carries the structure of \emph{presymplectic manifolds} as usual
for any coisotropic submanifold mentioned as before.
We first introduce automorphisms of presymplectic manifolds
$(Y,\omega)$ in general context.

\begin{defn}\label{defn:presymplectic-morphism} Let $(Y,\omega)$ and
$(Y^\prime, \omega^\prime)$ be two presymplectic manifolds. A diffeomorphism $\phi: Y\to Y'$ is called
{\it presymplectic} if $\phi^*\omega^\prime = \omega$.
We denote by ${\cP}Symp(Y,\omega)$ the set of presymplectic diffeomorphisms.
\end{defn}
(We refer to \cite{oh-park} for some detailed discussion on the geometry of presymplectic manifolds
and their automorphisms and their application to the deformation problem of coisotropic submanifolds.)

Then we note that any diffeomorphism $\phi: (M,\del M) \to (M,\del M)$ satisfying
\be\label{eq:liouville-diffeomorphism}
\phi^*\lambda = \lambda + df
\ee
for some function $f$, \emph{not necessarily compactly supported}, induces a presymplectic diffeomorphism
$$
\phi_\del: = \phi|_{\del M}
$$
on $\del M$ equipped with the presymplectic form
$$
\omega_{\del} := d\lambda_\del, \quad \lambda_\del := \iota^*\lambda
$$
for the inclusion map $\iota: \del M \to M$.
\begin{lemma}\label{lem:phi-preserving-cD} The presymplectic diffeomorphism $\phi_\del: \del M \to \del M$ preserves
the characteristic foliation of $\del M$.
\end{lemma}
\begin{proof}
We have
$$
\cD_{\del M} = \ker \omega_\del.
$$
Since any Liouville automorphism $\phi$ of $(M, \del M)$ satisfies \eqref{eq:liouville-diffeomorphism},
we have
$$
\phi_\del^*\omega_\del = \omega_\del.
$$
Therefore we have
$$
\phi_*(\cD_{\del M}) = \cD_{\del M}
$$
which finishes the proof.
\end{proof}

In fact, for the current case of our interest $Y = \del M$, the presymplectic form $\omega_\del$
is exact in that
$$
\omega_\del = d\lambda_\del, \quad \lambda_\del: = \iota^*\lambda.
$$
Furthermore \eqref{eq:liouville-diffeomorphism} implies that $\phi$
actually restricts to an exact presymplectic diffeomorphism
$$
\phi_\del: (\del M, \omega_\del) \to (\del M,\omega_\del)
$$
on $\del M$ in that
$$
\phi_\del^*\lambda_\del -\lambda_\del = d h, \quad h = f\circ \iota
$$
where the function $h: \del X \to \RR$ is not necessarily compactly supported.

We have a natural restriction map
\be\label{eq:restriction-phi}
\aut(M,\lambda) \to {\cP} \text{\rm Symp}(\del M, \omega_\del); \quad \phi \mapsto \phi_\del.
\ee

\begin{defn}[Pre-Liouville automorphism group $\aut(\del M,\lambda_\del)$] We call
a diffeomorphism $\phi: (\del M,\lambda_\del) \to (\del M,\lambda_\del)$ a
\emph{pre-Liouville diffeomorphism} if the form $\phi^*\lambda_\del -\lambda_\del$ is exact.
We say $\phi$ is a \emph{pre-Liouville automorphism} if   it satisfies
$$
\phi^*\lambda_\del= \lambda_\del + dh
$$
for a compactly supported function $h: \del M \to \RR$. We denote by $\aut(\del M,\lambda_\del)$
the set of pre-Liouville automorphisms  of $(\del M,\lambda_\del)$.
\end{defn}

The following is an immediate consequence of the definition.

\begin{cor}\label{lem:phi-del} The restriction map \eqref{eq:restriction-phi} induces
a canonical group homomorphism
$$
\aut(M,\lambda) \to \aut(\del M,\lambda_\del).
$$
\end{cor}

We recall that $\del M$ carries a canonical transverse symplectic structure
arising from the presymplectic form $d\lambda_\del$. (See \cite[Sectioon 4]{oh-park}.)
\begin{prop}\label{prop:phi-intertwine}
The induced pre-Liouville automorphism
$\phi_\del:= \phi|_{\del M} : \del M \to \del M$ descends to a (stratawise) symplectic diffeomorphism
$$
\phi_{\cN_{\del M}}: \cN_{\del M} \to \cN_{\del M}
$$
and satisfies
$$
\pi_{\del M} \circ \phi_\del = \phi_{\cN_{\del M}} \circ \pi_{\del M}
$$
when we regard both $\del M$ and $\cN_{\del M}$ as manifolds with corners.
\end{prop}

\subsection{Automorphism group of Liouville $\sigma$-sectors}

Now we are ready give the geometric structure of \emph{Liouville $\sigma$-sectors}.

\begin{defn}[Structure of Liouville $\sigma$-sectors]\label{defn:structure}
We say two Liouville $\sigma$-sectors $(M,\lambda)$ and $(M', \lambda')$ are
isomorphic, it there exists a diffeomorphism $\psi: M \to M'$ (as a manifold with corners) such that
$\psi^*\lambda' = \lambda + df$ for some compactly supported function $f: M \to \R$.
A \emph{structure of Liouville $\sigma$-sectors} is defined to be an isomorphism class of Liouville $\sigma$-sectors.
\end{defn}

With this definition of the structure of Liouville $\sigma$-sectors in our disposal,
the following is an easy consequence of the definition and Proposition \ref{prop:phi-intertwine},
which shows that the definition of an automorphism of a Liouville sector $(M,\lambda)$
is in the same form as the case of Liouville manifold given by
the defining equation
$$
\psi^*\lambda = \lambda +df
$$
for some compactly supported function $f: M \to \R$, except that $\psi$ is a self diffeomorphism
of $M$ as a stratified manifold and the equality of the above equation as in the sense of
Remark \ref{rem:stratawise-presymplectic}.

\begin{theorem}[Automorphism group]\label{thm:automorphism-group}
Let $(M,\lambda)$ be a Liouville $\sigma$-sector.
Suppose a diffeomorphism $\psi: M \to M$ satisfies
\be\label{eq:defining-psi}
\psi^*\lambda = \lambda + df
\ee
for some compactly supported function $f: M \to \R$. Then $\psi$ is an automorphism of the \emph{structure
of Liouville $\sigma$-sectors}.
\end{theorem}
\begin{proof}
We first discuss how the action of diffeomorphisms $\psi$ satisfying
$
\psi^*\lambda = \lambda + df
$
affects the structure of Liouville $\sigma$-sectors, when the function $f$ is compactly supported.
In particular it implies
\begin{itemize}
\item $\psi^*d\lambda = d\lambda$,
\item $\psi^*\lambda = \lambda$ near infinity.
\end{itemize}
Then $\psi$ restricts to a presymplectic diffeomorphism $\psi_\del: \del M \to \del M$
which is also pre-Liouville, i.e., satisfies
$$
(\psi|_{\del M})^*\lambda_\del = \lambda_\del + dh
$$
for a \emph{compactly supported} function $h$ on $\del M$.

We need to show that the \emph{structure} of Liouville $\sigma$-sectors with respect to
$$
(M,\psi^*\lambda) = (M,\lambda + df)
$$
is isomorphic to that of $(M,\lambda)$.
For this, we make a choice of $\sigma = \{\sigma_1, \cdots, \sigma_m\}$
associated to a transverse coisotropic collection $\{H_1, \dots, H_m\}$
 for each sectorial corner $\delta$ of
$M$ with
$$
C_\delta = H_1 \cap \cdots \cap H_m.
$$
Such a collection exists by definition for $(M,\lambda)$ being a Liouville $\sigma$-sector.

Now we consider the pushforward collection of hypersurfaces
$$
\{H_1', \cdots, H_m'\} := \{\psi(H_1), \ldots, \psi(H_m)\}.
$$
Since smooth diffeomorphisms between two manifolds with corners
preserve strata dimensions by definition, we work with
the defining data of $(M, \psi^*\lambda)$ stratawise of the fixed dimensional strata.

We first need to show that each $H_i'$ is $\sigma$-sectorial hypersurface by finding
a collection
$$
\sigma' = \{\sigma_1',\ldots, \sigma_m\}
$$
where each $\sigma_i'$ is a smooth section o $H_i'$ respectively.
For this purpose, we prove the following

\begin{lemma} Choose the sections $\sigma_i$s so that
$$
\Image \sigma_i \subset M \setminus \supp df.
$$
Then there exists a neighborhood $\nbhd(\del_\infty M)$ such that the following hold:
\begin{enumerate}
\item The map $\psi: \nbhd(\del_\infty M) \cap H_i \to H_i$ descends to a
diffeomorphism $[\psi]: \cN_{H_i} \to \cN_{H_i}$.
\item The map $\sigma_i^\psi: \cN_{H_i} \to \psi(H_i)$ defined by
$$
\sigma_i^\psi: = \psi \circ \sigma_i \circ [\psi]^{-1}
$$
is a section of the projection $\psi(H_i) \to \cN_{\psi(H_i)} = \cN_{H_i}$.
\end{enumerate}
\end{lemma}
\begin{proof} Since $\Image \sigma_i \subset M \setminus \supp df$, we have
$$
\psi^*\lambda = \lambda
$$
on $\Image \sigma_i:=F_i$. In particular, the projection $\pi_{H_i}:H_i \to \cN_{H_i}$
restricts to a bijective map on $F_i$. Furthermore since $\psi^*\lambda = \lambda$
on $\nbhd(\del_\infty M)$, the associated Liouville vector field $Z_\lambda$ of $\lambda$ satisfies
$$
\psi_*Z_\lambda = Z_\lambda
$$
thereon. Recall that $\psi$ restricts to a diffeomorphism on $\del M$ (as a map on manifold with corners).
Then the equality $\psi^*\lambda = \lambda$ implies
$\psi_\del^*d\lambda_\del = d\lambda_\del$ and hence
$$
d\psi_\del (\ker d\lambda_\del)= \ker d\lambda_\del
$$
on $\nbhd(\del M) \cap H_i$. Therefore $\psi$ descends to a diffeomorphism $[\psi]: \cN_{H_i} \to \cN_{H_i}$
so that we have the commutative diagram
$$
\xymatrix{H_i \ar[d]^{\pi_{H_i}} \ar[r]^{\psi} & \psi(H_i) \ar[d]^{\pi_{\psi(H_i)}}\\
\cN_{H_i} \ar[r]^{[\psi]} & \cN_{H_i}.
}
$$
By composing $\sigma_i' = \psi \circ \sigma_i$ with
$ \pi_{\psi(H_i)}$ to the left, we obtain
$$
\pi_{\psi(H_i)} \sigma_i' = \pi_{\psi(H_i)} \circ \psi \circ \sigma_i
= [\psi]\circ \pi_{H_i} \circ \sigma_i = [\psi]
$$
which is a diffeomorphism.
Therefore the map
$$
\sigma_i^\psi : = \psi \circ \sigma_i' = \psi \circ \sigma_i \circ [\psi]^{-1}
$$
is a section of the projection $H_I' \to \cN_{H_i'}$. This finishes the proof.
\end{proof}

Clearly any diffeomorphism preserves the  transverse intersection property.
This proves that any diffeomorphism $\psi$ satisfying $\psi^*\lambda = \lambda + df$
with compactly supported $f$ is an automorphism of the \emph{structure of Liouville $\sigma$-sectors}.
(See Definition \ref{defn:sectorial-collection} and \ref{defn:structure}.) This finishes the proof of the theorem.
\end{proof}

Based on this discussion, we will unambiguously denote by $\aut(M)$ the automorphism group of
Liouville $\sigma$-sector $(M,\lambda)$ as in the case of Liouville manifolds.

\begin{remark} \begin{enumerate}
\item The above proof shows that the group $\aut(M,\lambda)$ is manifestly
the automorphism group of the structure of Liouville $\sigma$-sectors. We alert the readers that
this is not manifest in the original definition of Liouville sectors \emph{with corners}
from \cite{gps}, \cite{gps2}.
\item This simple characterization of the automorphism groups of Liouville $\sigma$-sectors with corners
enables one to define the bundle of Liouville sectors with corners in the same way for the
case of Liouville manifolds (with boundary) \emph{without corners}. See \cite{oh-tanaka:liouville-bundles}
for the usage of such bundles in the construction of continuous actions of Lie groups
on the wrapped Fukaya category of Liouville sectors (with corners).
\item Recall that the Liouville structure $\lambda$ on $M$ induces a natural contact structure
on its ideal boundary $\del_\infty M$.
We denote the associated contact structure by $\xi_\infty$. Then
we have another natural map
$$
\Aut(M,\lambda) \to \Cont(\del_\infty M,\xi_\infty)
$$
where $(\del_\infty M,\xi_\infty)$ is the group of \emph{contactomorphisms of the
contact manifold} $ (\del_\infty M,\xi_\infty)$.
(See \cite{giroux}, \cite{oh-tanaka:smooth-approximation} for the details.)
\end{enumerate}
\end{remark}

\appendix

\section{Giroux's ideal completion}\label{sec:giroux}

For the main purpose of the present paper,
we need to recall a more detailed description of Giroux's construction given in \cite{giroux}.

Under the definition of idea Liouville form $\beta$ in Definition \ref{defn:giroux}, the vector field
$Z_\beta$ uniquely determined by the equation
$$
Z_\beta \rfloor \omega = \lambda
$$
is called the \emph{ideal Liouville vector field}.

The following result is proved by Giroux \cite{giroux}.

\begin{prop}[Ideal Liouville forms; Corollary 4 \cite{giroux}] On any Liouville domain $(F, \omega)$,
ideal Liouville forms constitute an affine space. Given a function $u: R \to \R_{\geq 0}$ with 
regular level set $\del_\infty F = {u=0}$, the underlying vector space can be described as consisting of all 
\emph{closed} one-forms $\kappa$ on $\Int F$ satisfying the following equivalent conditions:
\begin{enumerate}
\item The form $u \kappa$ extends to a smooth form on $F$.
\item The vector field ${\underset{\rightarrow}{\kappa}}/u$ ext4ends to a smooth vector field on $F$
(which is automatically tangent to $K : = \del_\infty F$).
\item There exists a function $f: F \to \R$ such that $\kappa - d(f \log u)$ is the 
restriction of a closed one-form on $F$.
\end{enumerate}
\end{prop}

The following corollary is derived in \cite{giroux} which is credited to 
\cite[Lemma 1.1 \& the subsequent remark]{bourgeois-ekholm-eliashberg}.

\begin{cor}[Corollary 5 \cite{giroux}] Let $(F,\omega)$ be an ideal Liouville domain and
$\lambda_t$ ($t \in [0,1]$) a path of ideal Liouville forms in $\Int F$. Then there is a symplectic
isotopy $\psi_t$  ($t \in [0,1]$) of $F$, relative to the boundary, such that $\psi_0 = \id$ and,
for every $t \in [0,1]$, the form $\psi_t^*\lambda_t - \lambda_r = dh_t$ for some function $h$ with 
compact support in $\Int F$.
\end{cor}

Here is the precise definition of the notion of \emph{ideal completion} of the Liouville domain
$(F,\lambda)$.

\begin{defn}[Example 9 \cite{giroux}] Let $(F,\lambda)$ be a Liouville domain, and let 
$u: F \to \R_{\geq 0}$
be a function with the following properties:
\begin{itemize}
\item $u$ admits $K: = \del_\infty F$ as its regular level set $\{u=0\}$,
\item $Z[\log u] < 1$ at every point in $\Int F$.
\end{itemize} 
Define 
$$
\omega: = d\left(\lambda/u\right)
$$
to be a symplectic form on $\Int F$ on the ideal Liouville domain $(F,\omega)$
which we call the \emph{ideal completion} of the Liouville domain $(F,\lambda)$.
\end{defn}

\section{Proof of Lemma \ref{lem:Hausdorff}}
\label{sec:Hausdorff}

 The subspace topology of 
$F_{\text{\rm ref}} = \Image \sigma_{\text{\rm ref}}$ is Hausdorff since $H$ is Hausdorff.
(See \cite[Theorem 1.3 in p. 138]{dugundji}, for example.)

Furthermore we can show that $\Image \sigma_{\text{\rm ref}}$ is a closed subset
of $H$ as follows.  Let $x \not \in F_{\sigma_{\text{\rm ref}}}$ and set $\ell_x: = \pi(x)$. Then we have
$$
x \neq\ \sigma_{\text{\rm ref}}(\pi(x))=:x', \quad \pi(x) = \pi(x').
$$
Since $\pi^{-1}(\ell_x) \subset H$ as a subspace of $H$ is Hausdorff, we can find two 
relatively compact open subsets  $U_1, \, U_2$ of $H$, which is locally compact Hausdorff,
such that 
$\overline U_1 \cap \overline U_2 = \emptyset$ and 
$x \in U_1$ and $x' \in U_2$. 

For each point $y' \in \overline U_2$,   because 
 $ \overline U_2 \subset  H \setminus U_1$, we can find open neighborhoods
 $U_{y'}^1$ of $x$, and  $U_{y'}^2$ of $y'$ respectively  such that
 $$
U_{y'}^1 \cap U_{y'}^2 = \emptyset, \quad U_{y'}^1 \subset U_1.
$$
In particular, we have
$$
\emptyset = U_{y'}^1 \cap U_{y'}^2  \supset (U_{y'}^2 \cap F_{\text{\rm ref}}) .
$$
By construction $\{U^2_{y'}\}$ is an open cover of $\overline U_2$,
compactness of $\overline U_2$ implies that there is a finite subcover 
$\{U_{y_1'}^2, \ldots, U_{y_k'}^2\}$ of
$\overline U_2$ out of $\{U_{y'}^2\}_{y' \in \overline U}$ such that
$$
\emptyset = \left(\bigcap_{i=1}^k U_{y_i'}^1\right) \cap \left(\bigcup_{i=1}^k  (U_{y'_i} \cap F_{\text{\rm ref}}) \right)
\supset \left(\bigcap_{i=1}^k U_{y_i'}^1\right) \cap (\overline U_2 \cap F_{\text{\rm ref}}).
$$
In particular the open neighborhood $U_1' : = \cap_{i=1}^k U_{y_i'}^1$ of $x$
does not intersect $\overline U_2  \cap F_{\text{\rm ref}}$. 

It remains to show that $U_1'$ does not intersect $F_{\text{\rm ref}} \setminus \overline U_2$
either. Suppose to the contrary that there exists a point 
$w \in F_{\text{\rm ref}} \setminus \overline U_2$
such that $w \in U_1'$. In particular, we have  $\pi(w) \in \pi(U_1')$.
Recall $\pi(x) \in \pi(U_1') \cap \pi(U_2) \subset \pi(U_2)$ since $\pi(x) = \pi(x')$ with $x' \in U_2$.
This implies that we have
$$
\sigma_{\text{\rm ref}}(\pi(w)) = \sigma_{\text{\rm ref}}(\pi(w'))
$$
for some $w' \in U_2$. Since $\sigma_{\text{\rm ref}}$ is one-to-one, this proves
$$
\pi(w) = \pi(w'), \quad w \in \overline U_2  \cap F_{\text{\rm ref}}, \, 
w' \in (H \setminus \overline U_2) \cap F_{\text{\rm ref}}
$$
Since $\pi$ is also one-to-one on $F_{\text{\rm ref}}$, we obtain $w = w'$,  a contradiction.  
Therefore this proves closedness of $F_{\text{\rm ref}}$.
Once this is proved, it follows that $\cN_H$ is Hausdorff by the classical 
fact. (See \cite[Theorem in p. 138]{dugundji}, for example.)

\section{Stability theorem of Liouville sectors} \label{sec:stability-sectors}

In this section, we extend a stability theorem of Liouville manifold proved in 
\cite[Theorem 9.2]{oh:sectorial} to the case of Liouville sectors of our current context.
A complete proof of the theorem for the case of Liouville manifolds, i.e., for 
the case with $\del M = \emptyset$ is given in \cite{oh:sectorial}. Therefore we have only
to ensure the boundary behavior laid out in Statements (2) and (3) below
\emph{under the additional hypothesis that 
$Z$ is tangent to $\del M$ near infinity.}

For this purpose, from the aforementioned hypothesis and the $C^3$-tameness of $(M,\omega)$,
we have already shown that there exists $\delta > 0, \, C > 0$ and a 
neighborhood $\nbhd(\del_\infty M) =\{s \geq N\} \times \del_\infty M$ such that 
for $F = \del_\infty M \cap \del M$, we have
\bea\label{eq:Iinfty-sinfty}
V_{\delta, C}: = F \times [-\delta, 0] \times \{|I| \geq C\} & \subset& 
 (\Phi^{-1})\left(\{s \geq N\} \times \del_\infty M\right), \\
F \times\{0\} \times \{|I| \geq C\} & \subset&  (\Phi^{-1})\left(\{s \geq N\} \times \del M \cap \del_\infty M\right),
\eea

\begin{theorem}[Compare with Theorem 9.2 \cite{oh:sectorial}] \label{thm:stability-sectors}
Let $(M,\lambda)$ be a Liouville sector,  and $\lambda_t$ be a
family of Liouville forms such that  for all $t \in [0,1]$ 
\begin{enumerate}
\item  $d\lambda_t = d \lambda$ and $Z_t$ is tangent to $\nbhd(\del_\infty M)$,
\item they satisfy
 $$
 \lambda_t - \lambda = dk_t
 $$
 for some smooth functions $k_t$   satisfying the bound
 $$
\left\|\frac{\del k_t}{\del t}\right\|_{C^1} < C.
$$
\end{enumerate}
Then there exists a diffeomorphism $\phi_t$ such that for all $t \in [0,1]$ 
\begin{enumerate} 
\item $\phi_t^*(\lambda_t) = \lambda_0$,
\item 
$\supp \phi \subset \supp \left(\frac{\del k_t}{\del t}\right)$, and
\item there exists a constant $\delta' > 0$, $C_1> C_2 > 0$ such that
$$
\phi_t (\overline V_{\delta',C'}) \subset V_{\delta,C},
$$
and
$$
\phi_t(F \times \{0\} \times \{|I| \geq C_1\}) \subset F \times \{0\} \times \{|I| \geq C_2\}
$$
where $V_{\delta,C}: = F \times (-\delta, 0] \times \{|I| > C_2\}$. 
\end{enumerate}
\end{theorem}

\begin{proof} In this proof, we mostly duplicate the proof of \cite[Theorem 9.2]{oh:sectorial} with
some adaptation to ensure the properties (2) and (3) required above.

We consider one-parameter family of contactifications on $Q = M \times \R$
with contact forms given by
$$
\alpha_\kappa = dt - \pi^*(\lambda + dk_\kappa)
$$
which are contact by the hypothesis $d\lambda_t = d\lambda$.
They define a family of contact structures on $Q$ given by
$$
\xi_\kappa: = \ker \alpha_\kappa \quad \text{\rm for } \, \kappa \in [0,1].
$$
We write $\lambda_\kappa = \lambda + dk_\kappa$.

Considering the `space-time' $Q=M \times \R$,  we denote the coordinate of the $\R$-factor
by $t$.
We note that the Reeb vector fields $R_{\alpha_\kappa}$ of each $\alpha_\kappa$ is given by
$$
R_{\alpha_\kappa} = \frac{\del}{\del t} \quad \text{\rm for all }\, \kappa \in [0,1].
$$
We lift the $s$-dependent function $k_\kappa$ to the product $M \times \R$
$$
\widetilde k_\kappa(x,t) := \pi^*k_\kappa(x)
$$
which we emphasize \emph{does not} depend on $t$-coordinate of
 the `space-time' $Q = M \times \R$.  This will be important when we
 go back to the study of the family $\lambda_t$ of Liouville one-forms
from our application of  Gray's stability theorem in the contactification.

 As in the general proof of the stability theorem, we will try to find a one-parameter
 family of contactomorphisms $\psi_\kappa^*\alpha_\kappa = e^{g_\kappa}\alpha_0$.
 In the current context of our interest, we will try to find
 \emph{strict} contactomorphisms for which $g_\kappa \equiv 0$.
Then we can choose a family of $s$-dependent vector fields
$$
\widetilde X_\kappa \in \xi_\kappa
$$
that we highlight satisfies
\be\label{eq:moser1}
d(\widetilde X_\kappa \rfloor \alpha_\kappa) + \widetilde X_\kappa \rfloor d\alpha_\kappa + \frac{\del \alpha_\kappa}{\del \kappa} = h_\kappa \,  \alpha_\kappa
\ee
for $h_\kappa = \frac{\del g_\kappa}{\del \kappa} \circ \psi_\kappa^{-1} \equiv 0$.
Then it follows that
\be\label{eq:hspsis-1}
\frac{\del\widetilde k_{s}}{\del t} = 0
\ee
and
$$
\alpha_\kappa - \alpha_0 = \lambda - \lambda_\kappa = - d\pi^*k_\kappa.
$$
Therefore we have
$$
\frac{\del \alpha_\kappa}{\del \kappa} = - \pi^* d\dot k_\kappa, \quad \dot k_\kappa: = \frac{\del k_\kappa}{\del \kappa}.
$$
Therefore \eqref{eq:moser1} with $h_\kappa \equiv 0$ is equivalent to
\be\label{eq:moser-equation}
d(\widetilde X \rfloor \alpha_\kappa) + \widetilde X \rfloor d\alpha_\kappa
- \pi^* d\dot k_\kappa = 0.
\ee

 Then the vector field $\widetilde X_\kappa$ is uniquely  determined by the equation
\be\label{eq:tildeXs}
\widetilde X_\kappa \in \xi_\kappa, \quad \widetilde X_\kappa \rfloor d\alpha_\kappa = - \pi^* d_M
\left(\frac{\del \widetilde k_\kappa}{\del \kappa}\right).
\ee
By the hypothesis $\left\|\frac{\del k_t}{\del t}\right\|_{C^1} < C$, the vector field
$\widetilde X_\kappa$ is globally Lipschitz.
This implies that the flow of $\widetilde X_\kappa$ exists on $Q = M \times \R$ and satisfies
$$
\widetilde \psi_\kappa^*\alpha_\kappa = \alpha
$$
for all $s \in [0,1]$.

Now we write $\widetilde \psi_\kappa(x,t) = (\psi_\kappa(x,t),b_\kappa(x,t))$ and its generating vector field
\be\label{eq:Xs}
\widetilde X_\kappa(x,t) = X_\kappa(x,t) \oplus a_\kappa(x,t)\frac{\del}{\del t}.
\ee
The condition $\widetilde X_\kappa \in \xi_\kappa$ also implies $0 = \alpha_\kappa(\widetilde X_\kappa) = a_\kappa - \lambda_\kappa(X_\kappa)$, i.e.,
\be \label{eq:das}
a_\kappa  = \lambda_\kappa(X_\kappa).
\ee
Then Moser's deformation equation \eqref{eq:moser-equation} is equivalent to
\be
X_\kappa \rfloor (-d\lambda) = d_M \dot k_\kappa
\ee
where we utilize the identity $d\lambda_{s} = d\lambda$ for all $s$.
Hence we obtain
\be\label{eq:suppXs}
\supp X_\kappa \subset \supp d_M \dot k_\kappa
\ee
and $\supp \psi_\kappa \subset \supp d_M \dot k_\kappa$.

Now we rewrite the projection to $M$ of the equation $\widetilde \psi_\kappa^*\alpha_\kappa = \alpha$
into
$$
\widetilde \psi_\kappa^*(dt - \pi^* \lambda_\kappa) = dt - \pi^* \lambda
$$
which is equivalent to
$$
db_\kappa - \psi_\kappa^*\lambda_\kappa = dt - \pi^*\lambda.
$$
From this, we derive
$$
\psi_\kappa^*\lambda_\kappa = \lambda + d_M b_\kappa, \quad \frac{\del b_\kappa}{\del t} \equiv 1.
$$
Noting the initial condition $(\psi_0(x,t), b_0(x,t)) = (x,t)$,
we in particular proved $b_\kappa(x,t) = t$ for all $x$.
Then by setting $s = 1$, we define
$$
\phi_t(x): = \psi_1(x,t),
$$
which then satisfies
$$
\phi_t^*\lambda_t = \lambda.
$$
Obviously each $\phi_t: M \to M$ is invertible for each $t \in [0,1]$
since the diffeomorphism $(x,t) \mapsto \psi_1(x,t)$ maps each $t$-slice to itself.

Finally we consider the case $k_t = h_V$ of our interest so that $\lambda = \lambda_V + dh_V$
and $\lambda_\kappa = \lambda_V + \kappa\, dh_V$ for which we have $\|h\|_{C^2} < C < \infty$ and
$$
\frac{\del h_V}{\del I}\Big|_{\del M} = 0
$$
on $V = F \times \{|I| \geq C\} \cap \{s \geq N\} \supset \del_\infty \cap \del M$.
We first check the completeness of the flow $\phi_t$. For this purpose, we
go back to the defining equation \eqref{eq:tildeXs} of the vector field $\widetilde X_\kappa$
$$
\widetilde X_\kappa \rfloor d\alpha_\kappa = - \pi^*d_M\widetilde k_\kappa
$$
with $\widetilde k_\kappa(t,x) = \kappa\, h_V(x)$ so that we have
$$
d_M\widetilde k_\kappa = \kappa\,  h_V(x).
$$
The $C^2$-boundedness of $h_V$
in particular implies that the vector fields $X_t$ are uniformly Lipschitz.
Therefore the flow exists for all time until it hits the boundary $\del M$
of our interest.  Furthermore, $X_t$ is also tangent to $\del M$ on
$V$ since $\frac{\del h_V}{\del I} = 0$ on $\del M$.

This now completes the proof of
Theorem  \ref{thm:stability-sectors}, and hence follows Lemma \ref{lem:stability-sectors}.
\end{proof}

\def\cprime{$'$}
\providecommand{\bysame}{\leavevmode\hbox to3em{\hrulefill}\thinspace}
\providecommand{\MR}{\relax\ifhmode\unskip\space\fi MR }
\providecommand{\MRhref}[2]{%
  \href{http://www.ams.org/mathscinet-getitem?mr=#1}{#2}
}
\providecommand{\href}[2]{#2}



\begin{thebibliography}{GPS24b}

\bibitem[AM78]{abraham-marsden}
R.~Abraham and J.~Marsden, \emph{Foundations of mechanics}, Benjamin/Cummings,
  Reading, MA, 1978, 2nd Edition.

\bibitem[Arn88]{arnold:mechanics}
V.~I. Arnol{\cprime}d, \emph{Mathematical {M}ethods of {C}lassical
  {M}echanics}, Graduate Texts in Mathematics, vol.~60, Springer-Verlag, New
  York, NY, 1988, 2nd Edition.

\bibitem[Asp23]{asplund}
Johan Asplund, \emph{Simplicial descent for chekanov-eliashberg dg-algebras},
  J. Topol. \textbf{16} (2023), no.~2, 489--541.

\bibitem[BEE12]{bourgeois-ekholm-eliashberg}
Fr\'ed\'eric Bourgeois, Tobias Ekholm, and Yasha Eliashberg, \emph{Effect of
  {L}egendrian surgery}, Geom. Topol. \textbf{16} (2012), no.~1, 301--389.

\bibitem[CC00]{candel-conlon}
Alberto Candel and Lawrence Conlon, \emph{Foliations {I}}, Graduate Studies in
  Mathematics, 23., American Mathematical Socienty, 2000.

\bibitem[Che66]{chern}
Shing-Shen Chern, \emph{The geometry of {G}-structures}, Bulletin of the
  American Mathematical Society \textbf{72} (1966), no.~2, 167--219.

\bibitem[CO25]{choi-oh:quasiisometry}
Jaeyoung Choi and Y.-G. Oh, \emph{Injectivity radius lower bound of convex sum
  of tame {R}iemannian metrics and applications to symplectic topology},
  Adv. in Math. (2025), in press,
  https://doi.org/10.1016/j.aim.2025.110443.

\bibitem[dLLV19]{dMV}
Manuel de~Le\'on and Manuel Lainz~Valc\'azar, \emph{Contact {H}amiltonian
  systems}, J. Math. Phys. \textbf{60} (2019), no.~10, 102902, 18 pp.

\bibitem[Dug65]{dugundji}
James Dugundji, \emph{Topology}, Allyn and Bacon Series in Adv. Math., Allyn
  and Bacon, Inc., Boston, 1965.

\bibitem[Ful93]{fulton:toric}
William Fulton, \emph{Introduction to toric varieties}, Annals of Mathematics
  Studies, vol. 131, Princeton University Press, Princeton, NJ, 1993, The
  William H. Roever Lectures in Geometry.

\bibitem[Gir17]{giroux}
E.~Giroux, \emph{Ideal {L}iouville domains - a cool gadget}, arXiv:1708.08855,
  2017.

\bibitem[Got82]{gotay}
Mark~J Gotay, \emph{On coisotropic imbeddings of presymplectic manifolds},
  Proc. Amer. Math. Soc. \textbf{84} (1982), no.~1, 111--114.

\bibitem[GPS20]{gps}
Sheel Ganatra, John Pardon, and Vivek Shende, \emph{Covariantly functorial
  wrapped {F}loer theory on {L}iouville sectors}, Publ. Math. Inst. Hautes
  {\'E}tudes Sci. \textbf{131} (2020), 73–200.

\bibitem[GPS24a]{gps-microlocal}
\bysame, \emph{Microlocal {M}orse theory of wrapped {F}ukaya categories}, Ann.
  of Math. (2) \textbf{199} (2024), no.~3, 943--1042.

\bibitem[GPS24b]{gps2}
\bysame, \emph{Sectorial descent for wrapped {F}ukaya categories}, J. Amer.
  Math. Soc. \textbf{37} (2024), no.~2, 499--635.

\bibitem[GS68]{gelfand-shilov}
I.M. Gelfand and G.E. Shilov, \emph{Generalized {F}unctions, vol 2}, Academic
  Press, New York, 1968.

\bibitem[Oh21]{oh:contacton-Legendrian-bdy}
Y.-G. Oh, \emph{Contact {H}amiltonian dynamics and perturbed contact instantons
  with {L}egendrian boundary condition}, preprint, arXiv:2103.15390(v2), 2021.

\bibitem[Oh24a]{oh:sectorial}
\bysame, \emph{Geometry of {L}iouville sectors and the maximum principle},
  Asian J. Math. \textbf{28} (2024), no.~5, 653--734.

\bibitem[Oh24b]{oh:gradient-sectorial}
\bysame, \emph{Monoid of gradient-sectorial {L}agrangians in {L}iouville
  sectors with corners}, Asian J. of Math. \textbf{28} (2024), no.~5, 735--756.

\bibitem[OP05]{oh-park}
Y.-G. Oh and Jae-Suk Park, \emph{Deformations of coisotropic submanifolds and
  strong homotopy {L}ie algebroids}, Invent. Math. \textbf{161} (2005), no.~2,
  287--360. \MR{2180451 (2006g:53152)}

\bibitem[OT]{oh-tanaka:localizations}
Y.-G. Oh and Hiro~Lee Tanaka, \emph{{$A_\infty$}-categories, their
  $\infty$-category, and their localizations}, Homology, Homotopy and
  Applications, (to appear), arXiv:2003.05806.

\bibitem[OT19]{oh-tanaka:actions}
\bysame, \emph{Continuous and coherent actions on wrapped {F}ukaya categories},
  arXiv:1911.00349, 2019.

\bibitem[OT20]{oh-tanaka:liouville-bundles}
\bysame, \emph{Holomorphic curves and continuation maps in {L}iouville
  bundles}, arXiv:2003.04977, 2020.

\bibitem[OT22]{oh-tanaka:smooth-approximation}
\bysame, \emph{Smooth approximation for classifying spaces of diffeomorphism
  groups}, Algebr. Geom. Topol. \textbf{22} (2022), no.~3, 1177--1216.

\bibitem[Sik94]{sikorav:tame}
J.~C. Sikorav, \emph{Some properties of holomorphic curves in almost complex
  manifolds}, Chapter V of Holomorphic Curves in Symplectic Geometry, ed.,
  Audin, M. and Lafontaine, J., Birkh\"auser, Basel.

\bibitem[SS95]{steen-seebach}
Lynn~Arthur Steen and J.~Arthur Seebach, \emph{Counterexamples in {T}opology},
  Springer-Verlag, Berlin, New York, 1995, Dover reprint of 1978 ed.

\bibitem[Ste83]{sternberg}
Shlomo Sternberg, \emph{Lectures on {D}ifferential {G}eometry}, 2nd ed. ed.,
  Chelsea Publishing Co, New York, 1983.

\bibitem[Wei79]{weinstein-cbms}
A.~Weinstein, \emph{Lectures on symplectic manifolds}, CBMS Regional Conference
  Series in Mathematics, vol.~29, American Mathematical Society, Providence,
  R.I., 1979, Corrected reprint, ii + 48 pp.

\end{thebibliography}

\end{document}